\documentclass[12pt]{amsart}
\usepackage{amscd,amsmath,amssymb,amsfonts,verbatim}
\usepackage{graphicx}
\usepackage[active]{srcltx}

\usepackage{amscd,amsmath,amssymb,amsfonts} 
\usepackage[cmtip, all]{xy}

\newtheorem{thm}{Theorem}[section]
\newtheorem{prop}[thm]{Proposition}
\newtheorem{lem}[thm]{Lemma}
\newtheorem{cor}[thm]{Corollary}

\newtheorem*{claim}{Claim}
 
\theoremstyle{definition}
\newtheorem{exm}[thm]{Example}
\newtheorem{defn}[thm]{Definition}

\theoremstyle{remark}
\newtheorem{remk}[thm]{Remark}
\newtheorem{remks}[thm]{Remarks}
\newtheorem{exms}[thm]{Examples}
\newtheorem{notat}[thm]{Notation}

\numberwithin{equation}{section}

{\hfill$\square$\end{defn}}

{\hfill$\square$\end{remk}}

{\hfill$\square$\end{remks}}

{\hfill$\square$\end{exm}}

{\hfill$\square$\end{exms}}

{\hfill$\square$\end{notat}}

\newcommand{\CH}{{\rm CH}}
\newcommand{\tm}{{\rm th}}

\newcommand{\surj}{\twoheadrightarrow}
\newcommand{\inj}{\hookrightarrow}

\newcommand{\codim}{{\rm codim}}
\newcommand{\rank}{{\rm rank}}

\newcommand{\Hom}{{\rm Hom}}

\newcommand{\Spec}{{\rm Spec \,}}

\newcommand{\0}{\emptyset}
\newcommand{\sHom}{{\mathcal{H}{om}}}

\newcommand{\Sch}{{\operatorname{\mathbf{Sch}}}}

\newcommand{\Sm}{{\mathbf{Sm}}}

\newcommand{\ds}{{/\kern-3pt/}}

\newcommand{\ov}{\overline}
\newcommand{\wt}{\widetilde}
\newcommand{\wh}{\widehat}
\newcommand{\mg}{MGL^{*,*}_G}
\newcommand{\mbb}{MGL^{*,*}_B}
\newcommand{\mh}{MGL^{*,*}_H}
\newcommand{\mt}{MGL^{*,*}_T}
\newcommand{\m}{MGL^{*,*}}

\newcommand{\tuborg}{\left\{\begin{array}{ll}}
\newcommand{\sluttuborg}{\end{array}\right.}

\newcommand{\sB}{{\mathcal B}}

\newcommand{\sE}{{\mathcal E}}
\newcommand{\sF}{{\mathcal F}}

\newcommand{\sH}{{\mathcal H}}

\newcommand{\sK}{{\mathcal K}}
\newcommand{\sL}{{\mathcal L}}

\newcommand{\sO}{{\mathcal O}}

\newcommand{\sS}{{\mathcal S}}

\newcommand{\sU}{{\mathcal U}}
\newcommand{\sV}{{\mathcal V}}

\newcommand{\sX}{{\mathcal X}}

\newcommand{\A}{{\mathbb A}}

\newcommand{\C}{{\mathbb C}}

\newcommand{\G}{{\mathbb G}}

\newcommand{\bL}{{\mathbb L}}

\renewcommand{\P}{{\mathbb P}}
\newcommand{\Q}{{\mathbb Q}}

\newcommand{\Z}{{\mathbb Z}}

\begin{document}
\title{The motivic cobordism for group actions}
\author{Amalendu Krishna}
\address{School of Mathematics, Tata Institute of Fundamental Research,  
1 Homi Bhabha Road, Colaba, Mumbai, 400 005, India}
\email{amal@math.tifr.res.in}

\baselineskip=13.85pt


\begin{abstract}
Let $k$ be a field of characteristic zero. For a linear algebraic group $G$ 
over $k$ acting on $k$-schemes, we define the equivariant version of
the Voevodsky's motivic cobordism $MGL$ and show that it is an 
oriented equivariant cohomology theory on the category of smooth $G$-schemes
which satisfies the localization sequence. We give several applications.
In particular, we study the
motivic cobordism rings for the classifying spaces and the cycle class maps
to the singular cohomology of such spaces.  
\end{abstract}


\maketitle


\section{Introduction}\label{section:Intro}
Let $k$ be a field of characteristic zero. Let $G$ be a linear algebraic
group over $k$. In this paper, we use the 
techniques of $\A^1$-homotopy theory to construct an equivariant version
of the motivic cobordism theory discovered for smooth schemes by 
Voevodsky \cite{Voev1}.
We use the notion of Thom and Chern structures on the cohomology theories
of motivic spaces to show that the new equivariant motivic cobordism
is an oriented equivariant cohomology theory on the category of
smooth $G$-schemes. One of the main results about this equivariant
motivic cobordism theory is that it satisfies the expected localization
sequence. 

In order to relate the equivariant motivic cobordism with the equivariant
analogue of the geometric cobordism of Levine and Morel \cite{LM},
studied earlier in \cite{Krishna1}, we look at the equivariant motivic
cobordism from a different approach. This approach allows us to 
show that the equivariant cobordism theory of \cite{Krishna1} is the
degree zero part of the equivariant motivic cobordism studied in this
paper, if we work with rational coeffcients. This is an equivariant
analogue of a result of Levine \cite{Levine1}. This also allows us to prove 
many interesting properties of the equivariant motivic cobordism with
rational coefficients. We show that the two approaches give the same
answer for the torus action on smooth projective schemes.

We prove the self-intersection formula for the equivariant motivic cobordism.
This formula allows us to deduce the localization theorems for the
equivariant motivic cobordism for torus action. This is an analogue
of the similar localization theorem for the equivariant $K$-theory
and generalizes a similar result for the geometric equivariant cobordism in 
\cite{Krishna2}.

We prove a decomposition theorem for the equivariant motivic cobordism
of smooth projective schemes with torus action. This decomposition is used
to give a simple formula for the equivariant and ordinary motivic cobordism
of flag varieties.

The representability of motivic cohomology in the stable $\A^1$-homotopy 
category implies that there is a natural map form the equivariant motivic
cobordism to the equivariant higher Chow groups of smooth schemes with 
group actions which
were defined by Edidin and Graham \cite{EG}.
We study the complex realization map from the equivariant motivic to the
complex cobordism ring of smooth varieties over $\C$. This realization map
is used to study the cycle class map from the equivariant Chow groups
of a smooth projective scheme to its equivariant singular cohomology. This
generalizes a result of Totaro \cite{Totaro1}. 

We show that our equivariant motivic cobordism generalizes to a theory
of motivic cobordism of all quotient stacks and it is a cohomology theory
with localization sequence on the category of smooth quotient stacks. 
More applications of the results
presented here and some computations of the motivic cobordism for group actions
will appear in a separate paper. 

\section{Basic constructions}\label{section:BDefn}
Let $k$ be a field of characteristic zero and let $\Sm_k$ denote the category
of smooth schemes of finite type over $k$. 

A {\sl linear algebraic group} $G$ over $k$ will mean a smooth and affine 
group scheme over $k$. By a closed subgroup $H$ of an algebraic group $G$,
we shall mean a morphism $H \to G$ of algebraic groups over $k$ which is a
closed immersion of $k$-schemes. In particular, a closed subgroup of a linear
algebraic group will be of the same type and hence smooth. Recall from
\cite[Proposition~1.10]{Borel} that a linear algebraic group over $k$ is a
closed subgroup of a general linear group, defined over $k$.    
Let $\Sch^G_k$ (resp. $\Sm^G_k$) denote the category of quasi-projective 
(resp. smooth) $k$-schemes with $G$-action and $G$-equivariant maps.
An object of $\Sch^G_k$ will be often be called a $G$-scheme.

Recall that an action of a linear algebraic group $G$ on a $k$-scheme $X$ 
is said to be {\sl linear} if $X$ admits a 
$G$-equivariant ample line bundle, a condition which is always satisfied
if $X$ is normal ({\sl cf.} \cite[Theorem~2.5]{Sumihiro} for $G$ connected
and \cite[5.7]{Thomason1} for $G$ general). All $G$-actions in this paper
will be assumed to be linear. We shall use the following other
notations throughout this text.
\begin{enumerate}
\item
${\rm Nis}/k :$ \ The Grothendieck site of smooth schemes over $k$ with
Nisnevich topology. 
\item
${\rm Shv}\left({\rm Nis}_k\right) :$ \ 
The category of sheaves of sets on ${\rm Nis}/k$.
\item
$\Delta^{\rm op}{\rm Shv}\left({\rm Nis}_k\right) :$ \ 
The category of sheaves of simplicial sets on ${\rm Nis}/k$. 
\item
$\sH(k) :$ \ The unstable $\A^1$-homotopy category of simplicial sheaves
on ${\rm Nis}/k$, as defined in \cite{MV}.
\item
$\sH_{\bullet}(k) :$ \ The unstable $\A^1$-homotopy category of pointed 
simplicial sheaves on ${\rm Nis}/k$, as defined in \cite{MV}.
\item
$\sS\sH(k) :$ \ The stable $\A^1$-homotopy category of pointed 
simplicial sheaves on ${\rm Nis}/k$ as defined, for example, in \cite{Voev1}. 
\end{enumerate}

Following the notations of \cite{Voev1}, an object of
$\Delta^{\rm op}{\rm Shv}\left({\rm Nis}_k\right)$ will be called a
{\sl motivic space} (or simply a space) and we shall
often write this category of motivic spaces as ${\bf {Spc}}$. The category of 
pointed motivic spaces over $k$ will be denoted by ${\bf {Spc}}_{\bullet}$.
For any $X, Y \in {\bf {Spc}}$, $S(X, Y)$ denotes the simplicial set of
morphisms between spaces as in \cite{MV}.

\subsection{Admissible gadgets}\label{subsection:Agdt}
Let $G$ be a linear algebraic group over $k$. 
All representations of $G$ in this text will be assumed to be
finite-dimensional. We shall say that a pair $(V,U)$ of smooth schemes over $k$
is a {\sl good pair} for $G$ if $V$ is a $k$-rational representation of $G$ and
$U \subseteq V$ is a $G$-invariant open subset on which $G$ acts freely
such that the quotient $U/G$ is a smooth quasi-projective scheme.
It is known ({\sl cf.} \cite[Remark~1.4]{Totaro1}) that a good pair for $G$ 
always exists.

\begin{defn}\label{defn:Add-Gad}
A sequence of pairs $\rho = {\left(V_i, U_i\right)}_{i \ge 1}$ of smooth schemes
over $k$ is called an {\sl admissible gadget} for $G$, if there exists a
good pair $(V,U)$ for $G$ such that $V_i = V^{\oplus i}$ and $U_i \subseteq V_i$
is $G$-invariant open subset such that the following hold for each $i \ge 1$.
\begin{enumerate}
\item
$\left(U_i \oplus V\right) \cup \left(V \oplus U_i\right)
\subseteq U_{i+1}$ as $G$-invariant open subsets.
\item
$\codim_{U_{i+2}}\left(U_{i+2} \setminus 
\left(U_{i+1} \oplus W_{i+1}\right)\right) > 
\codim_{U_{i+1}}\left(U_{i+1} \setminus \left(U_{i} \oplus W_{i}\right)\right)$.
\item
$\codim_{V_{i+1}}\left(V_{i+1} \setminus U_{i+1}\right)
> \codim_{V_i}\left(V_i \setminus U_i\right)$.
\item
The action of $G$ on $U_i$ is free with quotient a quasi-projective scheme. 
\end{enumerate}
\end{defn}

The above definition is a special case of the more general notion
of admissible gadgets in \cite[\S 4.2]{MV}, where these terms are
defined for vector bundles over any given scheme. 
An example of an admissible gadget for $G$ can be constructed as follows. 
Choose a faithful $k$-rational representation $W$ of $G$ dimension $n$. Then 
$G$ acts freely on an open subset $U$ of $V = W^{\oplus n}$. 
Let $Z = V \setminus U$.
We now take 
$V_i = V^{\oplus i}, U_1 = U$ and  $U_{i+1} = \left(U_i \oplus V\right) \cup
\left(V \oplus U_i\right)$ for $i \ge 1$. Setting 
$Z_1 = Z$ and $Z_{i+1} = U_{i+1} \setminus \left(U_i \oplus V\right)$ for 
$i \ge 1$, one checks that $V_i \setminus U_i = Z^i$ and
$Z_{i+1} = Z^i \oplus U$.
In particular, $\codim_{V_i}\left(V_i \setminus U_i\right) =
i \codim_V(Z)$ and 
$\codim_{U_{i+1}}\left(Z_{i+1}\right) = (i+1)d - i\dim(Z)- d =  i \codim_V(Z)$,
where $d = \dim(V)$. Moreover, $U_i \to {U_i}/G$ is a principal $G$-bundle.

The definition of equivariant motivic cobordism needs one to consider certain 
kind of mixed quotient spaces which in general may not be a scheme even if the 
original space is a scheme. The following well known 
({\sl cf.} \cite[Proposition~23]{EG}) lemma shows that this
problem does not occur in our context and all the mixed quotient spaces in this
paper are schemes with ample line bundles.
\begin{lem}\label{lem:sch}
Let $H$ be a linear algebraic group acting freely and linearly on a 
$k$-scheme $U$ such that the quotient $U/H$ exists as a quasi-projective
variety. Let $X$ be a $k$-scheme with a linear action of $H$.
Then the mixed quotient $X \stackrel{H} {\times} U$ for the 
diagonal action on $X \times U$ exists as a scheme and is quasi-projective.
Moreover, this quotient is smooth if both $U$ and $X$ are so.
In particular, if $H$ is a closed subgroup of a linear algebraic group $G$ 
and $X$ is a $k$-scheme with a linear action of $H$, then the quotient 
$G \stackrel{H} {\times} X$ is a quasi-projective scheme.
\end{lem}
\begin{proof} It is already shown in \cite[Proposition~23]{EG} using
\cite[Proposition~7.1]{GIT} that the quotient $X \stackrel{H} {\times} U$ 
is a scheme. Moreover, as $U/H$ is quasi-projective, 
\cite[Proposition~7.1]{GIT} in fact shows that $X \stackrel{H} {\times} U$ 
is also quasi-projective. The similar conclusion about 
$G \stackrel{H} {\times} X$ follows from the first case by taking $U = G$
and by observing that $G/H$ is a smooth quasi-projective scheme
({\sl cf.} \cite[Theorem~6.8]{Borel}). The assertion about the smoothness
is clear since $X \times U \to X \stackrel{H} {\times} U$ is an 
$H$-torsor. 
\end{proof} 

In this text, $\Sm^G_{{free}/k}$ will denote the full subcategory of $\Sm^G_k$
whose objects are those schemes $X$ on which $G$ acts freely such that
the quotient $X/G$ exists and is quasi-projective over $k$. The previous
result shows that if $U \in \Sm^G_{{free}/k}$, then $X \times U$
is also in $\Sm^G_{{free}/k}$ for every $G$-scheme $X$.

\subsection{The Borel spaces}\label{subsection:Borel}
Let $X \in \Sm^G_k$. For an admissible gadget $\rho$, let 
$X^i_G(\rho)$ denote the mixed quotient space 
$X \stackrel{G} {\times} U_i$.
If the admissible gadget $\rho$ is clear from the given context,
we shall write $X^i_G(\rho)$ simply as $X^i_G$.

We define the {\sl motivic Borel space} $X_G(\rho)$ to be the colimit 
${colim}_i \ X^i_G(\rho)$, where
colimit is taken with respect the inclusions $U_i \subset
U_i \oplus V \subset U_{i+1}$ in the category of motivic spaces. 
We can think of $X_G(\rho)$ as a smooth ind-scheme in ${\bf {Spc}}$. 
The finite-dimensional
Borel spaces of the type $X^i_G(\rho)$ were first considered by 
Totaro \cite{Totaro1} in order to define the Chow ring of the classifying 
spaces of linear algebraic groups. 
For an admissible gadget $\rho$, we shall 
denote the spaces $colim_i \ U_i$ and $colim_i \left({U_i}/G\right)$ by 
$E_G(\rho)$ and $B_G(\rho)$ respectively. 
The definition of the motivic spaces $X_G$ is based on the following 
observations.

\begin{lem}\label{lem:elem}
For any $X \in \Sm^G_k$, the natural map
$X_G(\rho) \xrightarrow{\cong} X \stackrel{G}{\times} E_G(\rho)$
is an isomorphism in ${\bf {Spc}}$.
\end{lem}
\begin{proof}
We first observe that the map $X \times U_i \to X \times U_{i+1}$
is a closed immersion of smooth schemes and 
$colim_i \ (X \times U_i)$ is the union of its finite-dimensional
subschemes $(X \times U_i)$'s. Moreover, $G$ acts freely on 
$colim_i \ (X \times U_i)$ such that each $X \times U_i$ is $G$-invariant.
Since any $G$-equivariant map $f: colim_i \ (X \times U_i) \to Y$ with
trivial $G$-action on $Y$ factors through a unique map
$colim_i \ \left(X\times U_i\right)/G \to Y$, we see that the map
$X_G(\rho) \to {\left(colim_i \ (X \times U_i)\right)}/G$ is an
isomorphism. Thus we only need to show that the natural map
$colim_i \ (X \times U_i) \to X \times E_G(\rho)$ is an isomorphism.

To show this, it suffices to prove that these two spaces coincide as 
representable functors on ${\bf {Spc}}$. Any object of ${\bf {Spc}}$ is a 
colimit of simplicial sheaves of the form $Y \times \Delta[n]$, where
$Y$ is a smooth scheme. Since 
$\Hom_{{\bf {Spc}}}(colim \ F, -) = lim \ \Hom_{{\bf {Spc}}}(F, -)$,
we only need to show that the map 
\[
\Hom_{{\bf {Spc}}}\left(Y \times \Delta[n], colim_i \ (X \times U_i)\right) \to 
\Hom_{{\bf {Spc}}}\left(Y \times \Delta[n],  X \times E_G(\rho)\right)
\]
is bijective for all $Y \in \Sm_k$ and all $n \ge 0$.

For any $\sF \in \Delta^{\rm op}{\rm Shv}\left({\rm Nis}_k\right)$,
there are isomorphisms
\[
\Hom_{{\bf {Spc}}}\left(Y \times \Delta[n], \sF\right) \cong \sF_n(Y)
= \Hom_{{\rm Shv}\left({\rm Nis}_k\right)}\left(Y, \sF_n\right) = 
\Hom_{{\bf {Spc}}}\left(Y, \sF_n\right),
\]
where $\sF_n$ is the $n$-th level of the simplicial sheaf $\sF$.
Since $colim_i \ (X \times U_i)$ and $X \times E_G(\rho)$ are constant
simplicial sheaves, we are reduced to showing that the map
\[
\Hom_{{\bf {Spc}}}\left(Y, colim_i \ (X \times U_i)\right)
\to 
\Hom_{{\bf {Spc}}}\left(Y,  X \times E_G(\rho)\right)
\]
is bijective.

On the other hand, it follows from \cite[Proposition~2.4]{Voev1} that
\[
\begin{array}{lll}
\Hom_{{\bf {Spc}}}\left(Y, colim_i \ (X \times U_i)\right) & \cong &
colim_i \ \Hom_{{\bf {Spc}}}\left(Y, X \times U_i \right) \\
& \cong & colim_i \ 
\left[\Hom_{{\bf {Spc}}}(Y, X) \times \Hom_{{\bf {Spc}}}(Y, U_i)\right] \\
& \cong & 
\Hom_{{\bf {Spc}}}(Y, X) \times \left[colim_i \ 
\Hom_{{\bf {Spc}}}(Y, U_i)\right] \\
& \cong & \Hom_{{\bf {Spc}}}(Y, X) \times 
\Hom_{{\bf {Spc}}}(Y, colim_i \ U_i) \\   
& \cong & \Hom_{{\bf {Spc}}} \left(Y , X \times E_G(\rho)\right)
\end{array}
\]
which proves the lemma.
\end{proof}

\begin{prop}\label{prop:Rep-ind}
For any two admissible gadgets $\rho$ and $\rho'$ for $G$ and for any
$X \in \Sm^G_k$, there is a canonical isomorphism $X_G(\rho) 
\cong X_G(\rho')$ in $\sH(k)$.
\end{prop} 
\begin{proof}
This was proven by Morel-Voevodsky \cite[Proposition~4.2.6]{MV}
when $X = \Spec(k)$ and a similar argument works in the general case as well.

For $i, j \ge 1$, we consider the smooth scheme $\sV_{i,j} =
{\left(X \times U_i \times V'_j\right)}/G$ and the open 
subscheme $\sU_{i,j} = {\left(X \times U_i \times U'_j\right)}/G$.
For a fixed $i \ge 1$, this yields a sequence ${\left(\sV_{i,j}, 
\sU_{i,j}, f_{i,j}\right)}_{j \ge 1}$, where $\sV_{i,j} \xrightarrow{\pi_{i,j}}
X^i_G(\rho)$ is a vector bundle, $\sU_{i,j} \subseteq \sV_{i,j}$ is an open
subscheme of this vector bundle and $f_{i,j} : \left(\sV_{i,j}, \sU_{i,j}\right)
\to \left(\sV_{i,j+1}, \sU_{i,j+1}\right)$ is the natural map of pairs
of smooth schemes over $X^i_G(\rho)$. Then ${\left(\sV_{i,j}, 
\sU_{i,j}, f_{i,j}\right)}_{j \ge 1}$ is an admissible gadget over 
$X^i_G(\rho)$ in the sense of \cite[Definition~4.2.1]{MV}.
Setting $\sU_{i} = colim_j \ \sU_{i,j}$ and $\pi_i = colim_j \ \pi_{i,j}$, it 
follows from [{\sl loc. cit.}, Proposition~4.2.3] that the map
$\sU_i  \xrightarrow{\pi_i} X^i_G(\rho)$ is an $\A^1$-weak equivalence. 

Taking the colimit of these maps as $i \to \infty$ and using [{\sl loc. cit.},
Corollary~1.1.21], we conclude that the map
$\sU \xrightarrow{\pi}  X_G(\rho)$ is an $\A^1$-weak equivalence,
where $\sU = colim_{i, j} \ \sU_{i,j}$. Reversing the roles of $\rho$ and
$\rho'$, we find that the obvious map $\sU \xrightarrow{\pi'}  X_G(\rho')$ is 
also an $\A^1$-weak equivalence. This yields the canonical
isomorphism $\pi' \circ \pi^{-1} :X_G(\rho) \xrightarrow{\cong}
X_G(\rho')$ in $\sH(k)$.
\end{proof}

\subsubsection{Admissible gadgets associated to a given $G$-scheme}
\label{subsubsection:Add-Sc}
A careful reader may have observed in the proof of 
Proposition~\ref{prop:Rep-ind} 
that we did not really use the fact that $G$ acts freely on an open subset 
$U_i$ (resp. $U'_j$) of the $G$-representation $V_i$ (resp. $V'_j$). One only 
needs to know that for each $i, j \ge 1$, the quotients $(X \times U_i)/G$
and $(X \times U'_j)/G$ are smooth schemes and the maps
${\left(X \times U_i \times V'_j\right)}/G \to (X \times U_i)/G$ and
${\left(X \times V_i \times U'_j\right)}/G \to (X \times U'_j)/G$
are vector bundles with appropriate properties.
This observation leads us to the following variant of
Proposition~\ref{prop:Rep-ind} which will sometimes be useful.

Let $G$ be a linear algebraic group over $k$ and let $X \in \Sch^G_k$.
We shall say that a pair $(V,U)$ of smooth schemes over $k$
is a good pair for the $G$-action on $X$, if $V$ is a $k$-rational 
representation of $G$ and $U \subseteq V$ is a $G$-invariant open subset such
that $X \times U$ is an object of $\Sch^G_{{free}/k}$.
We shall say that the sequence of pairs 
$\rho = {\left(V_i, U_i\right)}_{i \ge 1}$ of smooth schemes over $k$
is an {\sl admissible gadget for the $G$-action on $X$}, if there exists a
good pair $(V,U)$ for the $G$-action on $X$ such that $V_i = V^{\oplus i}$ and 
$U_i \subseteq V_i$ is $G$-invariant open subset such that the following hold 
for each $i \ge 1$.
\begin{enumerate}
\item
$\left(U_i \oplus V\right) \cup \left(V \oplus U_i\right)
\subseteq U_{i+1}$ as $G$-invariant open subsets.
\item
$\codim_{U_{i+2}}\left(U_{i+2} \setminus 
\left(U_{i+1} \oplus V\right)\right) > 
\codim_{U_{i+1}}\left(U_{i+1} \setminus \left(U_{i} \oplus V\right)\right)$.
\item
$\codim_{V_{i+1}}\left(V_{i+1} \setminus U_{i+1}\right)
> \codim_{V_i}\left(V_i \setminus U_i\right)$.
\item
$X \times U_i \in \Sch^G_{{free}/k}$.
\end{enumerate}

Notice that an admissible gadget for $G$ as in 
Definition~\ref{defn:Add-Gad} is an admissible gadget
for the $G$-action on every $G$-scheme $X$.

\begin{prop}\label{prop:Rep-ind-V}
Let $\rho_X$ and $\rho'_X$ be two admissible gadgets for the $G$-action on 
a smooth scheme $X$.
Then there is a canonical isomorphism of motivic spaces
\[
colim_i \ \left(X\stackrel{G}{\times}U_i\right) \ \cong \
colim_j \ \left(X\stackrel{G}{\times}U'_j\right).
\] 
\end{prop}

In view of Proposition~\ref{prop:Rep-ind}, we shall denote a motivic
space $X_G(\rho)$ simply by $X_G$. The motivic space
$B_G$ is called the {\sl classifying space} of the 
linear algebraic group $G$ following the notations of \cite{MV}.
It follows from \cite[Proposition~4.2.3]{MV} that the space $E_G(\rho)$ is 
$\A^1$-contractible in $\sH(k)$ and Lemma~\ref{lem:elem} implies that
$B_G(\rho)$ is the quotient of $E_G(\rho)$ for the free $G$-action.
Given $X \in \Sm^G_k$, {\sl the motivic Borel space} of $X$ will mean the 
motivic space  $X_G \in {\bf {Spc}}$.

\begin{cor}\label{cor:Morita0}
Let $H$ be a closed normal subgroup of a linear algebraic group $G$ and let
$F = G/H$. Let $f: X \to Y$ be a morphism in $\Sm^G_k$ 
which is an $H$-torsor for the restricted action. Then there is a
canonical isomorphism $X_G \cong Y_F$ in $\sH(k)$.
\end{cor}
\begin{proof}
We first observe from \cite[Corollary~12.2.2]{Springer} that
$F$ is also a linear algebraic group over the given ground field $k$.
Let $\rho = (V_i, U_i)_{i \ge 1}$ be an admissible gadget for $G$. The natural
morphism $G \to F$ shows that each $V_i$ is a $k$-rational representation
of $G$ such that the open subset $U_i$ is $G$-invariant,
even though $G$ may not act freely on $U_i$.
In particular, $G$ acts on the product $X \times U_i$ via the diagonal
action.
Since $H$ acts freely on $X$ and $F$ acts freely on $U_i$, 
it follows that the map $X \times U_i \to X \stackrel{G}{\times} U_i$
is a $G$-torsor and hence $\rho =(V_i, U_i)_{i \ge 1}$ is an admissible pair
for the $G$-action on $X$.

Since the map $X \stackrel{G}{\times} U_i \to 
Y \stackrel{F}{\times} U_i$ is an isomorphism for every $i \ge 1$,
we conclude from Proposition~\ref{prop:Rep-ind-V} that
$X_G \cong colim_i \ \left(X\stackrel{G}{\times}U_i\right) 
\xrightarrow{\cong} Y_F(\rho)$ in $\sH(k)$.
\end{proof}

\begin{cor}[Morita isomorphism]\label{cor:Morita1}
Let $H$ be a closed subgroup of a linear algebraic group $G$ and let
$X \in \Sm^H_k$. Let $Y$ denote the space $X \stackrel{H}{\times} G$
for the action $h \cdot (x, g) = (h\cdot x, gh^{-1})$. 
Then there is a canonical isomorphism $X_H \cong Y_G$ in $\sH(k)$.
\end{cor}
\begin{proof}
Define an action of $H \times G$ on $G \times X$ by
\begin{equation}\label{eqn:MoritaI**1}
(h, g) \cdot (x, g') = \left(hx, gg'h^{-1}\right)
\end{equation}
and an action of $H \times G$ on $X$ by $(h, g) \cdot x = hx$. 
Then the projection map $X \times G \xrightarrow{p} X$ is 
$\left(H \times G\right)$-equivariant and a 
$G$-torsor. Hence there is canonical isomorphism $X_H \cong
(X \times G)_{H \times G}$ in $\sH(k)$ by Corollary~\ref{cor:Morita0}.

On the other hand, the projection map $X \times G 
\to X \stackrel{H}{\times} G$ is $\left(H \times G\right)$-equivariant and 
an $H$-torsor. Hence there is a  canonical isomorphism 
$(X \times G)_{H \times G} \cong Y_G$ in $\sH(k)$ again by 
Corollary~\ref{cor:Morita0}.
Combining these two isomorphisms, we get
$X_H \cong Y_G$ in $\sH(k)$.
\end{proof}

\section{The Bar construction}\label{section:Bar}
Let $G$ be a linear algebraic group over $k$ and let $E^{\bullet}_G$ denote
the simplicial scheme 

\begin{equation}\label{eqn:Bar-EG}
E_G^{\bullet}  := \left(\cdots 
\stackrel{\longrightarrow}{\underset{\longrightarrow}\rightrightarrows}
G \times G \times G 
\stackrel{\longrightarrow}{\underset{\longrightarrow}\to} G \times G 
\rightrightarrows G\right) 
\end{equation}
with the face maps $d^i_n: G^{n+1} \to  G^{n}$ ($0 \le i \le n$)
given by the projections
\[
d^i_n(g_0, \cdots , g_n) = (g_0, \cdots , g_{i-1}, \wh{g_i}, g_{i+1}, \cdots , g_n),
\]
where $\wh{g_i}$ means that this coordinate is omitted. The degeneracy maps
$s^i_n: G^{n} \to  G^{n+1}$ ($0 \le i \le n$) are the various diagonals on $G$.
 
For any $X \in \Sch^G_k$, let $E_G^{\bullet} \times X$ denote the product of the 
simplicial scheme $E^{\bullet}_G$ with the constant simplicial scheme $X$.
Thus, $\left(E_G^{\bullet} \times X\right)^n = G^{n+1} \times X$ in which the
face and the degeneracy maps are identity on $X$. 
Notice that as $G$ is smooth over $k$, the face maps of $E_G^{\bullet} \times X$ 
are all smooth. Moreover, the degeneracy maps are all regular closed
immersions. In particular, they have finite Tor-dimension. We shall use these
facts while defining the algebraic $K$-theory of simplicial schemes.

The group $G$ acts on $E_G^{\bullet} \times X$ by 
$g\cdot \left(g_0, \cdots , g_n, gx) =
(g_0g^{-1}, \cdots ,g_ng^{-1}, x\right)$. It is easy to check that all the
face and degeneracy maps are $G$-equivariant with respect to this action
and hence $E_G^{\bullet}$ and  $E^{\bullet}_G \times X$ are
$G$-simplicial schemes such that the projections maps
$X \leftarrow E^{\bullet}_G \times X \rightarrow E_G^{\bullet}$ are
$G$-equivariant. 
 
We also consider the simplicial scheme 

\begin{equation}\label{eqn:Bar-BG}
X_G^{\bullet}  := \left(\cdots 
\stackrel{\longrightarrow}{\underset{\longrightarrow}\rightrightarrows}
G \times G \times X 
\stackrel{\longrightarrow}{\underset{\longrightarrow}\to} G \times X
\rightrightarrows X \right)
\end{equation}
where the face maps $p^i_n : G^n \times X \to G^{(n-1)} \times X$ 
are given by
\begin{equation}\label{eqn:Bar-EG0}
p^i_n\left(g_1, \cdots , g_n, x\right) = \left\{
\begin{array}{ll}
\left(g_2, \cdots , g_n, g_1 x\right) & \mbox{if $i = 0$} \\
\left(g_1, \cdots, g_{i-1}, g_i g_{i+1}, g_{i+2}, \cdots, g_n, x \right) &
\mbox{if $0 < i < n$} \\
\left(g_1, \cdots , g_{n-1}, x \right) & \mbox{if $i = n$}.
\end{array}
\right .
\end{equation}
The degeneracy maps $s^i_n: G^n \times X \to G^{n+1} \times X$
are given by $s^i_n(g_1, \cdots, g_n, x) = (g_1, \cdots , g_{i-1}, e, g_i, 
\cdots , g_n, x)$, where $e \in G$ is the identity element.
One observes again that all the face maps of $X_G^{\bullet}$ are smooth and all 
the degeneracy maps are regular closed immersions.
The simplicial scheme $X_G^{\bullet}$ will be called 
the {\sl bar construction} associated to the $G$-action on $X$. 
We shall often denote $X_G^{\bullet}$ by $B_G^{\bullet}$
when $X = \Spec(k)$, in analogy with the known bar construction associated to 
the classifying spaces in topology. 

It is easy to verify that there is a natural morphism of simplicial schemes
\begin{equation}\label{eqn:Bar-EG1}
\pi_X : E_G^{\bullet} \times X \to X_G^{\bullet} ;
\end{equation}
\[
\pi_X \left(g_0, \cdots , g_n, x \right) = \left(g_0 g^{-1}_1, g_1 g^{-1}_2,
\cdots , g_{n-1}g^{-1}_n, g_n x\right)
\]
which makes $X_G^{\bullet}$ the quotient of $E_G^{\bullet} \times X$
for the free $G$-action. Hence, $\pi_X$ is a principal $G$-bundle of
simplicial schemes and there is a natural isomorphism of
simplicial schemes 
\begin{equation}\label{eqn:Bar-EG2}
{\ov{\pi}}_X: E_G^{\bullet} \stackrel{G}{\times} X \xrightarrow{\cong}
X_G^{\bullet}. 
\end{equation}

\subsection{Geometric models for the bar construction}
\label{subsection:GMBC}
Recall from \cite[\S~4]{MV} that if $G$ is a sheaf of groups on
${\rm Nis}/k$, then a left (resp. right) action of $G$ on a motivic space
(simplicial sheaf) $X$ is a morphism $\mu : G \times X \to X$ (resp.
$\mu : X \times G \to X$) such that the usual diagrams commute.
A (left) action is called (categorically) free if the morphism
$G \times X \to X \times X$ of the form $(g, x) \mapsto (gx, x)$
is a monomorphism. For a $G$-action on $X$, the quotient $X/G$ is the 
motivic space such that
\begin{equation}\label{eqn:P-out} 
\xymatrix@C2pc{
G \times X \ar[r]^>>>>>>{\mu} \ar[d]_{p_X} & X \ar[d] \\
X \ar[r] & X/G}
\end{equation}
is a pushout diagram of motivic spaces.

A principal $G$-bundle ($G$-torsor) over a motivic space $X$ is a morphism
$Y \to X$ together with a free $G$-action on $Y$ over $X$ such that the 
map $Y/G \to X$ is an isomorphism.

Let $G$ be a linear algebraic group over $k$.
Since a simplicial smooth scheme is also a simplicial sheaf on
${\rm Nis}/k$, we see that $E_G^{\bullet} \times X$ and 
$X_G^{\bullet}$ are objects of 
$\Delta^{\rm op}{\rm Shv}\left({\rm Nis}_k\right)$. 
More generally, given a sheaf of sets $F$ on ${\rm Nis}/k$ with a free 
$G$-action, we can consider a simplicial sheaf of sets

\begin{equation}\label{eqn:Bar-EG-Gen}
E^{\bullet}_G(F) = \left(\cdots 
\stackrel{\longrightarrow}{\underset{\longrightarrow}\rightrightarrows}
F \times F \times F 
\stackrel{\longrightarrow}{\underset{\longrightarrow}\to} F \times F 
\rightrightarrows F \right) 
\end{equation}
where the face and the degeneracy maps are given exactly like in
~\eqref{eqn:Bar-EG}. Then $G$ acts freely on $E^{\bullet}_G(F)$
and we denote the quotient by $B^{\bullet}_G(F)$.

Let $\pi: E^{\bullet}_G \to B^{\bullet}_G$ denote the principal $G$-bundle
of ~\eqref{eqn:Bar-EG1}. This is called the universal $G$-torsor over
$B^{\bullet}_G$.  Let $B^{\bullet}_G \xrightarrow{\phi} \sB_G$ be a trivial 
cofibration of motivic spaces with $\sB_G$ fibrant. We shall assume in the 
rest of this section that  
\begin{equation}\label{eqn:special}
H^1_{\rm Nis}\left(U, G\right) \xrightarrow{\cong} 
H^1_{\rm et}\left(U, G\right) \ \ {\rm for \ all} \ \ U \in \Sm_k.
\end{equation}

This condition is equivalent to saying that all {\'e}tale locally trivial
$G$-torsors on smooth schemes over $k$ are also locally trivial in the
Nisnevich topology. Such a condition is always satisfied if $G$ is
{\sl special}. Under this condition, it follows from 
\cite[Proposition~4.1.15]{MV} (see also [{\sl ibid.}, p. 131]) that there is a 
universal $G$-torsor $\sE_G \to \sB_G$ such that for any sheaf of 
sets $F$ on ${\rm Nis}/k$ with free $G$-action, there is a Cartesian square 
of motivic spaces

\begin{equation}\label{eqn:motivic0}
\xymatrix@C2pc{
E^{\bullet}_G(F) \ar[r]^>>>>>{{\bar{\phi}}_F} \ar[d] & \sE_G \ar[d] \\
B^{\bullet}_G(F) \ar[r]_>>>>{\phi_F} & \sB_G}
\end{equation}
where $\phi_{\sF}$ is well defined up to a simplicial homotopy.
The following result is an easy consequence of \cite[Proposition~4.1.20]{MV}.

\begin{lem}\label{lem:Bar-ind}
Let $\rho = \left(V_i, U_i\right)_{i \ge 1}$ be an admissible gadget for $G$
and let $F = {\rm colim}_i \ U_i$. Then the horizontal morphisms in the 
Cartesian diagram ~\eqref{eqn:motivic0} are simplicial weak equivalences.
\end{lem}
\begin{proof}
Assume that $\rho$ is given by a good pair $(V,U)$ and let $x \in U$ be
a $k$-rational point.
Using the condition ~\eqref{eqn:special} and \cite[Proposition~4.1.20]{MV},
we have to show that if $S$ is smooth henselian local and 
$E = S \times G \to S$ is the trivial principal $G$-bundle over $S$, then the 
morphism $E \stackrel{G}{\times} U_i \to S$ splits for some $i \gg 0$ 
in order to prove that $\phi_F$ is a simplicial weak equivalence. 
To find such a splitting, it suffices to find a $G$-equivariant
morphism $G \to U$. But it is given by the $G$-orbit $Gx \inj U$.

To show that ${\bar{\phi}}_F$ is a simplicial weak equivalence, we need to show
using \cite[Lemma~3.1.11]{MV} that for every smooth henselian local ring
$R$, the map $E^{\bullet}_G(F) (R) \to \sE_G(R)$ is a weak equivalence of 
simplicial sets.

The vertical maps in ~\eqref{eqn:motivic0} are local fibrations with
fiber $G$ ({\sl cf.} \cite[Definition~2.1.11, Lemma~4.1.12]{MV}).
These local fibrations yield for us a commutative diagram of simplicial sets
\begin{equation}\label{eqn:Bar-ind}
\xymatrix@C2pc{
G(R) \ar[r] \ar[d] & E^{\bullet}_G(F)(R) \ar[r] \ar[d] & B^{\bullet}_G(F)(R) \ar[d] 
\\
G(R) \ar[r] & \sE_G(R) \ar[r] & \sB_G(R)}
\end{equation}
where the rows are fibration sequences of simplicial sets.
It follows from \cite[Proposition~3.6.1]{Hovey1} that the rows remain
fibration sequences upon taking the geometric realizations. 

We have shown above
that the right vertical map is a simplicial weak equivalence and hence a
weak equivalence of geometric realizations. The left vertical map is an
isomorphism. We conclude from the long exact sequence of homotopy groups of
fibrations that the middle vertical map is a weak equivalence of
geometric realizations. Hence the map $E^{\bullet}_G(F) (R) \to \sE_G(R)$ is a 
weak equivalence of simplicial sets.
\end{proof}

The following result shows that for an action of a linear algebraic group $G$
on a smooth scheme $X$, the Borel spaces are the geometric models for the
associated bar construction. This is a generalization of the geometric
construction of $B^{\bullet}_G$ given in \cite[\S~4]{MV}. As a consequence,
we get a more conceptual proof of Proposition~\ref{prop:Rep-ind}.

\begin{prop}\label{prop:Bar-ind-Gen}
Let $\rho = \left(V_i, U_i\right)_{i \ge 1}$ be an admissible gadget for a 
linear algebraic group $G$ over $k$. Then for any $X \in \Sm^G_k$, there is a
canonical isomorphism $X_G(\rho) \cong X^{\bullet}_G$ in $\sH_{\bullet}(k)$.
In particular, $X_G(\rho)$ does not depend on the choice of the admissible
gadget $\rho$.
\end{prop}
\begin{proof}
It suffices to show using Lemma~\ref{lem:elem} that there is a 
canonical isomorphism $E_G(\rho)\stackrel{G}{\times} X \cong X^{\bullet}_G$ in 
$\sH_{\bullet}(k)$, where $E_G(\rho) = {\rm colim}_i \ U_i$.

Taking $F = E_G(\rho) = {\rm colim}_i \ U_i$ in ~\eqref{eqn:motivic0}, we have 
a Cartesian diagram 
\begin{equation}\label{eqn:motivic1}
\xymatrix@C2pc{
E^{\bullet}_G(F) \ar[r]^>>>>>{{\bar{\phi}}} \ar[d] & \sE_G \ar[d] \\
B^{\bullet}_G(F) \ar[r]_>>>>{\phi} & \sB_G}
\end{equation}
of motivic spaces. We thus have a $G$-equivariant map $E^{\bullet}_G(F) \times X 
\xrightarrow{(\bar{\phi} \times {\rm id}) = {\bar{\phi}}_X} 
\sE_G \times X$, and hence a map of quotients 
$E^{\bullet}_G(F) \stackrel{G}{\times} X \xrightarrow{\phi_X} 
\sE_G \stackrel{G}{\times} X$.
We first show that this map is a simplicial weak equivalence.

To do this, we consider the commutative diagram

\begin{equation}\label{eqn:Bar-ind-Gen}
\xymatrix@C2pc{
E^{\bullet}_G(F) \stackrel{G}{\times} X  \ar[r]^<<<<{\phi_X} \ar[d] & 
\sE_G \stackrel{G}{\times} X \ar[d] \\
B^{\bullet}_G(F) \ar[r]_{\phi} & \sB_G}
\end{equation}
where the vertical maps are local fibrations with fiber $X$.
To show that $\phi_X$ is a simplicial weak equivalence, we need to show
that the map $\left(E^{\bullet}_G(F) \stackrel{G}{\times} X\right)(R) \to  
\left(\sE_G\stackrel{G}{\times} X\right) (R)$ is a weak 
equivalence of simplicial sets for every smooth henselian local ring $R$.

To show this, we use the commutative diagram 
\begin{equation}\label{eqn:Bar-ind}
\xymatrix@C2pc{
X(R) \ar[r] \ar[d] & \left(E^{\bullet}_G(F) \stackrel{G}{\times} X\right)(R) 
\ar[r] \ar[d] & \left(B^{\bullet}_G(F)\right)(R) \ar[d] \\
X(R) \ar[r] & \left(\sE_G\stackrel{G}{\times} X\right) (R) \ar[r] & \sB_G(R)}
\end{equation}
where the rows are fibration sequences of simplicial sets.
We have shown in Lemma~\ref{lem:Bar-ind} that the right vertical map
is a weak equivalence. One now argues as in the proof of
Lemma~\ref{lem:Bar-ind} to conclude that the map $\phi_X$ is a 
simplicial weak equivalence. 

We have thus shown that the map $E^{\bullet}_G(F) \stackrel{G}{\times} X 
\xrightarrow{\phi_X} \sE_G \stackrel{G}{\times} X$ is a simplicial weak 
equivalence. By taking $F = G$, we similarly see that the map
$X^{\bullet}_G \to \sE_G \stackrel{G}{\times} X$ is a simplicial weak 
equivalence. Combining the two, we get a simplicial weak equivalence
$X^{\bullet}_G \xrightarrow{\cong}  E^{\bullet}_G(F) \stackrel{G}{\times} X$.

We next think of $F = E_G(\rho)$ as a constant simplicial sheaf and
consider the $G$-equivariant map of simplicial sheaves $u: F \to E^{\bullet}_G(F)$
given in degree $n$ by $u_n(a) = (a, \cdots , a)$. 
This in turn gives a map 
$F \stackrel{G}{\times} X \xrightarrow{u_X} 
E^{\bullet}_G(F) \stackrel{G}{\times} X$. To finish the proof of the proposition,
it suffices now to show that this map is an $\A^1$-weak equivalence.
By \cite[Proposition~2.2.14]{MV}, it is sufficient to show
that each $u_{X, n}: (F \times X)/G \to
\left(F^{n+1} \times X\right)/G$ is an $\A^1$-weak equivalence.
In order to do so, it suffices to show that the projection
$\left(F^{n+1} \times X\right)/G \xrightarrow{p_{X,n}} 
\left(F^{n} \times X\right)/G$
is an $\A^1$-weak equivalence for each $n > 0$. 

However, it follows from \cite[Lemma~4.2.9]{MV} that for each $n, i \ge 1$,
the map $\left(F \times (U_i)^n \times X\right)/G \to
\left((U_i)^n \times X\right)/G$ is an $\A^1$-weak equivalence.
Taking the colimit as $i \to \infty$, we see from \cite[Proposition~19]{Deligne}
that $p_{X,n}$ is an $\A^1$-weak equivalence. This completes the proof of the
proposition.
\end{proof}

\section{Equivariant motivic cobordism}\label{section:EMC}
In this section we define our equivariant motivic cobordism and study
its basic properties. We need to work in the motivic stable homotopy category
in order to define our equivariant motivic cobordism. We briefly recall this
below.

\subsection{The motivic stable homotopy category}\label{subsection:MSC}
The motivic stable homotopy category $\sS\sH(k)$ is the homotopy category
of motivic $T$-spectra over $k$, where $T$ is the pointed space
$(\P^1_k, \infty)$. A $T$-spectrum (or a motivic spectrum) is a sequence of 
pointed spaces
$E = (E_0, E_1, \cdots )$ with the bounding maps $\sigma : T \wedge E_i \to
E_{i+1}$. A morphism of $T$-spectra is a morphism of the sequences of
pointed spaces which commute with the bounding maps. The category of
motivic spectra is denoted by ${\bf Spt}$. We mention the
following facts about motivic spectra for the convenience of the reader.

\begin{enumerate}
\item The category ${\bf Spt}$ has a model
structure in which weak equivalence (resp. fibration) is the levelwise 
$\A^1$-weak equivalence (resp. fibration) and a cofibration is the one
which has the left lifting property with respect to all acyclic fibrations.
This model structure is proper and cellular.
\item
There are isomorphisms of the pointed spaces
\[
T \cong {\A^1}/{\A^1 \setminus \{0\}} \cong S^1 \wedge \G_m
\]
where $S^1$ is the simplicial circle ${\Delta^1}/{\partial \Delta^1}$.
\item
For any $n \ge 0$, there are functors $F_n : Spc_{\bullet} \to {\bf Spt}$
and $Ev_n : {\bf Spt} \to Spc_{\bullet}$ with ${F_n(A)}_m  = T^{m-n}A$ if
$m \ge n$ and zero otherwise, and $Ev_n(E) = E_n$. These functors
form adjoint pairs $(F_n, Ev_n)$. The functor $F_0$ will often be denoted
by 
\[
\Sigma^{\infty}_T :  Spc_{\bullet} \to {\bf Spt}.
\]
\item
For any $A \in Spc_{\bullet}$ and $a \ge b \ge 0$, there is a Nisnevich 
sheaf $\pi^{\A^1}_{a,b}(A)$ on ${\rm Nis}/k$ which is the sheafification of the
presheaf $U \mapsto \Hom_{\sH_{\bullet}(k)}(U_{+} \wedge S^{a,b}, A)$. The 
space $S^{a,b}$ here is the {\sl weighted sphere} 
$S^{a-b}\wedge ({\G_m})^{\wedge b}$.
\item
The stable motivic homotopy category $\sS\sH(k)$ is obtained by localizing
the levelwise model structure (as in (1) above) so that the endomorphism
$\Sigma_T$ given by $(E_0, E_1, \cdots ) \mapsto (T \wedge E_0, T \wedge E_1,
\cdots )$ becomes invertible. More precisely, this is obtained as follows.

Given a $T$-spectrum $E = (E_0, E_1, \cdots )$, $U \in \Sm_k$ and a map
$f : U_{+} \wedge S^{2n+a, n+b} \to E_n$, we get maps
\[
U_{+} \wedge S^{2n+2+a, n+b+1} = U_{+} \wedge S^{2n+a, n+b} \wedge S^{2,1}
\to S^{2,1} \wedge E_n = T \wedge E_n \xrightarrow{\sigma} E_{n+1}.
\]
In particular, there is a sequence of motivic sheaves
\[
\cdots \to \pi^{\A^1}_{2n+a,n+b}(E_n) \to \pi^{\A^1}_{2n+2+a,n+1+b}(E_{n+1}) \to
\cdots 
\]
and this yields the stable homotopy sheaf
$\pi^{\A^1}_{a,b}(E) : = {\underset{n}\varinjlim} \ \pi^{\A^1}_{2n+a,n+b}(E_n)$.
The stable homotopy category $\sS\sH(k)$ is the localization of
the levelwise model structure on ${\bf Spt}$ so that $f : E \to F$ is
a stable weak equivalence if the map $f_* :\pi^{\A^1}_{a,b}(E) \to
\pi^{\A^1}_{a,b}(F)$ is an isomorphism for all $a \ge b \ge 0$.
The category $\sS\sH(k)$ is a triangulated category symmetric
monoidal category in which the 
shift functor is given by $E \mapsto S^1 \wedge E$ and ${\bf L}\Sigma_T$
becomes invertible with the inverse given by 
$(E_0, E_1, \cdots ) \mapsto (pt, E_0, E_1, \cdots )$.
We shall denote ${\bf L}\Sigma_T$ still by
$\Sigma_T$ in what follows. 
\item
The Quillen pair $(\Sigma^{\infty}_T, Ev_0) : Spc_{\bullet} 
\stackrel{\longleftarrow}{\to} {\bf Spt}$ induces an adjoint pair
of derived functors
$(\Sigma^{\infty}_T, Ev_0) : \sH_{\bullet}(k) \stackrel{\longleftarrow}{\to} 
\sS\sH(k)$.
\item
Let $\Sigma_s$ and $\Sigma_t$ denote the endofunctors 
$E \mapsto S^1 \wedge E$ and $E \mapsto \G_m \wedge E$ on $\sS\sH(k)$
and let $\Sigma^{a,b} = \Sigma^{a-b}_s\circ \Sigma^b_t$ for $a \ge b \ge 0$.
For any $E \in \sS\sH(k)$, we define the $E$-cohomology theory on
$\sS\sH(k)$ by
\begin{equation}\label{eqn:E-coh}
E^{a,b}(F) = \Hom_{\sS\sH(k)}\left(F, \Sigma^{a,b}E\right).
\end{equation}
For $X \in Spc$, one defines
\begin{equation}\label{eqn:E-coh1}
E^{a,b}(X) = 
\Hom_{\sS\sH(k)}\left(\Sigma^{\infty}_T X_{+}, \Sigma^{a,b}E\right). 
\end{equation}
\end{enumerate}

\subsection{The motivic Thom spectrum}\label{subsection:TS}
We now recall the construction of the motivic Thom spectrum $MGL$
defined by Voevodsky \cite{Voev1}. This will play the main role in our
definition and further study of the equivariant motivic cobordism.

Recall that for a vector bundle $p : E \to B$ on a smooth scheme $B$
with the 0-section $0_B$, the {\sl Thom space} of $E$ is the pointed  
space $Th(E) = E/{(E \setminus 0_B)} \in Spc_{\bullet}$.
There is a canonical isomorphism 
\begin{equation}\label{eqn:Thom0}
Th(E \oplus \sO_B) \cong T \wedge Th(E). 
\end{equation}
In particular, for a trivial bundle $E$ of rank $n$, one checks easily that
$Th(E) \cong T^n \wedge B_{+}$.

It follows from \cite[Proposition~4.3.7]{MV} that there is a canonical
isomorphism $BGL_n \cong G(n, \infty) = colim_i \ G(n, i)$, where $BGL_n$ is 
the classifying space ({\sl cf.} \S~\ref{subsection:Borel}) of the 
General linear group of rank $n$ and $G(n, i)$ is the Grassmannian scheme of
$n$-dimensional linear subspaces of $k^i$. The colimit of the universal 
rank $n$ vector bundles on $G(n, i)$'s defines a unique rank $n$ vector
bundle $p_n : E_n \to BGL_n$. The rank $n+1$ vector bundle
$E_n \oplus \sO_{BGL_n} \to BGL_n$ defines a unique map (a closed immersion)
$BGL_n \xrightarrow{i_n} BGL_{n+1}$ such that the diagram
\begin{equation}\label{eqn:bounding0}
\xymatrix@C.8pc{
E_n \oplus \sO_{BGL_n} \ar[r]^>>>>{v_n} \ar[d] & E_{n+1} \ar[d]^{p_{n+1}} \\
BGL_n \ar[r]_{i_n} & BGL_{n+1}}
\end{equation}
is Cartesian. This yields a unique map $Th(i_n) : T \wedge Th(U_n) 
\to Th(E_{n+1})$. The motivic Thom spectrum $MGL$ is given by
\begin{equation}\label{eqn:Thom-S} 
MGL : = (MGL_0, MGL_1, \cdots ), 
\end{equation}
where $MGL_n = Th(U_n)$ and the bounding map 
$e_n = Th(i_n) : T \wedge MGL_n \to MGL_{n+1}$.
The resulting bi-graded $MGL$-cohomology theory on $\Sm_k$ 
as defined in ~\eqref{eqn:E-coh} is called the {\sl motivic cobordism}
theory. This is an oriented ring cohomology theory in the sense of 
\cite{Panin1}.

\subsubsection{The ring structure on $MGL$}\label{subsubsection:MG-Ring}
Recall from \cite{Voev1} that $\sS\sH(k)$ is a symmetric monoidal category
often denoted by $\left(\sS\sH(k), \wedge, S^0\right)$, where $S^0$ is the 
spectrum $\Sigma^{\infty}_T\left({\Spec(k)}_{+}\right)$. We also recall that
$MGL$ is a represented by a symmetric spectrum. This spectrum is
a commutative ring spectrum in the sense that it is a commutative, associative 
and unitary monoid in $\left(\sS\sH(k), \wedge, S^0\right)$. In particular, the
isomorphism $\Sigma^{\infty}_T\left((X \times X')_{+}\right) \cong
\Sigma^{\infty}_T\left(X_{+}\right) \wedge \Sigma^{\infty}_T\left(X'_{+}\right)$
implies that there is an external product
\begin{equation}\label{eqn:MGL-prod}
MGL^{a,b}(X) \otimes_{\Z} MGL^{a',b'}(X') \xrightarrow{ext}
MGL^{a+a', b+b'}(X \times X')
\end{equation}
for $X, Y \in Spc$.
For $X = X'$, this yields the internal product
\[MGL^{a,b}(X) \otimes_{\Z} MGL^{a',b'}(X) \xrightarrow{ext}
MGL^{a+a', b+b'}(X \times X) \xrightarrow{\Delta^*_X}
MGL^{a+a', b+b'}(X)
\]
such that $\alpha \cdot \beta = (-1)^{aa'} \beta \cdot \alpha$.
The unit element $1 \in MGL^{0,0}(pt)$ is the constant map 
$\Sigma^{\infty}_T(S^0) \to MGL$ mapping $T^n$ to the base point of $MGL_n$.
For any $X \in Spc$, the unit element $1 \in MGL^{0,0}$ is given by the
composite map $\Sigma^{\infty}_T(X_{+}) \to \Sigma^{\infty}_T(S^0) \to MGL$
induced by the structure map $X \to pt$.

\subsection{The equivariant motivic cobordism}
\label{subsection:EMCT}
\begin{defn}\label{defn:EMC}
Let $G$ be a linear algebraic group over $k$ and let $X \in \Sm^G_k$.
For $0 \le b \le a$, we define the {\sl equivariant motivic cobordism} groups 
of $X$ by
\begin{equation}\label{eqn:EMC0}
MGL^{a,b}_G(X) : =  MGL^{a,b}(X_G)
\end{equation}
where $X_G$ is a Borel space of the type $X_G(\rho)$ as in 
\S~\ref{subsection:Borel}. We set 
\[
MGL^{*,*}_G(X) = 
{\underset{0 \le b \le a}\oplus} \ MGL^{a,b}_G(X).
\]
\end{defn}
We have seen in Proposition~\ref{prop:Rep-ind-V} that as an object of 
$\sH(k)$, $X_G$ does not depend on the choice of an admissible gadget $\rho$.

\begin{lem}\label{lem:prod}
For $X, X' \in \Sm^G_k$, there are external and internal products
\[
MGL^{a,b}_G(X) \otimes_{\Z} MGL^{a',b'}_G(X') \xrightarrow{\times}
MGL^{a+a', b+b'}_G(X \times X') ;
\]
\[
MGL^{a,b}_G(X) \otimes_{\Z} MGL^{a',b'}_G(X) \xrightarrow{\cup}
MGL^{a+a', b+b'}_G(X)
\]
such that $\alpha \cup \beta = (-1)^{aa'} \beta \cup \alpha$.
In particular, $MGL^{*,*}_G(X)$ is a bi-graded commutative ring.
\end{lem}
\begin{proof}
Set $Z = X \times X'$. In view of ~\eqref{eqn:MGL-prod}, we only need to show 
that there is a natural morphism $Z_G \to X_G \times X'_G$ in
$\sH(k)$. This will produce the desired map
$MGL^{*,*}(X_G \times X'_G) \to MGL^{*,*}(Z_G) = MGL^{*,*}_G(Z)$.

Let $\rho = (V_i, U_i)_{i \ge 1}$ and $\rho' = (V'_j, U'_j)_{j \ge 1}$ be two
admissible gadgets for $G$. Set 
$\sU_{i,j} = {\left(Z \times U_i \times U'_j\right)}/G$ and
$\sU = colim_{i,j} \ \sU_{i,j}$. 
We have shown in the proof of Proposition~\ref{prop:Rep-ind} that
the maps $\sU \to Z_G(\rho)$ and $\sU \to Z_G(\rho')$ are $\A^1$-weak 
equivalences. In other words, there is a canonical isomorphism
$Z_G \cong \sU$ in $\sH(k)$. 

On the other hand, there are natural projection maps $\sU \to X_G(\rho)$
and $\sU \to X'_G(\rho')$ which yield the desired morphism
$\sU \to X_G(\rho) \times X'_G(\rho')$ in $\sH(k)$.  
The internal product on $MGL^{*,*}_G(X)$ is obtained by composing the
external product with the pull-back via the diagonal of $X$.
\end{proof}

\section{Basic properties of equivariant motivic cobordism}
\label{section:Basic}
In this section we derive some basic properties of the equivariant motivic
cobordism. The main result is show that the equivariant motivic cobordism
is an example of an ``oriented cohomology theory'' on $\Sm^G_k$.
In order to do this, we first study the notion of Thom and Chern structures
({\sl cf.} \cite{Panin1}) on the motivic cobordism of ind-schemes.

\subsection{Ind-schemes as motivic spaces}\label{subsection:IS}
An {\sl ind-scheme} $X$ in this text will mean an object of the type
$colim_i \ X_i$ in $Spc$, where
$\{X_i\}_{i \ge 0}$ is a sequence of cofibrations (monomorphisms)
\[
X_0 \xrightarrow{f_0} X_{1} \xrightarrow{f_1} X_{2} \xrightarrow{f_2} \cdots 
\]
in $Spc$ such that each 
$X_i \in \Sm_k$. This definition of ind-schemes is more restrictive than
the more general notion where one considers arbitrary filtered colimit of
schemes. We shall denote the category of smooth ind-schemes by ${\bf ISm}_k$.

A morphism in ${\bf ISm}_k$ is a morphism of sequences $f: \{X_i\} \to \{Y_i\}$
of smooth schemes. We shall say that a morphism $f: Y \to X$ of ind-schemes has 
a given property (e.g., {\'e}tale, smooth, affine, projective) if each map
$f_i : Y_i \to X_i$ of the underlying sequences has that property.
The morphism $f$ is called a closed (resp. an open) immersion
if each $f_i : X_i \to Y_i$ is a closed (resp. an open) immersion of smooth
schemes such that $X_i = X_{i+1} \cap Y_i$ for each $i \ge 0$.
It is clear that the complement of a closed (resp. open)
immersion $Y \inj X$ of ind-schemes is the open (resp. closed)
ind-subscheme $U = colim_i \ (X_i \setminus Y_i)$.
For an ind-scheme $X$ and two ind-subschemes $U, V \subset X$,
the union $U \cup V$ and intersection $U \cap V$ are defined as the
colimits of the individual unions or intersections.

Since a colimit of motivic spaces commutes with finite limits, one has that 
$colim_i\ (X_i \times Y_i) \xrightarrow{\cong} X \times Y$.
One also checks that for a morphism of ind-schemes
$f : X \to Y$, there is a canonical isomorphism 
$colim_i \ ({Y_i}/{X_i}) \xrightarrow{\cong} Y/X$ in  $Spc$. 
There are obvious functors $\Sm_k \to {\bf ISm}_k \to
Spc$ where the first functor
$X \mapsto (X \xrightarrow{id} X \xrightarrow{id} \cdots )$ is 
a full embedding. 
The following result about the cofibration and weak equivalence between
ind-schemes will be used frequently in this text.

\begin{lem}\label{lem:Cof-weq}
Let $f : \{X_i\} \to \{Y_i\}$ be a morphism between two sequences of 
motivic spaces such that each $f_i : X_i \to Y_i$ is a cofibration
(resp. $\A^1$-weak equivalence). Then $f : colim_i \ X_i \to colim_i \ Y_i$
is also a cofibration (resp. $\A^1$-weak equivalence).
\end{lem}
\begin{proof}
It follows from \cite[Corollary~1.1.21]{MV} that the natural map
$hocolim_i \ X_i \to colim_i \ X_i$ is a simplicial weak equivalence and
hence an $\A^1$-weak equivalence by \cite[Proposition~3.3.3]{Hirsc}.
On the other hand, it follows from \cite[Theorem~18.5.3]{Hirsc} that
the map $hocolim_i \ X_i \to hocolim_i \ Y_i$ is an $\A^1$-weak
equivalence as every motivic space is cofibrant. The assertion about
cofibration is obvious.
\end{proof}

\subsubsection{Vector bundles and projective bundles over ind-schemes}
\label{subsubsection:VEC}
A vector bundle $p : E \to X$ of rank $n$ over an ind-scheme is
a sequence of vector bundles $\{E_i \xrightarrow{p_i} X_i\}$ of rank $n$
such that $E_i = f^*_i(E_{i+1})$ for each $i \ge 0$ and $E = colim_i \ E_i$.
The maps $i_E: X \inj E$ and $j_E: U_E \inj E$ will denote the 0-section
embedding and its complement respectively. One checks that
$U_E = colim_i \ U_i$ is an ind-scheme, where 
$U_i = E_i \setminus X_i$. 
The pointed space $Th(E) = E/{U_E} = colim_i \ Th(E_i)$ is the Thom space of 
the vector bundle $E$ over the ind-scheme $X$.
It follows from \cite[Example~3.2.2]{MV} and Lemma~\ref{lem:Cof-weq}
that a vector bundle morphism $p: E \to X$ between ind-schemes is
an $\A^1$-weak equivalence.
A sequence of maps
\[
0 \to E' \to E \to E'' \to 0
\]
between vector bundles on an ind-scheme $X$ is exact if
its restriction to each $X_i$ yields a short exact sequence of
associated locally free sheaves.

A vector bundle on an ind-schemes also gives rise to the associated projective
bundle $\pi: \P(E) \to X$ in ${\bf ISm}_k$ and $\P(E) = colim_i \ \P(E_i)$.
Moreover, we have the tautological line bundle $\sO_{E_i}(-1)$
on $\P(E_i)$ and cofibrations $ \P(E_i) \xrightarrow{h_i} \P(E_{i+1})$ such that 
$\sO_{E_i}(-1) = h^*_i\left(\sO_{E_{i+1}}(-1)\right)$. The colimit
of these line bundles gives the tautological line bundle $\sO_{E}(-1)$
on $\P(E)$.
The following elementary lemma shows that $Th(E)$ is the colimit
of a cofiber sequence of Thom spaces over smooth schemes.

\begin{lem}\label{lem:Cofib-qt}
Given a Cartesian square 
\begin{equation}\label{eqn:cofib-qt1}
\xymatrix@C.8pc{
V \ar[r]^{j'} \ar[d]_{i'} & Y \ar[d]^{i} \\
U \ar[r]_{j} & X}
\end{equation}
of monomorphisms in $\Sm_k$, the map $e: Y/V \to X/U$ is a cofibration in
$Spc$.   
\end{lem}
\begin{proof}
Since ~\eqref{eqn:cofib-qt1} is a Cartesian square of injective maps of
smooth schemes, it is easy to check that the map $Z = Y \coprod_V U \to X$
is a monomorphism, i.e., a cofibration in $Spc$ and the map
$Y/V \to Z/U$ is an isomorphism in $Spc$. So we only need to show that the
map $Z/U \to X/U$ is a cofibration. But this follows directly from
\cite[Lemma~7.2.15]{Hirsc}.
\end{proof}

\subsubsection{Motivic cobordism of ind-schemes}\label{subsubsection:MCIS}
The motivic cobordism of an ind-scheme (or any motivic space) $X$ is the
generalized cohomology 
\[
MGL^{a,b}(X) =  
\Hom_{\sS\sH(k)}\left(\Sigma^{\infty}_T X_{+}, \Sigma^{a,b}MGL\right).
\]

Given an ind-scheme $X$ and a closed ind-subscheme $Y$, the motivic cobordism
$MGL^{a,b}_Y(X)$ is defined to be the motivic cobordism of the
motivic space $X/(X \setminus Y)$. It follows from 
\cite[Proposition~6.5.3]{Hovey1} that there is a functorial long exact sequence
\begin{equation}\label{eqn:LES}
\cdots \to MGL^{a-1,b}(X\setminus Y) \xrightarrow{\partial} MGL^{a,b}_Y(X) \to
MGL^{a,b}(X) \to MGL^{a,b}(X \setminus Y) \to \cdots .
\end{equation}

Since $MGL$ is a commutative ring spectrum and since
$X_{+} \wedge pt \cong pt$ for any $X \in Spc$, we see that given
classes $\alpha_1 \in MGL^{a_1,b_1}_{Y_1}(X_1)$ and 
$\alpha_2 \in MGL^{a_2,b_2}_{Y_2}(X_2)$, there are maps
\[
\frac{(X_1 \times X_2)}{X_1 \times (X_2 \setminus Y_2) \cup (X_1 \setminus Y_1) 
\times X_2} \to MGL \wedge MGL \to MGL
\]
which yields a class $\alpha_1 \times \alpha_2 \in 
MGL^{a_1+ a_2,b_1+ b_2}_{Y_1 \times Y_2}(X_1 \times X_2)$.
If $X_1 = X_2 = X$, we can compose with the diagonal map
\[
\frac{X}{X \setminus (Y_1 \cap Y_2)} \to   
\frac{(X \times X)}
{X \times (X \setminus Y_2) \cup (X \setminus Y_1) \times X}
\]
to get a product
\begin{equation}\label{eqn:prod*}
MGL^{a_1,b_1}_{Y_1}(X) \times MGL^{a_2,b_2}_{Y_2}(X) \xrightarrow{\cup}
MGL^{a_1+ a_2,b_1+ b_2}_{Y_1 \cap Y_2}(X).
\end{equation}

\subsubsection{Milnor sequence}\label{subsubsection:Milnor}
Suppose there is a sequence of cofibrations of pointed spaces
\begin{equation}\label{eqn:cofiber}
pt \to X_0 \xrightarrow{f_0} X_1 \xrightarrow{f_1} X_2 \xrightarrow{f_2} \cdots
\end{equation}
with colimit $X$ and let $g_i : X_i \to X$ be the natural cofibration.
Let $E$ denote a motivic $\Omega$-spectrum which is a commutative ring 
spectrum. Given a class $\zeta \in E^{a,b}(X)$, we get a natural class
$(\zeta_i) = \left(g^*_i(\zeta)\right) \in {\underset{i} \prod} \ E^{a,b}(X_i)$. 
This defines an action of $ E^{*,*}(X)$ on ${\underset{i} \prod} \ E^{*,*}(X_i)$
by $\left(\zeta, (a_i)\right) \mapsto \left(\zeta_i a_i\right)$.
One also checks that there is a natural exact sequence
\[
0 \to {\underset{i} \varprojlim} \ E^{a,b}(X_i) \to 
{\underset{i} \prod} \ E^{a,b}(X_i) \xrightarrow{\theta}
{\underset{i} \prod} \ E^{a,b}(X_i) \to {\underset{i} \varprojlim^1}  
E^{a,b}(X_i) \to 0
\]
where $\theta = \left(id - (f_i)\right)$ is $E^{*,*}(X)$-linear.
In particular, there is a natural action of $E^{*,*}(X)$ on
${\underset{i} \varprojlim} \ E^{*,*}(X_i)$ and ${\underset{i} \varprojlim^1}  
E^{*,*}(X_i)$.
The following result from \cite[Proposition~7.3.2]{Hovey1}
explains the relation between the motivic cobordism of the colimit $X$ and
its components.

\begin{prop}[Milnor exact sequence]\label{prop:Limit}
There is a natural exact sequence
\begin{equation}\label{eqn:MES1}
0 \to {\underset{i} \varprojlim^1} \ E^{a-1, b}(X_i) \to E^{a,b}(X) \to
{\underset{i} \varprojlim} \ E^{a, b}(X_i) \to 0
\end{equation}
which is compatible with the action of $E^{*,*}(X)$.
\end{prop}
\begin{proof}
This is proven in \cite[Proposition~7.3.2]{Hovey1} (see also
\cite[Lemma~A.34]{PPR} and \cite[Corollary~9.3.3]{BK} ) and we give a sketch.

The cofiber sequence ~\eqref{eqn:cofiber} yields a sequence of
fibrations of simplicial sets
\[
S\left(X, \Sigma^{a,b}E\right) \to \cdots \to S\left(X_1, \Sigma^{a,b}E\right) 
\xrightarrow{f^*_0} S\left(X_0, \Sigma^{a,b}E\right) \to
S\left(pt, \Sigma^{a,b}E\right),
\]
where $S\left(X, \Sigma^{a,b}E\right) = lim_i \
S\left(X_i, \Sigma^{a,b}E\right)$. This yields a Cartesian diagram
\begin{equation}\label{eqn:cofiber2}
\xymatrix{
S\left(X, \Sigma^{a,b}E\right) \ar[r] \ar[d] & 
{\underset{i} \prod} \
\Hom_{Ssets} \left(\Delta^1, S\left(X_i, \Sigma^{a,b}E\right)\right) \ar[d]^{i^*} \\
{\underset{i} \prod} \ S\left(X_i, \Sigma^{a,b}E\right) \ar[r]_<<<<<{\theta^*} &
{\underset{i} \prod} \ S\left(X_i, \Sigma^{a,b}E\right) \times
S\left(X_i, \Sigma^{a,b}E\right)}
\end{equation}
in which the right vertical map is a fibration which is the restriction
of a 1-simplex to its boundary. In particular, we get a fiber sequence
\begin{equation}\label{eqn:cofiber1}
{\underset{i} \prod} \ S\left(S^1 \wedge X_i, \Sigma^{a,b}E\right) 
\to S\left(X, \Sigma^{a,b}E\right) \to 
{\underset{i} \prod} \ S\left(X_i, \Sigma^{a,b}E\right).
\end{equation}
Moreover, the above commutative diagram also shows that the diagram
\[
\xymatrix@C.8pc{
{\begin{array}{l}
{\underset{i} \prod} \ S\left(S^1 \wedge X_i, \Sigma^{a,b}E\right) \\
\hspace*{1cm}  \times \\
S\left(X, \Sigma^{a',b'}E\right)
\end{array}} \ar[r] \ar[d] & 
{\begin{array}{l}
S\left(X, \Sigma^{a,b}E\right) \\
\hspace*{1cm} \times \\
S\left(X, \Sigma^{a',b'}E\right)
\end{array}} \ar[r] \ar[d] & 
{\begin{array}{l}
{\underset{i} \prod} \ S\left(X_i, \Sigma^{a,b}E\right) \\
\hspace*{1cm}  \times \\
S\left(X, \Sigma^{a',b'}E\right)
\end{array}} \ar[d] \\
{\underset{i} \prod} \ S\left(S^1 \wedge X_i, \Sigma^{a+a',b+b'}E\right) 
\ar[r] & S\left(X, \Sigma^{a+a',b+b'}E\right) \ar[r] &
{\underset{i} \prod} \ S\left(X_i, \Sigma^{a+a',b+b'}E\right)}
\]
commutes in which the first vertical arrow is the map $(f, g) \mapsto h$ with 
$h(a \wedge b) = f(a \wedge b) \wedge g(b) \in \Sigma^{a+a',b+b'}E$.
The long exact sequence of homotopy groups corresponding to 
~\eqref{eqn:cofiber1} and \cite[Lemma~6.1.2]{Hovey1} now complete the
proof of the proposition.

The naturality of the Milnor exact sequence with respect to the map
of sequences $\{X_i\} \to \{Y_i\}$ is shown in \cite[Corollary~9.3.3]{BK}
and follows also from the diagram~\eqref{eqn:cofiber2}.
\end{proof}

As a consequence of Proposition~\ref{prop:Limit}, we get the following
form of excision for the motivic cobordism of ind-schemes.

\begin{cor}\label{cor:EXCN}
Let $f : U \to X$ be an {\'e}tale morphism of ind-schemes and let
$Z \subseteq X$ be a closed ind-subscheme such that the map
$f^{-1}(Z) \to Z$ is an isomorphism. Then the map
$f^*: MGL^{*,*}_Z(X) \to  MGL^{*,*}_Z(U)$ is an isomorphism.
\end{cor}
\begin{proof}
Translating the problem in the language of sequences of smooth schemes
and applying Lemma~\ref{lem:Cofib-qt} and Proposition~\ref{prop:Limit},
we get the following commutative diagram of short exact sequences. 
\begin{equation}\label{eqn:EXCN0}
\xymatrix@C.8pc{
0 \ar[r] &    
{\underset{i} {\varprojlim}^1} \ MGL^{*-1,*}_{Z_i}(X_i)
\ar[d]_{f^*} \ar[r] & MGL^{*,*}_{Z}(X)
\ar[r] \ar[d]^{f^*} & 
{\underset{i} \varprojlim} \ MGL^{*,*}_{Z_i}(X_i)
\ar[r] \ar[d]^{f^*} & 0 \\
0 \ar[r] & {\underset{i} {\varprojlim}^1} \ MGL^{*-1,*}_{Z_i}(U_i)     
\ar[r] & MGL^{*,*}_{Z}(U) \ar[r] &  
{\underset{i} \varprojlim} \ MGL^{*,*}_{Z_i}(U_i) \ar[r] & 0.}
\end{equation}
The vertical maps on the two ends are isomorphisms because
$(X,U) \mapsto MGL^{*,*}_{X \setminus U}(X)$ is a cohomology theory on the 
category of smooth pairs of schemes. We conclude that the middle vertical
map is also an isomorphism.
\end{proof}

\subsection{Thom and Chern structures on the cobordism of 
ind-schemes}\label{subsection:TCS}
In \cite{Panin1}, Panin shows that any ring cohomology theory on the
category $\Sm_k$ can in principle be equipped with three types of structures, 
namely, the Thom structure, the Chern structure and an orientation and 
these three structures are equivalent to each other.  
He further shows in \cite{Panin2} that each of these three structures on
a ring cohomology theory on $\Sm_k$ is equivalent to a trace structure.
This trace structure gives functorial push-forward maps on the cohomology
groups under a projective morphism. All these are well known in topology.
A ring cohomology theory with these structures is called an 
{\sl oriented cohomology theory}.
It is known ({\sl cf.} \cite[Example~3.8.7]{Panin1}) that the motivic
cobordism on $\Sm_k$ is an oriented cohomology theory.
The Thom and the Chern structures for this cohomology are given as 
follows. 

The structure of $T$-spectrum on $MGL$ gives unique maps
$\iota_n : T^n \wedge MGL_1 \to MGL_{n+1}$. Since $T \wedge pt = pt$, 
we get a canonical map 
\[
(MGL_1, T \wedge MGL_1, T^2 \wedge MGL_1, \cdots ) \to
(pt, T \wedge MGL_1, T \wedge MGL_2, \cdots )
\]
where $MGL_1 \to T \wedge pt = pt$ is the obvious map and 
$T^n \wedge MGL_1 \to T \wedge MGL_{n}$ is the map $\iota_{n-1} \wedge id_T$.
This gives a canonical map ${\tm} : Th\left(\sO_{\P^{\infty}_k}(-1)\right)
= MGL_1 \to \Sigma^{2,1}MGL$ and equivalently, a canonical element
${\tm} \in MGL^{2,1}\left(Th\left(\sO_{\P^{\infty}}(-1)\right)\right)$.
One defines the Chern class
\[
\xi = c_1\left(\sO_{\P^{\infty}}(-1)\right) = s^*\left({\tm}\right)
\in MGL^{2,1}\left(\P^{\infty}\right)
\]
where $s : \P^{\infty} \xrightarrow{i_{\sO_{\P^{\infty}}(-1)}} \sO_{\P^{\infty}}(-1) \to 
MGL_1$ is the composite map.

Our aim now is to suitably extend the notions of Thom and Chern structures
for the motivic cobordism of ind-schemes. This extension will in turn be used 
to establish some basic properties of the equivariant cobordism.

Let $X = colim_i \ X_i$ be an ind-scheme and let $\pi : L \to X$ be a line
bundle. We associate the Thom class $\tm(L) \in MGL^{2,1}_X(L)$ in the following
way. Let $\pi_i : L_i \to X_i$ be the restriction of $L$ to $X_i$ via 
the cofibration $g_i : X_i \to X$. The existence of the Thom 
structure on $\Sm_k$ yields an inverse system of
Thom classes $\{\tm(L_i) \in MGL^{2,1}_{X_i}(L_i)\}$. 
Set $\wh{\tm}(L) = {\underset{i} \varprojlim} \
\tm(L_i)$. Since $Th(L)$ is the colimit of a cofiber sequence,
we can use Proposition~\ref{prop:Limit} to get a short exact sequence
\[
0 \to {\underset{i} {\varprojlim}^1} \ MGL^{1,1}_{X_i}(L_i) \to
MGL^{2,1}_X(L) \to {\underset{i} \varprojlim} \ MGL^{2,1}_{X_i}(L_i) \to 0
\]
which is compatible with the action of $MGL^{*,*}_X(L)$.

On the other hand, the Thom isomorphism 
\[
\cup \ {\tm}(L_i) : MGL^{-1, 0}(X_i) 
\xrightarrow{\cong} MGL^{1,1}_{X_i}(L_i)
\] 
({\sl cf.} \cite[Definition~3.1]{Panin1}) implies that
each term $MGL^{1,1}_{X_i}(L_i)$ vanishes and hence there is a natural
isomorphism

\begin{equation}\label{eqn:Thom-cl}
MGL^{2,1}_X(L) \xrightarrow{\cong} {\underset{i} \varprojlim} \ 
MGL^{2,1}_{X_i}(L_i).
\end{equation}

We conclude that $\wh{\tm}(L)$ defines a unique class 
$\tm(L) \in MGL^{2,1}_X(L)$ whose restriction to $MGL^{2,1}_{X_i}(L_i)$
is the Thom class $\tm(L_i)$.
The element $\tm(L)$ is called the {\sl Thom class} of $L$. 
The {\sl Chern class} $c_1(L) \in MGL^{2,1}(X)$ is defined as the class 
$s^*_L(\tm(L))$, where $s_L: X \xrightarrow{i_L} L \xrightarrow{t_L} Th(L)$ is 
the composite map.
It is easy to see from the above construction of the Thom class (and hence
the Chern class) that given a map $f : Y \to X$ of ind-schemes and a line
bundle $L$ on $X$, one has $\tm\left(f^*(L)\right) = f^*\left(\tm(L)\right)$
and $c_1\left(f^*(L)\right) = f^*\left(c_1(L)\right)$. It follows in turn from
this that $c_1({\bf 1}_X) = 0$, where ${\bf 1}_X$ is the trivial
line bundle on $X$.

\subsubsection{Projective bundle formula and Chern classes for ind-schemes}
\label{subsubsection:PBFIS}
Let $X = colim_i \ X_i$ be an ind-scheme and let $E$ be a vector bundle of
rank $n$ over $X$ and let $\pi : \P(E) \to X$ be the associated projective
bundle. Let $\xi_E = c_1\left(\sO_{E}(-1)\right)$ be the Chern class of the
tautological line bundle. This yields a natural map
\begin{equation}\label{eqn:Proj-map}
\phi_X : MGL^{*,*}(X) \oplus \cdots \oplus MGL^{*,*}(X)
\to MGL^{*,*}\left(\P(E)\right) ;
\end{equation}
\[
\phi_X \left(\alpha_0, \alpha_1, \cdots , \alpha_{n-1}\right)
\ = \ \ \stackrel{n-1}{\underset{j = 0}\sum} \ \pi^*(\alpha_j) \cdot \xi^j.
\]

It follows from Proposition~\ref{prop:Limit} that there is a 
commutative diagram of short exact sequences

\begin{equation}\label{eqn:MAIN}
\xymatrix@C.8pc{
0 \ar[r] &    
{\underset{i} {\varprojlim}^1} \ \left(MGL^{*-1,*}(X_i)\right)^{\oplus n} 
\ar[r] \ar[d]_{{\rm lim}^1 \phi_i} & \left(MGL^{*,*}(X)\right)^{\oplus n}
\ar[r] \ar[d]_{\phi_X} & 
{\underset{i} \varprojlim} \ \left(MGL^{*,*}(X_i)\right)^{\oplus n} 
\ar[r] \ar[d]^{{\rm lim} \ \phi_i} & 0 \\
0 \ar[r] & {\underset{i} {\varprojlim}^1} \ MGL^{*-1,*}\left(\P(E_i)\right)      
\ar[r] & MGL^{*,*}\left(\P(E)\right) \ar[r] &  
{\underset{i} \varprojlim} \ MGL^{*,*}\left(\P(E_i)\right) \ar[r] & 0.}
\end{equation}

The left and the right vertical maps in ~\eqref{eqn:MAIN} are isomorphisms by 
\cite[Theorem~3.9]{Panin1} as each $X_i$ is a smooth scheme over $k$.
We conclude that the middle vertical map is an isomorphism.
This in particular yields the projective bundle formula for the vector
bundles on ind-schemes. If $E$ is a trivial bundle, then $\P(E)$ is the
pull-back of a map $\P(E) \to \P^{n-1}_k$ and hence $\xi^n_E = 0$, again
by \cite[Theorem~3.9]{Panin1}.

As in the case of schemes, the projective bundle formula gives rise to a
theory of Chern classes $\{c_0(E), c_1(E), \cdots , c_n(E)\}$
of a vector bundle $E$ of rank $n$ on an ind-scheme $X$ such that
$c_0(E) = 1$ and $c_i(E) \in MGL^{2i, i}(X)$ is the unique element such that
\begin{equation}\label{eqn:Chern-eqn}
\xi^n_E - \pi^*\left(c_1(E)\right)\cdot\xi^{n-1}_E + \cdots + (-1)^{n-1}
\pi^*\left(c_{n-1}(E)\right) \cdot \xi_E + (-1)^n\pi^*\left(c_n(E)\right) = 0
\end{equation}
in $MGL^{2n, n}\left(\P(E)\right)$. We set $c_i(E) = 0$ for $i > n$.
If $E$ is a vector bundle of rank one, then the map $\pi: \P(E) \to X$
is an isomorphism such that $\sO_E(-1) \cong E$ and hence the above
equation shows that $c_1(E)$ is same as the one defined before.
As another consequence of the projective bundle formula, we get the following
extension of \cite[Corollary~3.18]{Panin1} to ind-schemes.

\begin{cor}\label{cor:SESP}
Let $E_1$ and $E_2$ be two vector bundles on an ind-scheme $X$ and let
$\P(E_i) \xrightarrow{\iota_i} \P(E_1 \oplus E_2)$ be the inclusions of the
the projective bundles. Then there is a canonical short exact sequence
\[
0 \to MGL^{*,*}_{\P(E_1)}\left(\P(E_1 \oplus E_2)\right) \to
MGL^{*,*}\left(\P(E_1 \oplus E_2)\right) \xrightarrow{\iota^*_2}
MGL^{*,*}\left(\P(E_2)\right) \to 0.
\]
\end{cor} 
\begin{proof}
This is an immediate consequence of the projective bundle formula and 
the homotopy invariance and the proof is exactly like in the case of
schemes. We give a sketch. 

Let $U_i \xrightarrow{j_i} \P(E_1 \oplus E_2)$
denote the complement of $\P(E_i)$ for $i = 1,2$. Then there is a natural
projection $p_2 : U_1 \to \P(E_2)$ which is a vector bundle such that 
there is a factorization $\P(E_2) \xrightarrow{\eta_2} U_1 
\xrightarrow{j_1} \P(E_1 \oplus E_2)$ of $\iota_2$ with $\eta_2$ being
the zero-section. In particular, $\eta_2$ is an $\A^1$-weak equivalence.
Thus we can replace $MGL^{*,*}\left(U_1\right)$ with 
$MGL^{*,*}\left(\P(E_2)\right)$ and $j^*_1$ with $\iota^*_2$ in the long exact 
sequence ({\sl cf.} ~\eqref{eqn:LES})
\[
\cdots \to MGL^{*,*}_{\P(E_1)}\left(\P(E_1 \oplus E_2)\right) \to
MGL^{*,*}\left(\P(E_1 \oplus E_2)\right) \xrightarrow{j_1^*} 
MGL^{*,*}\left(U_1\right) \to \cdots .
\]

Hence one only needs to show that each map $\iota^*_i$ is surjective.
But this follows from the projective bundle formula ~\eqref{eqn:Proj-map}
and by noting that the tautological bundle on $\P(E_1 \oplus E_2)$
restricts to the tautological bundle on each $\P(E_i)$.
\end{proof}

\subsubsection{Cartan formula for Chern classes}\label{subsubsection:CF}
It follows directly from the above definitions that the Chern classes of
vector bundles on ind-schemes satisfy all the standard properties of a
Chern class theory except the Cartan formula. To establish this formula,
we need to prove some intermediate steps, following the approach
of Panin in the case of schemes.

\begin{lem}\label{lem:split}
Let $\alpha : E \to F$ be an epimorphism of vector bundles on an ind-scheme
$X$. Then there exists a morphism of ind-schemes $\pi : Y \to X$ which
is a $\A^1$-weak equivalence and such that the epimorphism $\pi^*(\alpha)$
splits.
\end{lem}
\begin{proof}
This was proven by Panin in the case of schemes. We show that
this approach also works for ind-schemes. So let $\alpha_i : E_i \to F_i$
be the epimorphism of vector bundles on the scheme $X_i$ where
$X = colim_i \ X_i$. In view of Lemma~\ref{lem:Cof-weq}, all we need to do
is to find a sequence of smooth schemes $\{Y_0 \xrightarrow{g_0} Y_1 
\xrightarrow{g_1} \cdots \}$ and map of sequences 
$\pi:\{Y_i\} \to \{X_i\}$ such that the following hold. 
\begin{enumerate}
\item
Each $\pi_i : Y_i \to X_i$ is an $\A^1$-weak equivalence,
\item
There exists a splitting $\beta_i : \pi_i^*(F_i) \to \pi_i^*(E_i)$
of $\pi_i^*(\alpha_i)$, and
\item
For each $i \ge 0$, $g_i^*(\beta_{i+1}) = \beta_i$.
\end{enumerate}
Given this datum, we take $Y = colim_i \ Y_i$ and $\beta = colim_i \
\beta_i$. Then $\beta : \pi^*(F) \to \pi^*(E)$ yields a splitting of
$\pi^*(\alpha)$ on the ind-scheme $Y$.

For each $i \ge 0$, consider the vector bundle $G_i = \sHom_{X_i}(F_i, E_i)$
whose fiber at any point $x \to X_i$ is the space of linear maps
$(F_i)_x \to (E_i)_x$ of $k(x)$-vector spaces.   
It follows that $f^*_i(G_{i+1}) = G_i$ where $f_i: X_i \to X_{i+1}$ is the
given cofibration. Setting $H_i = \sHom_{X_i}(F_i, F_i)$, we have the compatible 
system of maps $\alpha^*_i: G_i \to H_i$ given by $v \mapsto \alpha_i \circ v$. 
Let $Y_i \xrightarrow{\iota_i}  G_i$ be the subscheme such that the front and 
the back faces of the diagram

\begin{equation}\label{eqn:split1}
\xymatrix@C.8pc{
Y_i \ar[rr]^{\pi_i} \ar[dr]^{g_i} \ar[dd]_{\iota_i} & & X_i \ar[dr]^{f_i} 
\ar[dd]_<<<<<<{\theta_i} & \\  
& Y_{i+1} \ar[dd] \ar[rr]_{\pi_{i+1}} & & X_{i+1} \ar[dd]^{\theta_{i+1}} \\
G_i \ar[dr] \ar[rr] & & H_i \ar[dr] & \\
& G_{i+1} \ar[rr]_{\alpha^*_{i+1}} & & H_{i+1}}
\end{equation}
are Cartesian where $\theta_i \to H_i$ is the section of the projection
$H_i \to X_i$ corresponding to the identity map of $F_i$. Notice 
that all the maps in this diagram are cofibrations.
Since $f^*_i(G_{i+1}) = G_i$, one easily checks that $f^*_i(Y_{i+1}) = Y_i$.
The pairs $(\pi_i, \iota_i)$ uniquely define the maps $\beta_i :
\pi^*_i(F_i) \to \pi^*_i(E_i)$ such that $\pi^*(\alpha_i) \circ \beta_i =
1_{\pi^*(F_i)}$. Moreover, $f^*_i(Y_{i+1}) = Y_i$ is equivalent to saying that
$g^*_i(\beta_{i+1}) = \beta_i$. Finally, it is well known that each $\pi_i$
is an affine bundle and hence an $\A^1$-weak equivalence.
\end{proof}

\begin{lem}[Splitting principle]\label{lem:Sprinciple}
Let $E \to X$ be a vector bundle of rank $n$ over an ind-scheme $X$.
Then there exists a morphism of ind-schemes $\pi: Y \to X$ such that
\begin{enumerate}
\item
$\pi^*(E)$ is a direct sum of line bundles and
\item
for any morphism of ind-schemes $f : X' \to X$, the map
$MGL^{*,*}(X') \to MGL^{*,*}(Y \times_X X')$ is split injective.
\end{enumerate}
\end{lem} 
\begin{proof}
In view of the projective bundle formula for ind-schemes in
\S~\ref{subsubsection:PBFIS}, Lemma~\ref{lem:split} and the known standard
techniques in case of schemes, we only need to show that given a vector
bundle $E$ over $X$, one has a short exact sequence
\[
0 \to \sO_E(-1) \to p^*(E) \to E' \to 0
\]
of vector bundles on the projective bundle $p : \P(E) \to X$. But this is well 
known for schemes and moreover the maps $ \sO_{E_i}(-1) \to p^*_i(E_i)$ are
canonical and compatible with the cofibrations $f_i : X_i \to X_{i+1}$
since $E_i = f^*_i(E_{i+1})$. 
\end{proof} 

Using the splitting principle, we can extend the theory Chern classes of
vector bundles on ind-schemes as follows.

\begin{prop}\label{prop:CCT}
Given an ind-scheme $X$ and a vector bundle $E \to X$, there are Chern
classes $c_i(E) \in MGL^{2i, i}(X)$ such that
\begin{enumerate}
\item
$c_0(E) = 1, c_i(E) = 0$ for $i > {\rm rank}(E)$ such that $c_1(E)$ coincides
with the Chern class $c(E)$ as in \S~\ref{subsection:TCS} if $E$ is a
line bundle.
\item
$c_i(E) = c_i(E')$ if $E \cong E'$ and $f^*\left(c_i(E)\right) = 
c_i\left(f^*(E)\right)$ for a map of ind-schemes $f : Y \to X$.
\item
$c(E) = c(E')\cdot c(E'')$ if there is a short exact sequence
$0 \to E' \to E \to E'' \to 0$ of vector bundles, where 
$c(E) = 1 + c_1(E)t + c_2(E) t^2 + \cdots $ is the Chern polynomial.
\end{enumerate}
\end{prop}
\begin{proof}
We only need to show the Cartan formula $c(E) = c(E')\cdot c(E'')$ for which
we can use the second property and Lemma~\ref{lem:Sprinciple} to reduce to the 
case when 
$E = F \oplus L$, where $L$ is a line bundle and $F$ is a vector bundle of
rank $n$. 

Using the definition of the Chern classes, it suffices to show that
\begin{equation}\label{eqn:CCT0}
\left(\xi - c_1(L)\right) 
\left(\xi^n - c_1(F)\xi^{n-1} + \cdots +(-1)^{n}c_n(F)\right) = 0
\end{equation}
in $MGL^{2n+2, n+1}\left(\P(E)\right)$, where $\xi = c_1\left(\sO_E(-1)\right)$.

Set $\alpha(F) = \xi^n - c_1(F)\xi^{n-1} + \cdots +(-1)^{n}c_n(F)$.
Since the tautological line bundle on $\P(E)$ restricts to the
tautological line bundles on $\P(L)$ and $\P(F)$, it follows from
~\eqref{eqn:Chern-eqn} and Corollary~\ref{cor:SESP}
that $\xi - c_1(L) \in MGL^{2,1}_{\P(F)}\left(\P(E)\right)$ and
$\alpha(F) \in MGL^{2n,n}_{\P(L)}\left(\P(E)\right)$. In particular, we
conclude from ~\eqref{eqn:prod*} that the class
$\left(\xi - c_1(L)\right)\alpha(F)$ in $MGL^{2n+2, n+1}\left(\P(E)\right)$
is in the image of the map
\[
MGL^{2n+2, n+1}_{\P(F) \cap \P(L)}\left(\P(E)\right) \to 
MGL^{2n+2, n+1}\left(\P(E)\right).
\]
The desired assertion ~\eqref{eqn:CCT0} now
follows by observing that $\P(F) \cap \P(L) = \0$.
\end{proof}

\subsubsection{Thom classes and Thom isomorphism for ind-schemes}
\label{subsubsection:TCTI}
Using the theory of Chern classes, we now define the Thom classes of
vector bundles on ind-schemes and prove the Thom isomorphism.
More precisely, we prove the following.

\begin{prop}\label{prop:TIso}
Given an ind-scheme $X$ and a vector bundle $p: E \to X$ of rank $n$,
there exists a class $\tm(E) \in MGL^{2n,n}_X(E)$ such that \\
$(1) \ \tm(E) = \tm(F)$ if $E \cong F$; \\
$(2)\ f^*\left(\tm(E)\right) = \tm\left(f^*(E)\right)$ for a 
morphism $f : Y \to X$ of ind-schemes; \\
$(3)$ \ the map $\tm^E_X : MGL^{a,b}(X) \to MGL^{a + 2n, b + n}_{X}(E)$
given by $\alpha \mapsto p^*(\alpha) \cdot \tm(E)$ ({\sl cf.} 
~\eqref{eqn:prod*}) is an isomorphism; \\
$(4)$ \ given the projections $q_i : E_1 \oplus E_2 \to E_i$,
($i = 1,2$), one has
\[
q^*_1\left(\tm(E_1)\right) \cdot q^*_2\left(\tm(E_2)\right)
= \tm(E_1 \oplus E_2) ;
\]
$(5)$ \ for a line bundle $L$ on $X$, the class $\tm(L)$ coincides with
the Thom class defined in \S ~\ref{subsection:TCS}; \\
$(6)$ \ for a line bundle $L$ on $X$,  one has $\tm^L_X(1) = \tm(L)$ and 
$s^*_L \circ \tm^L_X(a) = c_1(L) \cdot a$ for every $a \in MGL^{*,*}(X)$
where $s_L: X \xrightarrow{i_L} L \xrightarrow{t_L} Th(L)$ is 
the composite map.
\end{prop} 
\begin{proof}
Let $F = E \oplus {\bf 1}_X$ and consider the projective bundle
$\pi : \P(F) \to X$. As in the case of schemes, there is a short exact sequence
$0 \to {\bf 1}_{\P(F)} \to \pi^*(E) \otimes \sO_F(1)  \to G \to 0$ of vector 
bundles on $\P(F)$. It follows from Proposition~\ref{prop:CCT} and 
Corollary~\ref{cor:SESP} that
$c_n\left(\pi^*(E) \otimes \sO_F(1)\right) \in MGL^{2n,n}_{X}\left(\P(F)\right)$.
On the other hand, we can apply Corollary~\ref{cor:EXCN} to the
inclusions $X \xrightarrow{i_E} E \xrightarrow{e_E} \P(F)$ to see that the
natural map $e^*_E: MGL^{*,*}_X\left(\P(F)\right) \to MGL^{*,*}_X(E)$
is an isomorphism. This gives us a unique element (the {\sl Thom class})
$\tm(E) = c_n\left(\pi^*(E) \otimes \sO_F(1)\right)$ in $MGL^{2n,n}_X(E)$.
The first and the second properties of these Thom classes follow from their
construction and the second point of Proposition~\ref{prop:CCT}.

To prove the third property, we first notice that the above construction of
the Thom class coincides with that in \cite{Panin1} if $X$ is a smooth
scheme. Moreover, since each $E_i$ is the restriction of $E$ on $X_i$,
it follows from the second property of the Thom classes that 
$\tm(E_i) = {\tm(E)}|_{E_i}$.
We can thus apply Lemma~\ref{lem:Cofib-qt} and Proposition~\ref{prop:Limit}
to get the following commutative diagram of short exact sequences. 
\begin{equation}\label{eqn:EXCN0}
\xymatrix@C.8pc{
0 \ar[r] &    
{\underset{i} {\varprojlim}^1} \ MGL^{a-1,b}(X_i)
\ar[d]_{\tm^{E_i}_{X_i}} \ar[r] & MGL^{a,b}(X)
\ar[r] \ar[d]^{\tm^{E}_X} & 
{\underset{i} \varprojlim} \ MGL^{a,b}(X_i)
\ar[r] \ar[d]^{\tm^{E_i}_{X_i}} & 0 \\
0 \ar[r] & {\underset{i} {\varprojlim}^1} \ MGL^{a+2n-1,b+n}_{X_i}(E_i)     
\ar[r] & MGL^{a+2n,b+n}_{X}(E) \ar[r] &  
{\underset{i} \varprojlim} \ MGL^{a+2n,b+n}_{X_i}(E_i) \ar[r] & 0.}
\end{equation}

The vertical maps on the two ends are isomorphisms by the Thom isomorphism
for the motivic cobordism of smooth schemes. It follows that the middle 
vertical map is an isomorphism too. This proves the property (3).
The property (4) follows directly from the Cartan formula in
Proposition~\ref{prop:CCT} and Corollary~\ref{cor:SESP} and the proof
works exactly like in the case of schemes 
({\sl cf.} \cite[Theorem~3.35]{Panin1}). We now prove the last property
to complete the proof of the proposition.

Let $p : L \to X$ be a line bundle and let us temporarily denote the Thom class
as defined in \S ~\ref{subsection:TCS} by $\wh{\tm}(L)$. We now recall from
~\eqref{eqn:Thom-cl} that the natural map
$MGL^{2,1}_{X}(L) \to {\underset{i} \varprojlim} \ MGL^{2,1}_{X_i}(L_i)$
is an isomorphism and by the construction, $\wh{\tm}(L) \in MGL^{2,1}_{X}(L)$
is the unique class which restricts to $\tm(L_i) \in MGL^{2,1}_{X_i}(L_i)$
for each $i$. By the property (2) of the proposition,
the class $\tm(L)$ also restricts to $\tm(L_i) \in MGL^{2,1}_{X_i}(L_i)$
for each $i$. It follows from the isomorphism ~\eqref{eqn:Thom-cl} that
we must have $\tm(L) = \wh{\tm}(L)$. 

The fact that $\tm^L_X(1) = \tm(L)$
now follows because the the horizontal maps in the commutative diagram
\begin{equation}\label{eqn:EXCN0*1} 
\xymatrix@C.8pc{
MGL^{0,0}(X) \ar[r] \ar[d]_{\tm^L_X} & 
{\underset{i}\varprojlim} \ MGL^{0,0}(X_i)     
\ar[d]^{\tm^{L_i}_{X_i}} \\
MGL^{2,1}_{X}(L) \ar[r] &  
{\underset{i} \varprojlim} \ MGL^{2,1}_{X_i}(L_i)}
\end{equation}
are isomorphisms and one knows from \cite[Theorem~3.35]{Panin1} that
$\tm^{L_i}_{X_i}(1) = \tm(L_i)$ for each $i \ge 0$.

To prove the last property, one first observes that $p^*$ is a ring
isomorphism and it follows from ~\eqref{eqn:prod*} that the map
$MGL^{*.*}_X(E) \xrightarrow{t^*_E} MGL^{*,*}(E)$ is $MGL^{*,*}(E)$-linear for any 
vector bundle $E$ on $X$. Now the desired assertion follows from property (5) 
and the definition of the first Chern class of line bundles in 
\S ~\ref{subsection:TCS}. 
\end{proof}

\section{Gysin map for motivic cobordism of ind-schemes}
\label{section:GM}
In this section, we construct the Gysin maps 
$\iota_* : MGL^{*,*}(Y) \to MGL^{*,*}(X)$
for a given closed embedding of ind-schemes $\iota: Y \inj X$.
For a closed immersion $Y \inj X$ of smooth schemes, let
$N_X(Y)$ denote the normal bundle of $Y$ in $X$.
Recall from \cite[Definition~2.1]{Panin2} that a commutative square of
smooth schemes 
\begin{equation}\label{eqn:trans}
\xymatrix@C.9pc{
Y' \ar[r]^{g} \ar[d]_{\iota'} & Y \ar[d]^{\iota} \\
X' \ar[r]_{f} & X}
\end{equation}
is called transverse if it is Cartesian, the map $\iota$ is a closed
immersion  and the map $N_{X'}(Y') \to g^*\left(N_{X}(Y)\right)$ is an
isomorphism.

\begin{defn}\label{defn:Strict}
We shall say that a closed immersion of ind-schemes 
$\iota: Y \inj X$ is {\sl strict} if the square

\begin{equation}\label{eqn:Strict1}
\xymatrix@C.9pc{
Y_i \ar[r]^{g_i} \ar[d]_{\iota_i} & Y_{i+1} \ar[d]^{\iota_{i+1}} \\
X_i \ar[r]_{f_i} & X_{i+1}}
\end{equation}
is transverse for each $i \ge 0$. 
An open immersion $j : Y \inj X$ of ind-schemes is strict if the square 
~\eqref{eqn:Strict1} is Cartesian.
If ~\eqref{eqn:trans} is commutative diagram of ind-schemes, then we
shall say that this square is {\sl transverse} if it is so at
each level $i \ge 0$ and $\iota$ is a strict embedding. 
\end{defn}

If $\iota : Y \inj X$ is a strict closed immersion, then the normal bundles
$\{N_{X_i}(Y_i)\}$ define a vector bundle $N_X(Y)$ on the ind-scheme $Y$ 
({\sl cf.} \S ~\ref{subsubsection:VEC}) of rank $d$ where 
$d = \codim_{X_i}(Y_i)$ is called the codimension of $Y$ in $X$.
This vector bundle will be called the normal bundle of $Y$ in $X$.
It is easy to check that if ~\eqref{eqn:trans} is a transverse square
of ind-schemes, then $\iota': Y' \inj X'$ is a strict
closed embedding and the map $N_{X'}(Y')  \to g^*\left(N_{X}(Y)\right)$ is an
isomorphism. Note also that if ~\eqref{eqn:trans} is a Cartesian 
square of ind-schemes such that $\iota$ is a strict embedding and
$f$ is an open immersion (not necessarily strict), then it is 
transverse. If ~\eqref{eqn:trans} is a transverse 
square of ind-schemes such that $\iota$ and $f$ are strict closed embeddings,
then we shall say that $X'$ and $Y$ intersect transversely in $X$.
In such a case, the sequence $\{Y_i \cap X'_i\}$ defines a 
strict closed ind-subscheme $Y \cap X'$ of $X$.

It is also easy to see that if $\iota: Y \inj X$ is a strict closed
embedding of ind-schemes with normal bundle $N$, then $Y \times \{0\} \to
X \times \A^1$ is also a strict closed embedding with the normal bundle
${\bf 1}_Y \oplus N$.
Let $M'$ denote the blow-up of $X \times \A^1$ along $Y \times \{0\}$ and let
$M = M' \setminus Bl_{Y\times \{0\}}(X \times \{0\}) = M \setminus Bl_Y(X)$.
Then $M'$ is an ind-scheme $\{M'_0 \xrightarrow{h_0} M'_1 \xrightarrow{h_1}
\cdots \}$ with the open ind-subscheme 
$M = \{M_0 \xrightarrow{h_0} M_1 \xrightarrow{h_1} \cdots \}$
({\sl cf.} \cite[Corollary~II.7.15]{Hart}). We obtain the following 
{\sl deformation to normal cone} diagram of ind-schemes.

\begin{equation}\label{eqn:DNC}
\xymatrix@C1.5pc{
Y \ar[r]_<<<<<<{k_0} \ar[d]_{i_N} & Y \times \A^1 \ar[d]^{F} \ar@/_{.5cm}/[l]_{p}
\ar@/^{.5cm}/[r]^{p} & Y \ar[l]^>>>{k_1}
\ar[d]^{\iota} \\
N \ar[d] \ar[r]_{j_0} & M \ar[d]^{j} & X \ar[l]^{j_1} \ar@{=}[d] \\
\P({\bf 1} \oplus N) \ar[r]_>>>>{j'_0} & M' & X \ar[l]^{j'_1}}
\end{equation}
where the two top squares and the bottom left square are transverse. 
This induces the maps of motivic spaces
\begin{equation}\label{eqn:DNC1}
Th(N) \xrightarrow{\ov{j}_0} M/{\left(M \setminus (Y \times \A^1)\right)}
\xleftarrow{\ov{j}_1} X/{\left(X \setminus Y\right)}.
\end{equation}
It follows from \cite[Theorem~3.2.23]{MV} and Lemma~\ref{lem:Cof-weq}
that the maps $\ov{j}_0$ and $\ov{j}_1$ are $\A^1$-weak equivalences.
In particular, we get a functorial $\A^1$-weak equivalence
\begin{equation}\label{eqn:DNC2}
t_{X,Y} = (\ov{j}_0)^{-1} \circ \ov{j}_1:  X/{\left(X \setminus Y\right)} \to
Th(N). 
\end{equation}

Let $G : V = M \setminus (Y \times \A^1) \inj M$ and 
$G': V' = M' \setminus (Y \times \A^1) \inj M'$ denote the open inclusions.
\begin{lem}\label{lem:extra}
The map
\[
(j'^*_0, G'^*) : MGL^{*,*}(M') \to MGL^{*,*}\left(\P({\bf 1} \oplus N)\right)
\oplus MGL^{*,*}(V')
\]
is injective.
\end{lem}
\begin{proof}
We first consider the commutative diagram
\begin{equation}\label{eqn:extra0}
\xymatrix@C1pc{
MGL^{*,*}_{Y \times \A^1}(M') \ar[r]^<<<{j'^*_0} \ar[d] & 
MGL^{*,*}_Y\left(\P({\bf 1} \oplus N)\right) \ar[d] \\
MGL^{*,*}_{Y \times \A^1}(M) \ar[r]_{j^*_0} & MGL^{*,*}_{Y}(N).}
\end{equation}  

The two vertical maps are isomorphisms by Corollary~\ref{cor:EXCN}
and we have seen above that the bottom horizontal map is an isomorphism.
It follows that $j'^*_0$ is an isomorphism.

The lemma now follows by using the commutative diagram
\begin{equation}\label{eqn:extra1}
\xymatrix@C1pc{
MGL^{*,*}_{Y \times \A^1}(M') \ar[r]^<<<{j'^*_0} \ar[d] & 
MGL^{*,*}_Y\left(\P({\bf 1} \oplus N)\right) \ar[d] \\
MGL^{*,*}(M') \ar[r]_<<<<{j'^*_0} & MGL^{*,*}\left(\P({\bf 1} \oplus N)\right),}
\end{equation}
Corollary~\ref{cor:SESP} and ~\eqref{eqn:LES}.
\end{proof}

\begin{defn}\label{defn:Gmap}
Given a strict closed embedding $\iota : Y \inj X$ of ind-schemes of codimension
$d$, we define the {\sl Gysin map} $\iota_* : MGL^{a,b}(Y) \to 
MGL^{a+2d, b+d}(X)$ as the composite
\begin{equation}\label{eqn:Gmap0}
MGL^{a, b}(Y) \xrightarrow{\tm^N_Y} MGL^{a+2d, b+d}_Y(N) 
\xrightarrow{t^*_{X,Y}} MGL^{a+2d, b+d}_Y(X) \xrightarrow{v^*_{X, Y}}
MGL^{a+2d, b+d}(X)
\end{equation}
where $X \xrightarrow{v_{X,Y}} X/{\left(X \setminus Y\right)}$ is the quotient
map. 
\end{defn}

Since all the maps in this sequence are $MGL^{*,*}(X)$-linear, we
see that the Gysin map is $MGL^{*,*}(X)$-linear. In particular, we have the
projection formula
\begin{equation}\label{eqn:Proj-form}
\iota_*\left(\iota^*(a)\right) = a \cdot \iota_*(1).
\end{equation}

\begin{prop}\label{prop:Gmap-Prop}
The Gysin maps $\iota_*: MGL^{*,*}(Y) \to 
MGL^{*, *}(X)$ satisfy the following functoriality properties.
\begin{enumerate}
\item
$Base \ Change : $ \ If ~\eqref{eqn:trans} is a transverse square of
ind-schemes, then the diagram

\begin{equation}\label{eqn:Gmap*0}
\xymatrix@C.9pc{
MGL^{*,*}(Y) \ar[d]_{\iota_*} \ar[r]^{g^*} & MGL^{*,*}(Y') \ar[d]^{\iota'_*} \\
MGL^{*,*}(X) \ar[r]_{f^*} & MGL^{*,*}(X')}
\end{equation}
commutes.

\item
$Identity :$ \ ${\rm id}_* = {\rm id}$.
\item
$Gysin \ exact \ sequence :$ For a strict closed embedding $\iota: Y \inj X$
of codimension $d$ with the complement $j : U \inj X$, the sequence
\[
\cdots \to MGL^{a-1, b}(U) \xrightarrow{\partial} MGL^{a-2d, b-d}(Y) 
\xrightarrow{\iota_*} MGL^{a, b}(X) \xrightarrow{j^*}
MGL^{a, b}(U) \xrightarrow{\partial} \cdots 
\]
is exact.
\item
$Section \ of \ a \ projective \ bundle :$ \
If $E$ is a rank $n$ vector bundle on $Y$ and $s : Y \to \P({\bf 1} \oplus E)$
is the zero-section of the projective bundle $p : \P({\bf 1} \oplus E) \to Y$,
then $s_* = (-) \cdot ({\tm}(E)) \circ p^*$. 
\item
$Smooth \ divisor :$ \ If $\iota : D \inj X$ is a strict embedding of
a smooth divisor, then $\iota_*(1) = c_1\left(L(D)\right)$.
\item
$Functoriality : $ \ For strict closed embeddings $Z \stackrel{\iota'}{\inj}
Y \stackrel{\iota}{\inj} X$ of ind-schemes, one has
$\iota_* \circ \iota'_* = (\iota \circ \iota')_*$.
\end{enumerate}
\end{prop}
\begin{proof}
Since the isomorphism $t^*_{X,Y}$ in ~\eqref{eqn:DNC2} and the map
$v^*_{X,Y}$ in ~\eqref{eqn:Gmap0} are functorial, the commutativity of
~\eqref{eqn:Gmap*0} is a direct consequence of the functoriality of
the Thom classes from Proposition~\ref{prop:TIso} and the transversality 
condition. The property (2) follows directly from the definition and
the property (3) is a direct consequence of the above definition,
the Thom isomorphism and ~\eqref{eqn:LES}. The property (4) follows 
immediately from the definition of the Gysin map and that of the Thom class
in Proposition~\ref{prop:TIso}, using Corollary~\ref{cor:EXCN}.

To prove property (5), we consider the diagram~\ref{eqn:DNC} and set
$F' = j \circ F$ and let $s : D \to \P({\bf 1} \oplus N)$ denote the
closed embedding. The transversality of the squares and property (1)
yield a commutative diagram 

\begin{equation}\label{eqn:Gmap*1}
\xymatrix@C1.2pc{
MGL^{*,*}(D) \ar[d]_{s_*} & \m(D \times \A^1) \ar[l]_<<<<{k^*_0} \ar[r]^>>>{k^*_1}
\ar[d]^{F'_*} & \m(D) \ar[d]^{\iota_*} \\
\m\left(\P({\bf 1} \oplus N)\right) & 
\m(M') \ar[l]^>>>>>{j'^*_0} \ar[r]_{j'^*_1} & \m(X).}
\end{equation}
Since $k^*_1$ is an isomorphism, it follows from the functoriality of the
Chern classes ({\sl cf.} Proposition~\ref{prop:CCT}) that it is enough to
show that 
\begin{equation}\label{eqn:Gmap*2}
F'_*(1) = c_1(L'), \ {\rm where} \ L' = L\left(D \times \A^1\right).
\end{equation}
 
Using property (4) and the definition of the Thom class, we get
\[
\begin{array}{lll}
j'^*_0\left(c_1(L')\right) & = & 
c_1\left(\sO_{{\bf 1} \oplus N}(1) \otimes p^*(N)\right) \\
& = & s_*(1) \\
& = & s_* \circ k^*_0 (1) \\
& = & {j'^*_0} \circ F'_*(1).
\end{array}
\]
Since the elements $c_1(L')$ and $F'_*(1)$ vanish in $\mg(V')$, we conclude
from Lemma~\ref{lem:extra} that $F'_*(1) = c_1(L')$ which proves
~\eqref{eqn:Gmap*2} and hence property (5).

We prove property (6) in several steps by imitating the proof for the
case of smooth schemes.
In the first step, we show that if $X$ is an ind-scheme and 
$\{D_j\}_{1 \le j \le n}$ is a collection of strict smooth divisors which 
intersect transversely
in $X$ with $Y = \stackrel{n}{\underset{j = 1}\cap} \ D_j$, then
\begin{equation}\label{eqn:Gmap*3}  
\iota_*(1) = c_n\left(\stackrel{n}{\underset{j = 1}\oplus} \ L_j\right)
\end{equation}
where $\iota : Y \inj X$ is the strict closed embedding and
$L_j = L(D_j)$. 

To prove this, we set $N = \stackrel{n}{\underset{j = 1}\oplus} \ \iota^*(L_j)$
and consider again the deformation diagram
\begin{equation}\label{eqn:Gmap*4}
\xymatrix@C1pc{
Y \ar[r]^<<<<<<{k_0} \ar[d]_{s} & Y \times \A^1 \ar[d]^{F'} & Y \ar[l]_>>>{k_1}
\ar[d]^{\iota} \\
\P({\bf 1} \oplus N) \ar[r]_>>>>{j'_0} & M' & X \ar[l]^{j'_1}.}
\end{equation}

Let $p: \P\left({\bf 1} \oplus N\right) \to Y$ be the projection map.
It is easy to check that the proper transform $M'_j$ of $D_j \times \A^1$
in $M'$ are all strict smooth divisors which intersect transversely 
with $Y \times \A^1 = \stackrel{n}{\underset{j = 1}\cap} \ M'_j$. Moreover, 
each $M'_j$ intersects $\P\left({\bf 1} \oplus N\right)$ transversely such that
the intersection $P_j = \P\left({\bf 1} \oplus N_j\right)$ is
a smooth divisor in $\P\left({\bf 1} \oplus N\right)$, where
$N_j$ is the direct sum of all line bundles on $Y$ except $\iota^*(L_j)$.
The line bundle $L(P_j)$ is isomorphic to the line bundle 
$p^* \circ \iota^*(L_j) \otimes \sO_{{\bf 1} \oplus N}(1)$.
It follows from Proposition~\ref{prop:CCT} that
$c_n\left(p^*(N) \otimes  \sO_{{\bf 1} \oplus N}(1)\right)$ is the
product of the Chern classes 
$c_1\left(p^* \circ \iota^*(L_j) \otimes  \sO_{{\bf 1} \oplus N}(1)\right)$.

The line bundle $L'_j = L(M'_j)$ restricts to $L(P_j)$ over 
$\P\left({\bf 1} \oplus N\right)$ and hence is isomorphic to the
line bundle $p^* \circ \iota^*(L_j) \otimes \sO_{{\bf 1} \oplus N}(1)$.
We also have $j'^*_1(L'_j) \cong L_j$. Therefore, we have the relation
$j'^*_0\left(c_1(L'_j)\right) = 
c_1\left(p^* \circ \iota^*(L_j) \otimes  \sO_{{\bf 1} \oplus N}(1)\right)$
in $MGL^{*,*}\left(\P\left({\bf 1} \oplus N\right)\right)$ and the relation
$j'^*_1\left(c_1(L'_j)\right) = c_1(L_j)$ in $\m(X)$. 
It follows from property (4) that 
\begin{equation}\label{eqn:Gmap*5}
s_*(1) \ = \ \stackrel{n}{\underset{j = 1}\cup} \ 
c_1\left(p^* \circ \iota^*(L_j) \otimes  \sO_{{\bf 1} \oplus N}(1)\right).
\end{equation}

Next we show that
\begin{equation}\label{eqn:Gmap*6}
F'_*(1) \ = \ \stackrel{n}{\underset{j = 1}\cup} \ 
c_1\left(L'_j\right).
\end{equation}

To prove this, we first observe that $F'_*(1)$ vanishes in $\m(V')$
as follows from ~\eqref{eqn:LES}.
On the other hand, 
$c_1\left(L'_j\right)$ vanishes in $\m\left(M' \setminus M'_j\right)$
and hence comes from $\m_{M'_j}(M')$ by ~\eqref{eqn:LES} again.
It follows from ~\eqref{eqn:prod*} that 
$\stackrel{n}{\underset{j = 1}\cup} \ c_1\left(L'_j\right)$ comes from
$\m_{Y \times \A^1}(M')$ and hence vanishes in $\m(V')$.
We now compute
\[
\begin{array}{lll}
j'^*_0\left(\stackrel{n}{\underset{j = 1}\cup} \ 
c_1\left(L'_j\right)\right) & = & 
\stackrel{n}{\underset{j = 1}\cup} \ \left(j'^*_0 (L'_j)\right) \\
& = & \stackrel{n}{\underset{j = 1}\cup} \ 
c_1\left(p^* \circ \iota^*(L_j) \otimes  \sO_{{\bf 1} \oplus N}(1)\right) \\
& =^{\dagger} & s_*(1) \\
& = & s_* \circ k^*_0(1) \\
& =^{\dagger \dagger} & j'^*_0 \circ F'_*(1),
\end{array}
\]
where the equalities $\dagger$ and $\dagger \dagger$ follow from
~\eqref{eqn:Gmap*5} and property (1) respectively. The relation
~\eqref{eqn:Gmap*6} now follows from Lemma~\ref{lem:extra}.

Finally, we have
\[
\begin{array}{lll}
c_n\left(\stackrel{n}{\underset{j = 1}\oplus} \ L_j\right) & = &
\stackrel{n}{\underset{j = 1}\cup} \ c_1\left(L_j\right) \\
& = &
j'^*_1\left(\stackrel{n}{\underset{j = 1}\cup} \ c_1\left(L'_j\right)\right) \\
& = & j^*_1 \circ F'_*(1) \\
& =^{\dagger} & \iota_* \circ k^*_1 (1) \\
& = & \iota_*(1)
\end{array}
\]
where the equality $\dagger$ follows from property (1).
This completes the proof of ~\eqref{eqn:Gmap*3}.

The second step is to show the following. Let $Y$ be an ind-scheme and
let $0 \to N \to M \to F \to 0$ be a short exact sequence of vector bundles
on $Y$ such that ${\rm rank}(F) = d$. Let 
$\iota: \P\left({\bf 1} \oplus N\right)
\inj \P\left({\bf 1} \oplus M\right)$ be the inclusion map and let
$p : \P\left({\bf 1} \oplus M\right) \to Y$ be the projection. Then
one has
\begin{equation}\label{eqn:Gmap*7} 
\iota_*(1) = c_d\left(p^*F \otimes \sO_{{\bf 1} \oplus M}(1)\right).
\end{equation}

Using the splitting principle (Lemma~\ref{lem:Sprinciple}) and property (1),
we can assume that $F = \stackrel{j= d}{\underset{j =1}\oplus}  L_j$
is a direct sum of line bundles, in order to prove ~\eqref{eqn:Gmap*7}.
Let $M_j$ be the preimage of the
direct sum of all summands of $F$ except $L_j$ and set $D_j = 
\P\left({\bf 1} \oplus M_j\right)$. Then $\{D_j\}$ is a collection of
strict smooth divisors on $\P\left({\bf 1} \oplus M\right)$ which intersect
transversely with the intersection $\P\left({\bf 1} \oplus N\right)$.
Moreover, the line bundle $L(D_j)$ is isomorphic to the line bundle
$p^*(L_j) \otimes \sO_{{\bf 1} \oplus M}(1)$. It follows from 
~\eqref{eqn:Gmap*3} that
\[
\begin{array}{lll}
\iota_*(1) & = & \stackrel{n}{\underset{j = 1}\cup} \ c_1\left(L(D_j)\right) \\
& = & \stackrel{n}{\underset{j = 1}\cup} \ 
c_1\left(p^*(L_j) \otimes \sO_{{\bf 1} \oplus M}(1)\right) \\
& = & c_d\left(\left(\stackrel{n}{\underset{j = 1}\oplus} \ p^*(L_j)\right)
\otimes \sO_{{\bf 1} \oplus M}(1)\right) \\
& = &  c_d\left(p^*(F) \otimes \sO_{{\bf 1} \oplus M}(1)\right)
\end{array}
\]
which proves ~\eqref{eqn:Gmap*7}.  

In the third step, we prove the composition property in the special case
of the strict inclusions $Y \xrightarrow{s'} \P\left({\bf 1} \oplus N\right)
\xrightarrow{\iota} \P\left({\bf 1} \oplus M\right)$ in the situation of the 
second step. Let $s = \iota \circ s'$ denote the section of the projection map 
$p$ using the identification $Y = \P({\bf 1})$. We then show that
\begin{equation}\label{eqn:Gmap*8} 
s_* = \iota_* \circ s'_*.
\end{equation}

To prove this, we note that both sides are $\m(Y)$-linear maps and hence
it suffices to show that $s_*(1) = \iota_* \circ s'_* (1)$.
Let $n = {\rm rank}(N)$. We then have in 
$\m\left(\P\left({\bf 1} \oplus M\right)\right)$:
\[
\begin{array}{lll}
\iota_* \circ s'_*(1) & = & 
\iota_*\left[c_n\left(p^*_N(N) \otimes \sO_{{\bf 1} \oplus N}(1)\right)\right] 
\\
& = & \iota_*\left[\iota^*\left(c_n\left(p^*(N) \otimes \sO_{{\bf 1} 
\oplus M}(1)\right)\right)\right] \\
& = &  \iota_*(1) \cup c_n\left(p^*(N) \otimes \sO_{{\bf 1} 
\oplus M}(1)\right) \\
& = & \left[c_d\left(p^*(F) \otimes \sO_{{\bf 1} 
\oplus M}(1)\right)\right] \cup 
\left[c_n\left(p^*(N) \otimes \sO_{{\bf 1} \oplus M}(1)\right)\right] \\
& = & c_{d+n} \left(p^*(M) \otimes \sO_{{\bf 1} 
\oplus M}(1)\right) \\
& = & s_*(1)
\end{array}
\]
which proves ~\eqref{eqn:Gmap*8}.

In the final step, we complete the proof of the composition property
of the Gysin map. So let $Z \xrightarrow{\iota'} Y \xrightarrow{\iota} X$
be given strict closed embeddings of ind-schemes. 
We consider the deformation diagram

\begin{equation}\label{eqn:Gmap*9}
\xymatrix@C1pc{
Z \ar[r]^<<<<<<{k_0} \ar[d]_{s'} & Z \times \A^1 \ar[d]_{F^Y} & Z \ar[l]_>>>{k_1}
\ar[d]^{\iota'} \\
\P\left({\bf 1} \oplus N_Y(Z)\right) \ar[d]_{s^Y} 
\ar[r]_>>>>{j^Y_0} & M'_Y \ar[d]_{J} & Y \ar[d]^{\iota}
\ar[l]^{j^Y_1} \\
\P\left({\bf 1} \oplus N_X(Z)\right) \ar[r]_>>>>{j^X_0} & M'_X & X \ar[l]^{j^X_1},}
\end{equation}
where $s = s^Y \circ s'$, $F^X = J \circ F^Y$ and $M'_Y$ is the proper
transform of $Y \times \A^1$ in $M'_X$. Moreover, all the squares are
transverse. Using this transversality of the left squares, we get
the relations
\[
s^Y_* \circ s'_* \circ k^*_0 = s^Y_* \circ (j^Y_0)_* \circ F^Y_* 
= (j^X_0)^* \circ J_* \circ F^Y_*
\]
and
\[
s_* \circ k^*_0 = (j^X_0)^* \circ (J \circ F^Y)_*.
\]
We have shown in ~\eqref{eqn:Gmap*8} that $s^Y_* \circ s'_* = s_*$.
Thus we get
\[
(j^X_0)^* \circ \left[(J \circ F^Y)_* - J_* \circ F^Y_*\right] = 0.
\]
Since $(J \circ F^Y)_*$  and $J_* \circ F^Y_*$ both vanish in
$\m(V'_X)$ where $V'_X = M'_X \setminus (Z \times \A^1)$, it follows from
Lemma~\ref{lem:extra} that $(J \circ F^Y)_* = J_* \circ F^Y_*$.

Using the transversality of the right squares in the above diagram, we
get the relations
\[
\begin{array}{lll}
(\iota \circ \iota')_* \circ k^*_1 & = & 
(j^X_1)^* \circ (J \circ F^Y)_* \\
& = & (j^X_1)^* \circ  J_* \circ F^Y_* \\
& = & \iota_* \circ (j^Y_1)^*  \circ F^Y_* \\
& = & \iota_* \circ \iota'_* \circ k^*_1. \\
\end{array}
\]
Since $k_1$ is an $\A^1$-weak equivalence, we conclude that
$(\iota \circ \iota')_* = \iota_* \circ \iota'_*$.
This completes the proof of the composition property and hence the proof
of the proposition.
\end{proof}

\section{Basic properties of equivariant cobordism}
\label{section:Bprop}
Let $G$ be a linear algebraic group over $k$. The following result describes
the basic properties of the equivariant motivic cobordism. Recall that 
a bigraded ring $R = {\underset{i, j \ge 0}\oplus} \ R_{i,j}$ is called 
commutative if for $a \in R_{i,j}, b \in R_{i',j'}$, one has
$ab = (-1)^{i i'} ba$. Let ${\bf R}^*$ denote the category of commutative
bigraded rings. 
We need the following elementary result in order to establish some 
basic properties of the equivariant cobordism.

\begin{lem}\label{lem:Elem-eq}
Let $U \in \Sm^G_{{free}/k}$ and let $i: V_1 \inj V_2$ be a closed (or open) 
immersion in $\Sm^G_{{free}/k}$. Consider a transverse square
\begin{equation}\label{eqn:Elem-eq0}
\xymatrix@C1.5pc{
Y' \ar[r]^{g} \ar[d]_{\iota'} & Y \ar[d]^{\iota} \\
X' \ar[r]_{f} & X}
\end{equation}
in $\Sm^G_k$ where $\iota$ is a closed embedding.
Then the squares 
\begin{equation}\label{eqn:Elem-eq1}
\xymatrix@C1.5pc{
Y \stackrel{G}{\times} V_1 \ar[r]^{i_Y} \ar[d]_{\iota_1} &  
Y \stackrel{G}{\times} V_2 \ar[d]^{\iota_2} & &
Y' \stackrel{G}{\times} U \ar[r]^{\ov{g}} \ar[d]_{\ov{\iota'}} &
Y \stackrel{G}{\times} U \ar[d]^{\ov{\iota}} \\
X \stackrel{G}{\times} V_1 \ar[r]_{i_X} & X \stackrel{G}{\times} V_2 & & 
X' \stackrel{G}{\times} U \ar[r]_{\ov{f}} & X \stackrel{G}{\times} U} 
\end{equation}
are transverse in $\Sm_k$.
\end{lem}
\begin{proof}
This is an elementary exercise. We give only give a sketch and leave the 
details for the readers. One first checks that if ~\eqref{eqn:Elem-eq0} is
any Cartesian square in $\Sm^G_{{free}/k}$, then the squares
\begin{equation}\label{eqn:Elem-eq2}
\xymatrix@C1.5pc{
Y'/G \ar[r]^{\ov{g}} \ar[d]_{\ov{\iota'}} & Y/G \ar[d]^{\ov{\iota}} & & 
X' \ar[r]^{f} \ar[d]_{\pi'} & X \ar[d]^{\pi} \\
X'/G \ar[r]_{\ov{f}} & X/G & & X'/G \ar[r]_{\ov{f}} & X/G}
\end{equation}
are also Cartesian. In particular, the map $T_{{X'}/({{X'}/G})} \to
f^*\left(T_{{X}/({{X}/G})}\right)$ of relative tangent bundles is an isomorphism. 
From this, it follows immediately that $N_X(X') \xrightarrow{\cong}
\pi'^*\left(N_{({X}/G)}({{X'}/G})\right)$ if $f$ is a closed immersion.

To prove the lemma, it suffices to show that if ~\eqref{eqn:Elem-eq0} is a 
transverse square in $\Sm^G_{{free}/k}$, then the associated square
of quotients in ~\eqref{eqn:Elem-eq2} is also transverse. To do this, we now
only need to show the appropriate isomorphism of the normal bundles. 
We consider the commutative diagram

\begin{equation}\label{eqn:Elem-eq3}
\xymatrix@C1.5pc{
Y' \ar[rr]^{g} \ar[dd]_{\iota'} \ar[dr]_{\pi'} & & Y \ar[dd] \ar[dr]^{\pi} & \\
& Y'/G \ar[dd] \ar[rr]^<<<<<<{\ov{g}} & & Y/G \ar[dd]^{\ov{\iota}} \\
X' \ar[rr]_<<<<<<<<<{f} \ar[dr]_{p'} & & X \ar[dr]^{p} & \\
& X'/G \ar[rr]_{\ov{f}} & & X/G.}
\end{equation}
Let $\ov{N}$ (resp. $\ov{N'}$) denote the normal bundle of $Y/G$ 
(resp. $Y'/G$) in $X/G$ (resp. $X'/G$). Since $\pi'$ is a smooth covering,
it suffices to show that the map 
\begin{equation}\label{eqn:Elem-eq4}
\pi'^*(\ov{N'}) \to 
\pi'^*\left({\ov{g}}^*(\ov{N})\right)
\end{equation}
is an isomorphism. But this easily follows from the fact that the back face of 
the above cube is transverse and we have shown above that the two side faces 
are also transverse.
\end{proof}

\begin{thm}\label{thm:BPEC*}
The equivariant motivic cobordism $MGL^{*,*}_G(-)$ is an oriented cohomology
theory on $\Sm^G_k$ in the following sense.
\begin{enumerate}
\item
$Contravariance :$ \ $X \mapsto MGL^{*,*}_G(X)$ is a functor 
$\left(\Sm^G_k\right)^{\rm op} \to 
{\bf R}^*$. 
\item
$Homotopy \ Invariance :$ \ For a $G$-equivariant vector bundle 
$p : E \to X$, the map $p^*: MGL^{*,*}_G(X) \to MGL^{*,*}_G(E)$ is an
isomorphism. 
\item
$Chern \ classes :$ \ For a $G$-equivariant vector bundle $E$ on $X$,
there are equivariant Chern classes $c^G_i(E) \in MGL^{2i,i}_G(X)$ such that
$c^G_0(E) = 1, c^G_i(E) = 0$ for $i > {\rm rank}(E)$, 
$f^*\left(c^G_i(E)\right) = c^G_i\left(f^*(E)\right)$ for a morphism
$f : Y \to X$ in $\Sm^G_k$ and $c^G(E) = c^G(E') \cdot c^G(E'')$
for an exact sequence of equivariant vector bundles
$0 \to E' \to E \to E'' \to 0$ on $X$.
\item
$Projective \ bundle \ formula :$
For an equivariant vector bundle $E$ of rank $n$ on $X$, the map
\[
\Phi_X : \mg(X) \oplus \cdots \oplus \mg(X) \to
\mg\left(\P(E)\right);
\]
\[
\Phi\left(a_0, \cdots , a_{n-1}\right) \ =
\ \stackrel{n-1}{\underset{i = 0}\sum} \ \pi^*(a_i) \cdot \xi^i
\]
is an isomorphism, where $\pi : \P(E) \to X$ is the projection map  
and $\xi = c^G_1\left(\sO_E(-1)\right)$.
\item
$Gysin \ map :$ For a closed embedding $\iota : Y \inj X$ in $\Sm^G_k$ of 
codimension $d$, there is a Gysin map $\iota_*: MGL^{a, b}_G(Y) \to 
MGL^{a+2d, b+d}_G(X)$ such that given a transverse square
\begin{equation}\label{eqn:trans*}
\xymatrix@C.9pc{
Y' \ar[r]^{g} \ar[d]_{\iota'} & Y \ar[d]^{\iota} \\
X' \ar[r]_{f} & X}
\end{equation}
in $\Sm^G_k$ where $\iota$ is a closed embedding, the diagram
\begin{equation}\label{eqn:trans*0}
\xymatrix@C.9pc{
MGL^{*,*}_G(Y) \ar[d]_{\iota_*} \ar[r]^{g^*} & MGL^{*,*}_G(Y') \ar[d]^{\iota'_*} \\
MGL^{*,*}_G(X) \ar[r]_{f^*} & MGL^{*,*}_G(X')}
\end{equation}
commutes. Moreover, given $G$-equivariant closed embeddings 
$Z \xrightarrow{\iota'} Y \xrightarrow{\iota} X$, one has
$(\iota \circ \iota')_* = \iota_* \circ \iota'_*$. 

\item
$Gysin \ exact \ sequence :$ For a closed embedding $\iota: Y \inj X$
of codimension $d$ in $\Sm^G_k$ with the complement $j : U \inj X$, the sequence
\[
\cdots \to MGL^{a-1, b}_G(U) \xrightarrow{\partial} MGL^{a-2d, b-d}_G(Y) 
\xrightarrow{\iota_*} MGL^{a, b}_G(X) \xrightarrow{j^*}
MGL^{a, b}_G(U) \xrightarrow{\partial} \cdots 
\]
is exact.
\item
$Change \ of \ groups :$ If $H \subseteq G$ is a closed subgroup and 
$X \in \Sm^G_k$, then there is a natural restriction map $r^G_{H,X}: 
MGL^{*,*}_G(X) \to MGL^{*,*}_H(X)$. 
In particular, there is a natural forgetful map
\begin{equation}\label{eqn:forget}
r^G_X: MGL^{*,*}_G(X) \to \m(X).
\end{equation}
\item
$Morita \ Isomorphism : $
If $H \subseteq G$ is a closed subgroup and $X \in \Sm^H_k$, then there is a 
canonical isomorphism
$MGL^{*,*}_G(X \stackrel{H}{\times} G) \cong MGL^{*,*}_H(X)$.
\item
$Free \ action : $ If $X \in \Sm^G_{{free}/k}$, then the map
$\m(X/G) \to \mg(X)$ is an isomorphism.
\end{enumerate}
\end{thm}
\begin{proof}
Any map $f : Y \to X$ in $\Sm^G_k$ gives rise to the corresponding map
of ind-schemes $f_G: Y_G \to X_G$ which in turn induces the map
$f^* = f^*_G: \mg(X) = \m(X_G) \to \m(Y_G) = \mg(Y)$.
If $p: E \to X$ is a $G$-equivariant vector bundle of rank $n$ , then we have 
seen in the proof of Lemma~\ref{lem:Elem-eq} that the map
$p:E_G \to X_G$ is a vector bundle of rank $n$ over the ind-scheme $X_G$. 
In particular, it is an $\A^1$-weak equivalence. This proves property (2). 

If $E$ is a $G$-equivariant vector bundle on $X$, then equivariant Chern 
classes $c^G_i(E)$ are defined by
\begin{equation}\label{eqn:BPEC*0}
c^G_i(E) : = c_i\left(E_G\right) \in MGL^{2i,i}(X_G) = MGL^{2i, i}_G(X).
\end{equation}
The fact that this defines a Chern class theory for equivariant bundles
in $MGL^{*,*}_G(-)$ follows immediately from Proposition~\ref{prop:CCT}.

Let $E$ be a $G$-equivariant vector bundle of rank $n$ on $X$ and let
$p: \P(E) \to X$ be the associated equivariant projective bundle. Let
$p_G: \P\left(E_G\right) \to X_G$ denote the projective bundle associated to the
vector bundle $E_G$ on $X_G$. The desired projective bundle formula then
follows from ~\eqref{eqn:Proj-map} and the canonical 
isomorphism of ind-schemes $\P\left(E_G\right) \cong {\P(E)}_G$.

If $\iota : Y \inj X$ is an equivariant closed embedding, then
it follows from Lemma~\ref{lem:Elem-eq} that $\iota_G: Y_G \inj X_G$
is a strict closed embedding of ind-schemes. The equivariant Gysin map
$i_* : MGL^{a, b}_G(Y) \to MGL^{a+2d, b+d}_G(X)$ is defined as in 
Proposition~\ref{prop:Gmap-Prop}. 
Given a transverse square ~\eqref{eqn:trans*} in $\Sm^G_k$, it follows from
Lemma~\ref{lem:Elem-eq} that corresponding square of Borel spaces is
also transverse. The commutativity of the square
~\eqref{eqn:trans*0} and the composition property now follow from 
Proposition~\ref{prop:Gmap-Prop}.

To prove the Gysin exact sequence, we first note that if $\iota : Y \inj X$ is 
an equivariant closed embedding with complement $j : U \inj X$, then 
Lemma~\ref{lem:Elem-eq} implies that $\iota_G : Y_G \inj X_G$ is a closed
embedding with complement $U_G$. The exact sequence (6) now follows from
Proposition~\ref{prop:Gmap-Prop}. 

If $H \subseteq G$ is a closed subgroup and $X \in \Sm^G_k$, then 
$p: X_H \to X_G$ is a morphism of ind-schemes with fibers $G/H$.
This induces the restriction map $r^G_{H, X} : \mg(X) \to MGL^{*,*}_H(X)$.
Taking $H = \{e\}$ and using the isomorphism $X_{\{e\}} \cong X$, we get
the forgetful map $r^G_X$. If $X \in \Sm^H_k$, then the isomorphism
$MGL^{*,*}_G(X \stackrel{H}{\times} G) \cong MGL^{*,*}_H(X)$ follows from
Corollary~\ref{cor:Morita1}. The last property about the free action follows 
from \cite[Lemma~4.2.9]{MV}.
\end{proof}

\subsection{Self-intersection formula}\label{subsection:SIF}
We now prove the self-intersection formula for the equivariant motivic
cobordism. Let $\iota : Y \inj X$ be a closed embedding of codimension $d \ge 0$
in $\Sm^G_k$ and let $N_X(Y)$ denote the equivariant normal bundle of $Y$
in $X$. 

\begin{prop}\label{prop:SIFM}
For any $a \in \mg(Y)$, one has $\iota^* \circ \iota_*(a) = 
c^G_d\left(N_X(Y)\right) \cdot a$.
\end{prop} 
\begin{proof}
It follows from Lemma~\ref{lem:Elem-eq} that $\iota_G: Y_G \inj X_G$ is
a strict closed embedding of ind-schemes with normal bundle 
$\left(N_X(Y)\right)_G$. Using the definitions of the equivariant cobordism
and the equivariant Chern classes, it suffices to show that if
$\iota:Y \to X$ is a strict closed embedding of ind-schemes of codimension
$d \ge 0$ with normal bundle $N$, then $\iota^* \circ \iota_*(a) = 
c^G_d\left(N\right) \cdot a$ for all $a \in \m(Y)$.
So we prove this statement.

To prove this, we consider the diagram~\eqref{eqn:DNC} and make the following
claim. 

\begin{claim}\label{claim:SIFM0}
Given any $a \in \m(Y)$, there exists $b \in \m(M)$ such that 
$\iota_*(a) = j^*_1(b)$ and $(i_N)_*(a) = j^*_0(b)$.
\end{claim}
$Proof \ of \ the \ Claim: $  
Set $b = F_* \circ p^*(a) \in \m(M)$. We then have 
\[
\begin{array}{lllll}
\iota_*(a) & = & \iota_* \circ (p \circ k_1)^* (a) & = & 
\iota_* \circ k^*_1 \circ p^*(a) \\
& {=}^{\dagger} & j^*_1 \circ F_* \circ p^*(a) & = &  j^*_1(b)
\end{array}
\]
where the equality ${=}^{\dagger}$ follows from the transversality of the top 
right square
in ~\eqref{eqn:DNC} and Proposition~\ref{prop:Gmap-Prop}. On the other hand,
we have
\[
\begin{array}{lllll}
(i_N)_*(a) & = &  (i_N)_* \circ (p \circ k_0)^* (a) & = &   
(i_N)_* \circ k^*_0 \circ p^*(a) \\
& {=}^{\dagger \dagger} & j^*_0 \circ F_* \circ p^*(a) & = &  j^*_0(b)
\end{array}
\]
where the equality ${=}^{\dagger \dagger}$ follows from the transversality of the 
top left square
in ~\eqref{eqn:DNC} and Proposition~\ref{prop:Gmap-Prop}.
This proves the claim.

We now prove the self-intersection formula for ind-schemes. 
There is nothing to prove if $d = 0$
and so we assume that $d \ge 1$. We first consider the case when 
$p_N: N \to Y$ is a vector bundle of rank $d$ and $\iota : Y \inj N$ is
the zero section embedding of ind-schemes. Let 
$p : \P\left({\bf 1} \oplus N\right) \to Y$ be the projectivization of $N$,
giving us the diagram 
\[
\xymatrix@C1.5pc{
Y \ar[d]_{\iota} \ar@{=}[r] & Y \ar[d]^{s} \\
N \ar[r]_<<<<{j} & \P\left({\bf 1} \oplus N\right)}
\]
which is clearly transverse. In particular, we have $\iota_* = j^* \circ s_*$
using property (1) of Proposition~\ref{prop:Gmap-Prop}. Combining this
the property (4) of  Proposition~\ref{prop:Gmap-Prop}, we get
\[
\begin{array}{lll}
\iota^* \circ \iota_*(a) & = & \iota^*\left[j^* \circ s_*(a)\right] \\
& = & \iota^*\left[j^*\left(c_d\left(p^*(N) \otimes
\sO_{{\bf 1} \oplus N}(1)\right) \cdot p^*(a)\right)\right] \\
& = & \iota^*\left[c_d\left(j^*\left(p^*(N) \otimes
\sO_{{\bf 1} \oplus N}(1)\right)\right) \cdot j^* \circ p^*(a)\right] \\
& = & \iota^*\left[c_d\left(p^*_N(N)\right) \cdot p^*_N(a)\right] \\
& = & \iota^*\left(c_d\left(p^*_N(N)\right)\right) \cdot \iota^* \circ p^*_N(a)
\\
& = & c_d(N) \cdot a.
\end{array}
\]
 
In the general case, we fix $a \in \m(Y)$ and let $b \in \m(M)$ be as in
Claim~\ref{claim:SIFM0}. We then have 
\[
\begin{array}{lll}
\iota^* \circ \iota_*(a) & = & \iota^* \circ j^*_1(b)  \\
& = & k^*_1 \circ F^*(b) \\
& = & k^*_0 \circ F^*(b) \\
& = & i^*_N \circ j^*_0(b) \\
& = & i^*_N \circ (i_N)_*(a) \\
& = & c_d\left(N\right) \cdot a. 
\end{array}
\]
This completes the proof of the Self-intersection formula.
\end{proof}

\section{Equivariant cobordism for torus actions}
\label{section:T-action}
In this section, we study certain special features of the equivariant
motivic cobordism when the underlying group is a torus. Our main result is
to show that the equivariant motivic cobordism of smooth projective schemes
with a torus action has a simple description. We shall see later in this 
paper that the equivariant cobordism for the action of a connected reductive
group can be described in terms of the equivariant cobordism for the
action of maximal tori of the group. As applications of these results, we
shall compute the equivariant and ordinary motivic cobordism of various classes
of smooth schemes with group actions.

\subsection{Borel-Moore homology associated to motivic cobordism}
\label{subsection:BMH}
Let {\bf SP} denote the category of pairs $(M, X)$ with $M \in \Sm_k$ and
$X \subseteq M$ a closed subset (possibly singular); a morphism $f: (M, X) \to
(N, Y)$ is a morphism $f: M \to N$ such that $f^{-1}(Y) \subseteq X$.
An object of the form $(M, X)$ is called a {\sl smooth pair}.
Recall from \cite{Panin1} that an oriented cohomology theory $A$ on
{\bf SP} is a contravariant functor $(M, X) \mapsto A_X(M)$ with values
in abelian groups together with a functor 
$\partial : A(M \setminus X) \to A_X(M)$ which satisfies Gysin exact 
sequence, excision, homotopy invariance and Chern classes for vector bundles.

Let ${\rm {\bf SP}}^{\prime}$ denote the category of smooth pairs $(M, X)$ where 
a morphism
$f: (M, X) \to (N, Y)$ is a projective morphism $f: M \to N$ such that
$f(X) \subseteq Y$. Let $A$ be an oriented bi-graded ring cohomology theory on 
{\bf SP}.
An integration with supports on $A$ is an assignment of a bi-graded
push-forward map $f_*: A_X(M) \to A_Y(N)$ which satisfies the usual 
functoriality properties, compatibility with pull-back and Chern class
operators ({\sl cf.} \cite[Definition~1.8]{Levine2}). Levine shows that
every oriented ring cohomology theory on {\bf SP} has a unique integration
with supports.

Using the existence of integration with supports, Levine \cite{Levine2} 
has further shown that any given oriented bi-graded
ring cohomology theory $A$ on {\bf SP} uniquely extends to an oriented bi-graded
Borel-Moore homology theory $H$ on $\Sch_k$ such that the pair $(H,A)$ is
an oriented duality theory on $\Sch_k$ in the sense of 
\cite[Definition~3.1]{Levine2}. 
In particular, the Borel-Moore homology theory $H$ has projective push-forward,
pull-back under open immersion and
smooth projection, Gysin exact sequence, weak homotopy invariance, 
exterior product, Chern classes for vector bundles and the Poincar{\'e}
duality $H_{a,b}(X) \cong A^{2d-a, d-b}(X)$ if $X$ is smooth of dimension $d$.

For $X \in \Sch_k$, $H(X)$ is defined by
choosing a closed embedding $X \subseteq M$ with $M \in \Sm_k$ (which
is possible since $X$ is quasi-projective) and setting 
$H(X) : = A_X(M)$. The main point of \cite{Levine2} is to show that this
is well defined and has all the properties mentioned above.
The projective push-forward is constructed by showing that
given smooth pairs $(M,X), (N, Y)$ and a map $f: M \to N$ such that
$f|_X: X \to Y$ is projective, the orientation and the excision 
property of $A$ yield a well-defined map $f_*: A_X(M) \to A_Y(N)$.

Apart from the above, the homology theory $H$ also satisfies the following 
commutativity property.

\begin{lem}\label{lem:BMCom}
Let 
\[
\xymatrix@C1.2pc{
U' \ar[r]^{j'} \ar[d]_{g} & X' \ar[d]^{f} \\
U \ar[r]_{j} & X}
\]
be a Cartesian square in $\Sch_k$ such that $f$ is projective 
and $j$ is an open immersion. One has then, 
$j^* \circ f_* = g_* \circ j'^* : H(X') \to H(U)$.
\end{lem}
\begin{proof}
We can write $f = p \circ i$ where $i : X' \inj \P^n \times X$ is a closed
embedding and $p : \P^n \times X \to X$ is the projection map.
Since $f_* = p_* \circ i_*$, one easily checks that the assertion of the lemma
holds if it holds when $f$ is a closed immersion and a projection
$\P^n \times X \to X$. Assume first that $f : X' \to X$ is a closed immersion.

We can embed the above Cartesian diagram 
into a commutative diagram
\begin{equation}\label{eqn:BMCom0}
\xymatrix@C1.2pc{
& U' \ar[rr]^{j'} \ar[dd] \ar[dl]_{g} & & X' \ar[dl]^{f} \ar[dd] \\
U \ar[rr]_<<<<<<<{j} \ar[dd] & & X \ar[dd]^<<<<<<<{i} & \\
& M \ar[dl]^{id} \ar[rr] & & N \ar[dl]^{id} \\
M \ar[rr]_{J} & & N & }
\end{equation}
where $N$ is smooth, $i: X \inj N$ is a closed immersion, $J$ is an open 
immersion and the front square is Cartesian. 
Since the top and the bottom squares are also Cartesian, the same holds for
the back square too. Moreover, the bottom square is clearly transverse
whose vertices are smooth. The desired equality $j^* \circ f_* = g_* \circ j'^*$
now follows immediately from 
\cite[Definitions~ 2.7-A(4) and 3.1-A(1), A(2)]{Levine2}.

If $f : \P^n \times X \to X$ is the projection map, then we only have to
replace 
$N \xrightarrow{id} N$ and $M \xrightarrow{id} M$ in the bottom square 
of the diagram~\eqref{eqn:BMCom0} by $\P^n \times N \xrightarrow{p} N$ and 
$\P^n \times M \xrightarrow{p} M$ 
respectively. This makes all squares Cartesian and the bottom square
transverse with all vertices smooth. One concludes the proof
as before using \cite[Definitions~ 2.7-A(4) and 3.1-A(1), A(2)]{Levine2}.
\end{proof}

The Borel-Moore homology theory associated to the motivic cobordism $MGL$
is denoted by $MGL'$. It is defined by setting $MGL'(X) = MGL_X(M) : = 
MGL(M/U)$, where $(M,X)$ is a smooth pair with $U = M \setminus X$. 
Our study of the equivariant cobordism of smooth projective schemes with torus
action is based on the following result about $MGL'$.

\begin{prop}\label{prop:filter-Gen}
Let $X$ be a $k$-scheme with a filtration by closed subschemes  
\begin{equation}\label{eqn:filtration-Gen}
{\emptyset} = X_{-1} \subsetneq X_0 \subseteq \cdots \subseteq X_n = X
\end{equation}
such that there are maps ${\phi}_m : W_m = (X_m \setminus X_{m-1}) \to Z_m$ for 
$0 \le m \le n$ which are all vector bundles. Assume moreover that each 
$Z_m$ is smooth and projective. 
Then there is a canonical isomorphism
\[
\stackrel{n}{\underset{m=0}{\bigoplus}} MGL'_{*,*}\left(Z_m\right)
\xrightarrow{\cong} MGL'_{*,*}\left(X\right).
\]
\end{prop}  
\begin{proof}
We prove it by induction on $n$.
For $n = 0$, the map $X = X_0 \xrightarrow{{\phi}_0} Z_0$ is a  
vector bundle over a smooth scheme and hence the proposition follows from the
homotopy invariance of the motivic cobordism. 
We now assume by induction that $1 \le m \le n$ and 
\begin{equation}\label{eqn:split0}
\stackrel{m-1}{\underset{j=0}{\bigoplus}} MGL'_{*,*}\left(Z_j\right)
\xrightarrow{\cong} MGL'_{*,*}\left(X_{m-1}\right).
\end{equation}

The Gysin exact sequence for the inclusions $i_{m-1} : X_{m-1} \inj X_m$ and 
$j_{m} : W_{m} = X_m \setminus X_{m-1}$ of the closed and open subschemes
yields a long exact sequence
\begin{equation}\label{eqn:split*1}
\cdots \to MGL'_{*,*}\left(X_{m-1}\right) \xrightarrow{i_{(m-1)*}}
MGL'_{*,*}\left(X_m\right) \xrightarrow{j^*_m}
MGL'_{*,*}\left(W_m\right) \xrightarrow{\partial} \cdots.
\end{equation}
Using ~\eqref{eqn:split0}, it suffices now to construct a canonical splitting 
of the pull-back $j^*_m$ in order to prove the proposition.  

Let $V_m \subset W_m \times Z_m$ be the graph of the projection 
$W_m \xrightarrow{{\phi}_m} Z_m$ and let $\ov{V}_m$ 
denote the closure of $V_m$ in $X_m \times Z_m$.  
Let $Y_m \to \ov{V}_m$ be a resolution of singularities. Since $V_m$ is 
smooth, we see that $V_m \stackrel{\ov{j}_m}{\inj} Y_m$ as an open subset. We 
consider the composite maps
\begin{equation}\label{eqn:split01}
p_m : V_m \inj W_m \times Z_m \to W_m, \ \ 
q_m : V_m \inj W_m \times Z_m \to Z_m \ \ {\rm and} 
\end{equation}
\[
{\ov{p}}_m : Y_m \to X_m \times Z_m \to X_m, \ \ 
{\ov{q}}_m : Y_m \to X_m \times Z_m \to Z_m. 
\]
Note that ${\ov{p}}_m$ is a projective morphism since $Z_m$ is projective.
The map $q_m$ is smooth and $p_m$ is an isomorphism. We consider the diagram
\begin{equation}\label{eqn:split1}
\xymatrix{
{MGL'_{*,*}\left(Z_m\right)} \ar[r]^{{\ov{q}}^*_m} 
\ar[d]_{{\phi}^*_m}^{\cong} &
{MGL'_{*,*}\left(Y_m\right)} \ar[d]^{{{\ov{p}}_m}_*} \\
{MGL'_{*,*}\left(W_m\right)} & 
{MGL'_{*,*}\left(X_m\right)} \ar[l]^{j^*_m}.}
\end{equation} 
Note that the maps ${{\ov{p}}_m}_*$ and $j^*_m$ exist by the above mentioned
properties of $MGL'$ and the maps ${\ov{q}}^*_m$ and ${\phi}^*_m$ exist
by the standard functoriality of $MGL$ as $Z_m, W_m$ and $Y_m$ are all smooth.

The map ${{\phi}^*_m}$ is an isomorphism by the homotopy invariance of the 
$MGL$-theory. It suffices to show that
this diagram commutes. For, the map 
$s_m : = {{\ov{p}}_m}_* \circ {\ov{q}}^*_m \circ
{{\phi}^*_m}^{-1}$ will then give the desired splitting of the map
$j^*_m$. 

We now consider the commutative diagram
\[
\xymatrix{
X_m & W_m \ar[l]_{j_m}& \\
Y_m \ar[u]^{{\ov{p}}_m} \ar[dr]_{{\ov{q}}_m} & V_m \ar[u]_{p_m} \ar[d]^{q_m}
\ar[l]^{{\ov{j}}_m} & W_m \ar[ul]_{id} 
\ar[l]^{(id, {\phi}_m)} \ar[dl]^{{\phi}_m} \\
& Z_m. & }
\]
Since the top left square is Cartesian with $Y_m$ smooth and $j_m$ an open 
immersion, it follows from Lemma~\ref{lem:BMCom}
that $j^*_m \circ {{\ov{p}}_m}_*  = {p_m}_* \circ {\ov{j}}^*_m$. 
Now, using the fact that $(id, {\phi}_m)$ is an isomorphism, 
we get 
\[
\begin{array}{lllll}
j^*_m \circ {{\ov{p}}_m}_* \circ {\ov{q}}^*_m & = & 
{p_m}_* \circ {\ov{j}}^*_m \circ {\ov{q}}^*_m & = &  
{p_m}_* \circ q^*_m \\
& = & {p_m}_* \circ {(id, {\phi}_m)}_* \circ {(id, {\phi}_m)}^* \circ
q^*_m & = & {id}_* \circ {\phi}^*_m \\
& = &  {\phi}^*_m. & &  
\end{array} 
\]
This proves the commutativity of ~\eqref{eqn:split1} and hence the proposition.
\end{proof}

\subsection{Equivariant cobordism of filtrable schemes}
\label{subsection:Filtrable}
Recall that a linear algebraic group $T$ over $k$ is said to be a {\sl split}
torus if it is isomorphic to $\left(\G_m\right)^n$ as a group scheme over
$k$ where $n \ge 1$ is a positive integer, called the rank of the torus.
We shall assume all tori to be split in this section. 

We recall from \cite[Section~3]{Brion2} that a $k$-scheme $X$ with an action 
of a torus $T$ is called {\sl filtrable} if the fixed point locus $X^T$ 
is smooth and projective, 
and there is an ordering $X^T = \stackrel{n}{\underset{m=0}{\coprod}}
Z_m$ of the connected components of the fixed point locus, a 
filtration of $X$ by $T$-invariant closed subschemes
\begin{equation}\label{eqn:filtration-BB}
{\emptyset} = X_{-1} \subsetneq X_0 \subseteq \cdots \subseteq X_n = X
\end{equation}
and maps ${\phi}_m : W_m = (X_m \setminus X_{m-1}) \to Z_m$ for $0 \le m \le n$ 
which are all $T$-equivariant vector bundles. It is important to note that
the closed subschemes $X_m$'s may not be smooth even if $X$ is so.
The following result was proven by Bialynicki-Birula \cite{BB}
when $k$ is algebraically closed and by Hesselink
\cite{Hessel} in general.

\begin{thm}[Bialynicki-Birula, Hesselink]\label{thm:BBH}
Let $X$ be a smooth projective scheme with an action of $T$. Then $X$ is
filtrable.
\end{thm}

\subsubsection{Canonical admissible gadgets}\label{subsubsection:CAG}
Let $T$ be a split torus of rank $r$. 
For a character $\chi$ of $T$, let $L_{\chi}$ denote the one-dimensional 
representation of $T$ where $T$ acts via $\chi$. Given a basis
$\{\chi_1, \cdots , \chi_r\}$ of the character group $\wh{T}$ of $T$ and given
$i \ge 1$, we set $V_i = \stackrel{r}{\underset{j = 1} \prod} 
L^{\oplus i}_{\chi_j}$ and $U_i = \stackrel{r}{\underset{j = 1} \prod} 
\left(L^{\oplus i}_{\chi_j} \setminus \{0\}\right)$.
Then $T$ acts on $V_i$ by $(t_1, \cdots , t_r)(x_1, \cdots , x_r)
= \left(\chi_1(t_1)(x_1), \cdots , \chi_n(t_r)(x_r)\right)$.
It is then easy to see that $\rho = \left(V_i, U_i\right)_{i \ge 1}$ is an 
admissible gadget for $T$ such that ${U_i}/T \cong \left(\P^{i-1}\right)^r$.
Moreover, the line bundle $L_{\chi_j} \stackrel{T_j}{\times} 
\left(L^{\oplus i}_{\chi_j}\setminus \{0\}\right) \to \P^{i-1}$ is the line bundle
$\sO(\pm 1)$ for each $1 \le j \le r$.
An admissible gadget for $T$ of this form will be called a 
{\sl canonical} admissible gadget in this text.

Let $X \in \Sm^T_k$ be a filtrable scheme with the filtration given by 
~\eqref{eqn:filtration-BB}.
Let $\rho = \left(V_i, U_i\right)_{i \ge 1}$ be a canonical admissible gadget for 
$T$ and set 
\[
X^i = X \stackrel{T}{\times} U_i, \ X^i_m = X_m \stackrel{T}{\times} U_i, \
W^i_m = W_m \stackrel{T}{\times} U_i \ {\rm and} \ 
Z^i_m = Z_m \stackrel{T}{\times} U_i.
\]

Given the $T$-equivariant filtration of $X$ as in ~\eqref{eqn:filtration-BB}, 
it is easy to see that for each $i \ge 1$, there is an 
associated system of filtrations 
\begin{equation}\label{eqn:filter-Equiv1}
{\emptyset} = X_{-1}^i \subsetneq X_0^i \subseteq \cdots \subseteq X_n^i = X^i
\end{equation}
and maps ${\phi}_m : W_m^i = X_m^i \setminus X_{m-1}^i \to Z_m^i$ for 
$0 \le m \le n$ which are all vector bundles. Moreover, 
as $T$ acts trivially on each $Z_m$, we have that $Z_m^i 
\cong Z_m \times \left({U_i}/T\right) \cong Z_m \times \left(\P^{i-1}_k\right)^r$.
Since $Z_m$ is smooth and projective, this in turn implies that
each $Z_m^i$ is smooth and projective. We conclude that the filtration
~\eqref{eqn:filter-Equiv1} of $X^i$ satisfies all the conditions of 
Proposition~\ref{prop:filter-Gen}. In particular, there are split exact
sequences 
\begin{equation}\label{eqn:ft*}
0 \to MGL'_{*,*}\left(X_{m-1}^i\right) \to  MGL'_{*,*}\left(X_{m}^i\right)
\to  MGL'_{*,*}\left(W_{m}^i\right) \to 0
\end{equation}
for all $0 \le m \le n$ and $i \ge 1$.

\begin{lem}\label{lem:MLTorus}
Let $X \in \Sm^T_k$ be smooth and projective and let $a \ge b \ge 0$.
Then for any admissible gadget 
$\rho = \left(V_i, U_i\right)_{i \ge 1}$ for $T$, the inverse system
$\left\{MGL^{a,b}(X \stackrel{T}{\times} U_i)\right\}$
satisfies the Mittag-Leffler condition. In particular, the map 
\[
MGL^{a,b}_T(X) \ \to \ {\underset{i}\varprojlim} \ 
MGL^{a,b}\left(X \stackrel{T}{\times} U_i\right)
\]
is an isomorphism.
\end{lem}
\begin{proof}
Since $MGL^{a,b}_T(X) = MGL^{a,b}\left(X_G(\rho)\right)$, the second assertion
follows from the Mittag-Leffler condition using Proposition~\ref{prop:Limit}.
To prove the Mittag-Leffler condition, we first assume that
$\rho$ is a canonical admissible pair and prove the stronger
assertion that the
restriction map $\m(X^{i+1}) \to \m(X^i)$ is surjective for all $i \ge 1$, 
where $X^i = X \stackrel{T}{\times} U_i$.

Since the map $\m\left(X \stackrel{T}{\times}(U_i \oplus W_i)\right) \to
 \m(X^i)$ is an isomorphism by the homotopy invariance, we only have to show 
that the map $\m(X^{i+1}) \to 
\m\left(X \stackrel{T}{\times}(U_i \oplus W_i)\right)$ induced by the open 
immersion is an isomorphism. 

It follows from Theorem~\ref{thm:BBH} that $X$ is filtrable. 
Consider a $T$-equivariant filtration of $X$ as in ~\eqref{eqn:filtration-BB}.
Set $\ov{X_i} = 
X \stackrel{T}{\times}(U_i \oplus W_i)$.
We show by induction on $m \ge 0$ that the pull-back map
$MGL'_{*,*}(X^{i+1}_m) \to MGL'_{*,*}(\ov{X^i_m})$ induced by the open immersion,
is surjective for all $m \ge 0$.

For any $0 \le m \le n$, there is a commutative diagram
\begin{equation}\label{eqn:MLTorus0}
\xymatrix@C1.2pc{
\m\left(Z^{i+1}_m\right) \ar[r] \ar[d] & 
\m\left(\ov{Z^{i}_m}\right) \ar[r] \ar[d] & 
\m\left(Z^{i}_m\right) \ar[d] \\
\m\left(W^{i+1}_m\right) \ar[r] & 
\m\left(\ov{W^{i}_m}\right) \ar[r] & 
\m\left(W^{i}_m\right)}
\end{equation}
of the motivic cobordism of smooth schemes where all the vertical arrows
are isomorphisms by the homotopy invariance. The left horizontal arrows
in both rows are isomorphisms again by the homotopy invariance.

Next, we observe that $T$ acts trivially on each $Z_m$ and hence
$Z_m^i \cong Z_m \times ({U_i}/T) \cong Z_m \times \left(\P^{i-1}\right)^r$.
Hence the projective bundle formula for the motivic cobordism ({\sl cf.}
\S ~\ref{subsubsection:PBFIS}) implies that the map
$\m\left(Z^{i+1}_m\right) \to \m\left(Z^{i}_m\right)$ is surjective. 
In particular, all the arrows on the top row in ~\eqref{eqn:MLTorus0} are
surjective. We conclude that the map
$\m\left(W^{i+1}_m\right) \to \m\left(\ov{W^{i}_m}\right)$ is surjective
for all $0 \le m \le n$ and all $i \ge 1$.
Taking $m =0$, we see in particular that the map
$MGL'_{*,*}(X^{i+1}_0) \to MGL'_{*,*}(\ov{X^i_0})$ is surjective for all
$i \ge 1$.

Assume now that $m \ge 1$ and that this surjectivity assertion holds for all
$j \le m-1$. We consider the diagram
\begin{equation}\label{eqn:MLTorus1}
\xymatrix@C1.2pc{
0 \ar[r] & MGL'_{*,*}(X^{i+1}_{m-1}) \ar[r] \ar[d] &
MGL'_{*,*}(X^{i+1}_{m}) \ar[r] \ar[d] & 
MGL'_{*,*}(W^{i+1}_{m}) \ar[r] \ar[d] & 0 \\
0 \ar[r] &  MGL'_{*,*}(\ov{X^{i}_{m-1}}) \ar[r] &
MGL'_{*,*}(\ov{X^{i}_{m}}) \ar[r] & 
MGL'_{*,*}(\ov{W^{i}_{m}}) \ar[r] & 0.}
\end{equation}

The left square is commutative by Lemma~\ref{lem:BMCom} and the right square
is commutative by the functoriality of open pull-back 
({\sl cf.} \cite[p. 34]{Levine2}). 
The left vertical arrow is surjective by induction and we have shown
above that the right vertical arrow is surjective. 
The top sequence is exact by
~\eqref{eqn:ft*}. Suppose we know that the bottom sequence is also exact.
It will then follow that the middle
vertical arrow is surjective. Thus we are only left with showing that
the bottom sequence of ~\eqref{eqn:MLTorus1} is exact. 

Using the Gysin exact sequence, it is enough to show that the open pull-back
$MGL'_{*,*}(\ov{X^{i}_{m}}) \to MGL'_{*,*}(\ov{W^{i}_{m}})$ is surjective.
This is equivalent to showing that the map
$MGL^{*,*}_{\ov{X^{i}_{m}}}\left(\ov{X^{i}}\right) \to
MGL^{*,*}\left(\ov{W^{i}_{m}}\right)$ is surjective. For this, we consider the 
diagram

\begin{equation}\label{eqn:MLTorus1}
\xymatrix@C1.2pc{
0 \ar[r] & MGL^{*,*}_{X^i_{m-1}}\left(X^i\right) \ar[r] &
MGL^{*,*}_{X^i_m}\left(X^i\right) \ar[r] \ar[d]_{p^*} & 
MGL^{*,*}\left(W^i_m\right) \ar[d]^{p^*} \ar[r] & 0 \\
& & MGL^{*,*}_{\ov{X^{i}_{m}}}\left(\ov{X^{i}}\right) \ar[r] &
MGL^{*,*}\left(\ov{W^{i}_{m}}\right), & }
\end{equation}
where the top sequence is exact by ~\eqref{eqn:ft*}. The vertical arrows
are the pull-back maps induced by the vector bundle  
$p : \ov{X^{i}} \to X^i$ on the smooth scheme $X^i$
({\sl cf.} \cite[Definition~2.7-A(6)]{Levine2}). 
Hence the right vertical arrow is an isomorphism by the homotopy invariance
of the motivic cobordism since $W^i_m$ is smooth.
It follows that the bottom horizontal arrow is surjective.
This completes the proof of the first assertion for a canonical 
admissible pair.

Let us now assume that $\rho = \left(V_i, U_i\right)$ is any admissible
pair for $T$ and let $\rho' = \left(V'_i, U'_i\right)$ be a canonical pair.
Set ${X'}^i = X \stackrel{T}{\times} U'_i$ and 
$Y^{i,j} = X \stackrel{T}{\times} (U_i \oplus U'_j)$.

Fix $i_0 \ge 1$. We have shown in the proof of Lemma~\ref{lem:Ind} that
there exists $s_0 \gg 0$ such that the map 
$\alpha^*_{i_0, j} : MGL^{a,b}\left(X^{i_0}\right) \to 
MGL^{a,b}\left(Y^{i_0, j}\right) \to $ is
an isomorphism for all $j \ge s_0$. 
By reversing the role of the admissible pairs, let $i_1 \gg i_0$ be such that 
the map $\beta^*_{i, s_0} :  MGL^{a,b}\left({X'}^{s_0}\right) \to 
MGL^{a,b}\left(Y^{i, s_0}\right)$
is an isomorphism for all $i \ge i_1$.

Let us now fix an an element 
$a \in {\rm Image} \left(MGL^{a,b}(X^{i_1}) \to MGL^{a,b}(X^{i_0})\right)$.
For any $i \ge i_1$, we get a commutative diagram
\[
\xymatrix{
MGL^{a,b} (X^i) \ar[r]^{\cong} \ar[d] & 
MGL^{a,b} \left(Y^{i,s_1}\right) \ar[d] & 
MGL^{a,b} \left({X'}^{s_1}\right) \ar[l] \ar@{->>}[d] \\
MGL^{a,b}\left(X^{i_1}\right) \ar[r] \ar[d] & 
MGL^{a,b}\left(Y^{i_1,s_0}\right) \ar[d] &    
MGL^{a,b}\left({X'}^{s_0}\right) \ar[l]^{\cong} \\
MGL^{a,b}\left(X^{i_0}\right) \ar[r]_{\cong} &
MGL^{a,b}\left(Y^{i,s_0}\right) & }
\]
in which $s_1 \gg s_0$ is chosen so that the top left horizontal map is
an isomorphism. We have shown above that the extreme right 
vertical arrow is surjective. An easy diagram chase shows that
$a \in {\rm Image} \left(MGL^{a,b}(X^{i}) \to MGL^{a,b}(X^{i_0})\right)$.
This completes the proof of the lemma.
\end{proof}

The following consequence of Lemma~\ref{lem:MLTorus} improves
Corollary~\ref{cor:Mot-Geom} in the special cases of torus action.

\begin{cor}\label{cor:Mot-Geom-Torus}
For $X \in \Sm^T_k$ be smooth and projective.
Then for $q \ge 0$, there is a natural isomorphism
\begin{equation}\label{eqn:Mot-Geom0}
\Omega^q_T(X) \xrightarrow{\cong} MGL^{2q,q}_T(X). 
\end{equation}
\end{cor}  
\begin{proof}
Let $\rho = \left(V_i, U_i\right)_{i \ge 1}$ be a canonical admissible gadget 
for $T$.
It follows from \cite[Theorem~6.1]{Krishna1} that there is a natural
isomorphism ${\underset {i} \varprojlim} \ 
\Omega^q\left(X \stackrel{T}{\times} U_i\right) 
\xrightarrow{\cong} \Omega^q_T(X)$. 
On the other hand, the map 
$\Omega^q\left(X \stackrel{T}{\times} U_i \right) \to 
MGL^{2q,q}\left(X \stackrel{T}{\times} U_i\right)$ is an isomorphism for each
$i \ge 1$ by \cite[Theorem~3.1]{Levine1}. The corollary now follows
by applying Lemma~\ref{lem:MLTorus}.
\end{proof}

Let $X \in \Sm^T_k$ be smooth and projective with a filtration as in
~\eqref{eqn:filtration-BB}. Set $S_m = X \setminus X_{m-1}$ for $0 \le m \le n$. 
Let $V_m \subset W_m \times Z_m$ be the scheme 
defined in ~\eqref{eqn:split01} and let $Y_m \to 
\ov{V}_m$ be the canonical $T$-equivariant resolution of singularities.
One easily checks that all the maps in ~\eqref{eqn:split01} then become
$T$-equivariant. 

Let $\rho = \left(V_i, U_i\right)_{i \ge 1}$ be a canonical admissible gadget for 
$T$. This yields for every $0 \le m \le n$ and $i \ge 1$, the maps 
$j^i_m: (S^i_{m}, W^i_m) \to (X^i, X^i_m)$ in {\bf SP} and the maps
${\ov{p}}^i_m: (Y^i_m, Y^i_m) \to (X^i, X^i_m)$ in ${\rm {\bf SP}}^{\prime}$.
Note also that for every $i\ge 1$, there is a closed immersion
$\gamma^i_X: X^i \inj X^{i+1}$, which is natural with respect to maps in 
$\Sch^T_k$.


\begin{lem}\label{lem:filter-commute}
Let $X \in \Sm^T_k$ be smooth and projective and fix $0 \le m \le n$ and 
$i \ge 1$. Consider the notations of ~\eqref{eqn:split01} and  
~\eqref{eqn:split1}
where now all the maps are $T$-equivariant. Then the diagram
\begin{equation}\label{eqn:filt-comm1}
\xymatrix@C1pc{
{\m\left(Z_m^{i+1}\right)} 
\ar@/^1pc/[rr]^{{\ov{q}}^*_m} 
\ar[dd]_{{\phi}^*_m}^{\cong} \ar[dr] & &
{\m\left(Y^{i+1}_m\right)} \ar[dd]^{{{\ov{p}}_m}_*} \ar[dr] & \\
& {\m\left(Z_m^i\right)} \ar@/_1pc/[rr]_<<<<<<<<{{\ov{q}}^*_m} 
\ar[dd]_{{\phi}^*_m}^{\cong} & & {\m\left(Y_m^i\right)} 
\ar[dd]^{{{\ov{p}}_m}_*} \\
{\m\left(W_m^{i+1}\right)} \ar[dr] & &
{\m_{X^{i+1}_m}\left(X^{i+1}\right)} \ar@/^1pc/[ll]^<<<<<<{j^*_m} \ar[dr] & \\
& {\m\left(W_m^i\right)} & & {\m_{X^i_m}\left(X^i\right)}
\ar@/^1pc/[ll]^{j^*_m}}
\end{equation} 
commutes.
\end{lem}
\begin{proof}
We have shown in the proof of Proposition~\ref{prop:filter-Gen}
({\sl cf.} ~\eqref{eqn:split1}) that the front and the back squares commute.
The left and the top squares commute by the functoriality of pull-backs
in {\bf SP}. Notice that the map $\m_{W_m^i}\left(S^i_m\right) \to 
\m\left(W_m^i\right)$ induced by the the inclusion 
$(W^i_{m}, W^i_m) \inj (S^i_{m}, W^i_m)$ is an isomorphism.
Thus the bottom square commutes again by the functoriality of pull-backs
in {\bf SP}. We only need to explain why does the right square commute.

To do this, we consider another diagram
\begin{equation}\label{eqn:filt-comm2}
\xymatrix@C1.5pc{
& Y^i_m \ar[rr]^{{\ov{p}}^i_m} \ar@{=}[dd] \ar[dl]_{\gamma^i_{Y_m}} & & 
X^i_m \ar[dl]_{\gamma^i_{X_m}} \ar[dd] \\
Y^{i+1}_m \ar[rr]_<<<<<<<{{\ov{p}}^{i+1}_m} \ar@{=}[dd] & & X^{i+1}_m 
\ar[dd] & \\
& Y^i_m \ar[dl]^{\gamma^i_{Y_m}} \ar[rr]^>>>>>>>{P^i_m} & & X^i \ar[dl]^{\gamma^i_X} \\
Y^{i+1}_m \ar[rr]_{P^{i+1}_m} & & X^{i+1} & }
\end{equation}
induced by the maps $(X^i, X^i_m) \to (X^{i+1}, X^{i+1}_m)$,
$(Y^i_m, Y^i_m) \to (Y^{i+1}_m, Y^{i+1}_m)$ in {\bf SP} and the maps
$(Y^i_m, Y^i_m) \to (X^i, X^i_m), (Y^{i+1}_m, Y^{i+1}_m) \to (X^{i+1}, X^{i+1}_m)$
in ${\rm {\bf SP}}^{\prime}$.

It is easy to see that the top, bottom, left and right squares are Cartesian.
Moreover, it follows from Lemma~\ref{lem:Elem-eq} that the bottom square
is transverse with all vertices smooth. Since $P_m$ is projective and
$\gamma$ is a closed immersion, it follows from the standard
properties of an oriented cohomology theory having integration with supports
({\sl cf.} \cite[Definition~2.7-A(4)]{Levine2}) that the left square of
~\eqref{eqn:filt-comm1} commutes.
\end{proof}

Let $X \in \Sm^T_k$ be smooth and projective with the $T$-equivariant filtration
~\eqref{eqn:filtration-BB}.
This equivariant filtration on $X$ induces a commutative diagram
in {\bf SP}
\begin{equation}\label{eqn:filt-comm3}
\xymatrix@C1.5pc{
(X^i, X^i_{m-1}) \ar[d]_{\gamma^i_{m-1}} &
(X^i, X^i_{m}) \ar[d]^{\gamma^i_m} \ar[l]_{\iota^i_{m-1}}& (S^i_{m}, W^i_{m}) 
\ar[l]_{j^i_{m}} \ar[d]^{\gamma^i_{W_m}} \\
(X^{i+1}, X^{i+1}_{m-1})  &
(X^{i+1}, X^{i+1}_{m}) \ar[l]^{\iota^{i+1}_{m-1}} & (S^{i+1}_{m}, W^{i+1}_{m}) 
\ar[l]^{j^{i+1}_{m}}.}
\end{equation} 
If we consider the associated diagram of motivic cobordism with supports
and use the identification $MGL'(X^i_m) = MGL_{X^i_m}\left(X^i\right)$, we
obtain a commutative diagram

\begin{equation}\label{eqn:filt-comm4}
\xymatrix@C1.8pc{
0 \ar[r] & \m_{X^{i+1}_{m-1}}\left(X^{i+1}\right) 
\ar[r]^{\left({\iota^{i+1}_{m-1}}\right)^*} \ar[d]_{\left({\gamma^i_{m-1}}\right)^*} &
\m_{X^{i+1}_{m}}\left(X^{i+1}\right) \ar[r]^{\left({j^{i+1}_{m}}\right)^*} 
\ar[d]^{\left({\gamma^i_m}\right)^*} & \m\left(W^{i+1}_m\right) 
\ar[d]^{\left({\gamma^i_{W_m}}\right)^*} \ar[r] & 0 \\
0 \ar[r] & \m_{X^{i}_{m-1}}\left(X^{i}\right) 
\ar[r]_{\left({\iota^{i}_{m-1}}\right)^*} & \m_{X^{i}_{m}}\left(X^{i}\right)
\ar[r]_{\left({j^{i}_{m}}\right)^*} & \m\left(W^{i}_m\right) \ar[r] & 0.}
\end{equation}
Notice that $\left({\iota^{i}_{m-1}}\right)^*$ is same as 
$\left({\iota^{i}_{m-1}}\right)_*$ under the identification
$MGL'(X^i_m) = MGL_{X^i_m}\left(X^i\right)$ ({\sl cf.} 
\cite[Definition~1.8-(5)]{Levine2}). The two rows are exact 
and we have shown in the proof of Proposition~\ref{prop:filter-Gen}
({\sl cf.} ~\eqref{eqn:split*1}) that the map ${\left({j^{i}_{m}}\right)^*}$
is split by $s^i_m : = \left({\ov{p}}^i_m\right)_* \circ 
\left({\ov{q}}^i_m\right)^* \circ \left(\left(\phi^i_m\right)^*\right)^{-1}$
for each $i \ge 1$ ({\sl cf.} diagram~\eqref{eqn:filt-comm1}). 
In particular, the two rows form split short exact sequences. 
We now show that 
\begin{equation}\label{eqn:filt-comm5}
s^i_m \circ \left({\gamma^i_{W_m}}\right)^* = \left({\gamma^i_m}\right)^* \circ
s^{i+1}_m.
\end{equation}

To show this, it is equivalent to showing that
\begin{equation}\label{eqn:filt-comm6}
\left({\gamma^i_m}\right)^* \circ \left({\ov{p}}^{i+1}_m\right)_* \circ 
\left({\ov{q}}^{i+1}_m\right)^* = 
\left({\ov{p}}^i_m\right)_* \circ 
\left({\ov{q}}^i_m\right)^* \circ \left(\left(\phi^i_m\right)^*\right)^{-1} 
\circ \left({\gamma^i_{W_m}}\right)^* \circ \left(\phi^{i+1}_m\right)^*.
\end{equation}

On the other hand, it follows from Lemma~\ref{lem:filter-commute} that
\[
\begin{array}{lll}
\left(\left(\phi^i_m\right)^*\right)^{-1} 
\circ \left({\gamma^i_{W_m}}\right)^* \circ \left(\phi^{i+1}_m\right)^*
& = & \left(\left(\phi^i_m\right)^*\right)^{-1} 
\circ \left(\phi^i_m\right)^* \circ \left(\gamma^i_{Z_m}\right)^* \\
& = & \left(\gamma^i_{Z_m}\right)^*.
\end{array}
\]

Applying Lemma~\ref{lem:filter-commute} again, we get
\[
\begin{array}{lll}
\left({\ov{p}}^i_m\right)_* \circ 
\left({\ov{q}}^i_m\right)^* \circ \left(\left(\phi^i_m\right)^*\right)^{-1} 
\circ \left({\gamma^i_{W_m}}\right)^* \circ \left(\phi^{i+1}_m\right)^*
& = & 
\left({\ov{p}}^i_m\right)_* \circ 
\left({\ov{q}}^i_m\right)^* \circ \left(\gamma^i_{Z_m}\right)^* \\
& = & \left({\ov{p}}^i_m\right)_* \circ \left(\gamma^i_{Y_m}\right)^* \circ
\left({\ov{q}}^{i+1}_m\right)^* \\
& = & \left({\gamma^i_m}\right)^* \circ \left({\ov{p}}^{i+1}_m\right)_* \circ 
\left({\ov{q}}^{i+1}_m\right)^*.
\end{array}
\]
This shows ~\eqref{eqn:filt-comm6} and hence ~\eqref{eqn:filt-comm5}.

\begin{lem}\label{lem:filter-Equiv}
Let $X \in \Sm^T_k$ be smooth and projective with the $T$-equivariant filtration
~\eqref{eqn:filtration-BB}. Then for every $0 \le m \le n$,
there is a canonical split exact sequence
\begin{equation}\label{eqn:filter-Equiv*}
0 \to {\underset{i}\varprojlim} \  \m_{X^{i}_{m-1}}\left(X^{i}\right) 
\xrightarrow{\iota_{m-1}^*} 
{\underset{i}\varprojlim} \  \m_{X^{i}_{m}}\left(X^{i}\right) 
\xrightarrow{j_{m}^*} {\underset{i}\varprojlim} \
\m\left(W^{i}_m\right) \to 0.
\end{equation}
\end{lem}
\begin{proof}
It follows from ~\eqref{eqn:filt-comm4} and the left exactness of the inverse 
limit that there is a sequence as above which is exact except possibly at
the right end. But ~\eqref{eqn:filt-comm5} shows that $j^*_m \circ s^*_m$ is 
identity on ${\underset{i}\varprojlim} \ \m\left(W^{i}_m\right)$. 
This proves the lemma.
\end{proof}

\begin{lem}\label{lem:trivial-T}
Let $T$ be a split torus of rank $n$ acting trivially on a smooth
scheme $X$ of dimension $d$ and let $\{\chi_1, \cdots , \chi_n\}$ be a chosen
basis of $\widehat{T}$. Then the assignment $t_j \mapsto c^T_1(L_{\chi_j})$
induces an isomorphism of graded rings
\begin{equation}\label{eqn:trivial-T1}
\m(X)[[t_1, \cdots , t_n]] \xrightarrow{\cong} \mt(X)
\end{equation}
where $\m(X)[[t_1, \cdots , t_n]]$ is the graded power series ring over
$\m(X)$.
\end{lem}   
\begin{proof}
Let $\rho = \left(V_i, U_i\right)_{i \ge 1}$ be a canonical admissible gadget for 
$T$. Since $T$ acts trivially on $X$, we have $X \stackrel{T}{\times} U_i
\cong X \times \left({U_i}/T\right) \cong X \times \left(\P^{i-1}\right)^n$.
The projective bundle formula for the motivic cobordism implies that the map
$\m\left(X \stackrel{T}{\times} U_{i+1}\right) \to
\m\left(X \stackrel{T}{\times} U_{i}\right)$ is surjective.

We have seen in \S ~\ref{subsubsection:CAG} that for each $1 \le j \le n$,
$L_{\chi_j}$ defines the line bundle $\sO(\pm 1)$ on each factor of
the product $\left(\P^{i-1}\right)^n$. Let $\zeta_j$ denote the first Chern
class of this line bundle on $\P^{i-1}$. 
Applying Proposition~\ref{prop:Limit} and the non-equivariant projective 
bundle formula once again, we see that
\[
MGL^{a,b}_T(X) = {\underset{p_1, \cdots , p_n  \ge 0} \prod} 
MGL^{a-2(\stackrel{n}{\underset{i=1} \sum} p_i), b-(\stackrel{n}{\underset{i=1} \sum} p_i)}(X) 
\zeta^{p_1}_1 \cdots \zeta^{p_n}_n.
\]
Taking sum over $a \ge b \ge 0$, we see that the map
$\m(X)[[t_1, \cdots , t_n]] \xrightarrow{\cong} \mt(X)$ is an isomorphism
of graded rings. 
\end{proof}

We now prove our main result on  the description of the
equivariant cobordism of smooth and projective schemes with torus action.  

\begin{thm}\label{thm:Main-Str}
Let $T$ be a split torus of rank $n$ acting on a smooth and projective scheme
$X$ and let $i: X^T =  \stackrel{n}{\underset{m=0}\coprod} Z_m \inj X$ 
be the inclusion of the fixed point locus. Then there is a canonical
isomorphism
\[
\stackrel{n}{\underset{m = 0}\bigoplus} 
\mt(Z_m) \xrightarrow{\cong} \mt(X).
\]
of bi-graded $S(T)$-modules.
In particular, there is a canonical isomorphism of bi-graded $S(T)$-modules. 
\begin{equation}\label{eqn:Main-Str1}
\mt(X) \xrightarrow{\cong} \m(X)[[t_1, \cdots , t_n]]
\end{equation}
\end{thm} 
\begin{proof}
Let $\rho = \left(V_i, U_i\right)_{i \ge 1}$ be a canonical admissible gadget for 
$T$.
By inducting on $0 \le m \le n$ and using the homotopy invariance, it follows 
from Lemma~\ref{lem:filter-Equiv} that there is a canonical isomorphism 
\[
\stackrel{n}{\underset{m = 0}\bigoplus} \
{\underset{i}\varprojlim} \  \m\left(Z^{i}_m\right) 
\xrightarrow{\cong}
{\underset{i}\varprojlim} \  \m\left(X^{i}\right).
\]
Applying Lemma~\ref{lem:MLTorus}, we get
\[
\stackrel{n}{\underset{m = 0}\bigoplus} \ \mt(Z_m) \xrightarrow{\cong}
\mt(X).
\]
We have already shown in Proposition~\ref{prop:filter-Gen}
that there is a canonical isomorphism of $\bL$-modules
\[
\stackrel{n}{\underset{m = 0}\bigoplus} \ \m(Z_m) \xrightarrow{\cong}
\m(X).
\]
The last assertion follows from these two isomorphisms and 
Lemma~\ref{lem:trivial-T} since $T$ acts trivially on $X^T$.
\end{proof}

\section{Equivariant motivic cobordism : Another approach}
\label{EMC}
In order to study the topological $K$-theory of the classifying spaces,
Atiyah and Hirzebruch \cite{AH} had defined three different notions of the
topological $K$-theory of the classifying space $BG$ of a Lie group $G$. 
One of these notions is the theory $\sK^*(BG)$, which is defined as the
generalized cohomology of $BG$ given by the topological
$K$-theory spectrum. This is analogous to our equivariant motivic cobordism
$MGL^{*,*}_G(k) = \m(BG)$, discussed above. The other one is the $K$-theory
$k^*(BG)$ which is defined in terms of the projective limit of the usual
$K$-theory of the finite skeleta of the $CW$-complex $BG$.
The relations between these two notions and its applications have been
the subject of study in several works of Atiyah and his coauthors.

Motivated by the above topological construction, we consider a similar approach
in the algebraic context in this section. An outcome of this approach is
that one is able to invent another notion of equivariant motivic cobordism
based the $k^*(BG)$-theory of \cite{AH}. We denote this version of equivariant
motivic cobordism by $mgl^{*,*}_G(-)$. We shall prove certain results
which compare the two notions $MGL^{*,*}_G(-)$ and $mgl^{*,*}_G(-)$.
In particular, we shall show that these two coincide if we consider 
cohomology theory with rational coefficients.
This allows us to compare the equivariant cobordism rings defined in this paper
with the one studied earlier in \cite{Krishna1}.

\subsection{$mgl^{*,*}_G$-theory}\label{subsection:WEMC}
Let $G$ be a linear algebraic group over $k$. The following two results form the
basis for our definition of $mgl^{*,*}_G(X)$ for $X \in \Sm^G_k$.

\begin{lem}\label{lem:Iso-open} 
Let $V \in \Sm_k$ and let $i : W \inj X$ be a closed subset with complement
$j : U \inj X$. Given any $a \ge b \ge 0$, there exists a integer
$N \gg 0$ such that the open pull-back 
$MGL^{a,b}(V) \xrightarrow{j^*} MGL^{a,b}(U)$ is 
an isomorphism if ${\rm codim}_V(W) > N$.
\end{lem}
\begin{proof}
Let $X \mapsto \CH_i(X,j)$ denote the higher Chow groups of Bloch 
\cite{Bloch} for $X \in \Sch_k$. There is an isomorphism
$\CH^i(X,j) \cong \CH_{d-i}(X,j)$ if $X$ is smooth of dimension $d$.
To prove the lemma, we first claim that given any $m \ge 0$, 
the map $\CH^i(V, j) \xrightarrow{j^*} \CH^i(U,j)$
is an isomorphism for all $i \le m$ and for all $j \ge 0$ if 
${\rm codim}_V(W) > m$.  

To see this, we set $d = {\rm codim}_V(W)$ and consider the localization exact 
sequence
\[
\cdots \to \CH_{\dim(X)-i}(W,j) \xrightarrow{i_*} \CH^i(V,j) \xrightarrow{j^*}
\CH^i(U,j) \xrightarrow{\partial} \CH_{\dim(X)-i}(W, j-1) \to \cdots .
\]
Suppose that $d > m$. Then for any $i \le m$, we have $i < d$. This in turn
implies that $\dim(X) - i + j > \dim(W) + j$ for all $j \ge 0$ and 
hence $\CH_{\dim(X)-i}(W,j) = 0$ for all $j \ge 0$. The above exact
sequence shows that the map $j^*$ is an isomorphism for all $j \ge 0$ if
$d > m$. This proves the claim.  

Having proved the claim, we can now use the following motivic 
Atiyah-Hirzebruch spectral
sequence of Hopkins and Morel \cite{Levine1}:
\[
\xymatrix@C1.5pc{
E^{p,a-p}_2(V) \ar[d]^{j^*} & = & \CH^{b-a+p}(V, 2b-a)\otimes \bL^{a-p} \ar[d]^{j^*} &
\implies & MGL^{a,b}(V) \ar[d]^{j^*} \\
E^{p,a-p}_2(U) & = & \CH^{b-a+p}(U, 2b-a)\otimes \bL^{a-p} &
\implies & MGL^{a,b}(U).}
\]
The differentials for this spectral sequence are
given by $d^{p,q}_r : E^{p,q}_r \to E^{p+r, q-r+1}_r$. 

There is a finite filtration 
\[
MGL^{a,b}(V) = F^0MGL^{a,b}(V) \supseteq \cdots \supseteq F^nMGL^{a,b}(V) \supseteq 
F^{n+1} MGL^{a,b}(V) = 0
\]
such that ${F^pMGL^{a,b}(V)}/{F^{p+1}MGL^{a,b}(V)} =
E^{p,a-p}_{n_p}$ for some $n_p \gg 0$. The same holds for $MGL^{a,b}(U)$.
We see from this that there are only finitely many higher Chow groups
which completely determine $MGL^{a,b}(V)$ and $MGL^{a,b}(U)$ for given
integers $a \ge b \ge 0$. Hence by the 
above claim, we can choose some $N \gg 0$ such that all these finitely many
higher Chow groups of $V$ and $U$ are isomorphic if ${\rm codim}_V(W) > N$.
In particular, we get $j^*: E^{p,a-p}_{n_p}(V) \xrightarrow{\cong}
E^{p,a-p}_{n_p}(U)$ for all $0 \le p \le n$. By a descending induction on the
filtration, we get $j^*: MGL^{a,b}(V) \xrightarrow{\cong} MGL^{a,b}(U)$.
\end{proof}

\begin{lem}\label{lem:Ind} 
Let $\rho = \left(V_i, U_i\right)$ and $\rho' = \left(V'_i, U'_i\right)$ 
be two admissible gadgets for $G$ and let $a \ge b \ge 0$.
Then for any $X \in \Sm^G_k$, there is a canonical isomorphism
\[
{\underset{i}\varprojlim} \ MGL^{a,b}\left(X \stackrel{G}{\times} U_i\right) \
\cong  \ {\underset{i}\varprojlim} \
MGL^{a,b}\left(X \stackrel{G}{\times} U'_i\right).
\]
\end{lem}
\begin{proof}
Fix $a \ge b \ge 0$.
For any $i, j \ge 1$, we have the canonical maps
\[
\alpha_{ij}: X \stackrel{G}{\times} (U_i \oplus U'_j) \inj
X \stackrel{G}{\times} (U_i \oplus V'_j) \xrightarrow{p_{ij}} 
X \stackrel{G}{\times} U_i
\]
where the first map is the open immersion and the second map is a vector
bundle. In particular, the map $p^*_{ij}$ is an isomorphism on the motivic
cobordism. 
It follows from the property (iv) of an admissible gadget
({\sl cf.} Definition~\ref{defn:Add-Gad}) and Lemma~\ref{lem:Iso-open} 
that given $i \ge 1$, the map
\begin{equation}\label{eqn:Ind0}
\alpha^*_{ij}: MGL^{a,b}\left(X \stackrel{G}{\times} U_i\right) \to 
MGL^{a,b}\left(X \stackrel{G}{\times} (U_i \oplus U'_j)\right)
\end{equation}
is an isomorphism for all $j \gg 0$. Taking the limit, we see that the
map 
\[
\alpha^*: 
{\underset{i}\varprojlim} \
MGL^{a,b}\left(X \stackrel{G}{\times} U_i\right) \ \to
\ {\underset{i}\varprojlim} \ {\underset{j}\varprojlim} \
MGL^{a,b}\left(X \stackrel{G}{\times} (U_i \oplus U'_j)\right)
\]
is an isomorphism.
By reversing the roles of the admissible gadgets, we see that the map
\[
\beta^*: {\underset{i}\varprojlim} \
MGL^{a,b}\left(X \stackrel{G}{\times} U'_i\right) \ \to
\ {\underset{i}\varprojlim} \ {\underset{j}\varprojlim} \
MGL^{a,b}\left(X \stackrel{G}{\times} (U_i \oplus U'_j)\right)
\]
is also an isomorphism. The map ${\beta^*}^{-1} \circ \alpha^*$ is the desired
isomorphism. This proves the lemma.
\end{proof}

To prove the second part, suppose that the inverse system 
$\left\{MGL^{a,b}\left(X \stackrel{G}{\times} U'_i\right)\right\}$ satisfies
the Mittag-Leffler condition.

Fix $i_0 \ge 1$ and let $s_0 \gg 0$ be such that the map $\alpha^*_{i_0, j}$ is
an isomorphism for all $j \ge s_0$. Let $i_1 \gg i_0$ be such that the map
\[
\beta^*_{i, s_0} :  
MGL^{a,b}\left(X'_{i, s_0}\right) \to 
MGL^{a,b}\left(Y_{i, s_0}\right)
\]
is an isomorphism for all $i \ge i_1$.

Let $s_1 \ge s_0$ be such that for all $j \ge s_1$, we have 
\begin{equation}\label{eqn:Ind1}
{\rm Image} \left(MGL^{a,b}(X'_j) \to MGL^{a,b}(X'_{s_0})\right)
= {\rm Image} \left(MGL^{a,b}(X'_{s_1}) \to MGL^{a,b}(X'_{s_0})\right).
\end{equation}

Let us now fix an integer $i \ge i_1$ and an element $a \in
{\rm Image} \left(MGL^{a,b}(X_{i_1}) \to MGL^{a,b}(X_{i_0})\right)$.
We get a commutative diagram

\begin{defn}\label{defn:WeakEC}
Let $G$ be a linear algebraic group over $k$ and let $X \in \Sm^G_k$.
We let $mgl^{a,b}_G(X)$ be the group
\[
mgl^{a,b}_G(X) = \ {\underset{i}\varprojlim} \
MGL^{a,b}\left(X \stackrel{G}{\times} U_i\right)
\]
where $\rho = \left(V_i, U_i\right)$ is an admissible gadget for $G$.
\end{defn}
It follows from Lemma~\ref{lem:Ind} that $mgl^{a,b}_G(X)$ is well defined.
It also follows from the definition of $MGL^{*,*}_G(X)$ that there is a
natural map
\begin{equation}\label{eqn:STCob}
\phi_X: MGL^{a,b}_G(X) \to mgl^{a,b}_G(X)
\end{equation} 
and this map is surjective by Proposition~\ref{prop:Limit}.
We set $mgl^{*,*}_G(X) = \ {\underset{a \ge b \ge 0}\oplus} \
mgl^{a,b}_G(X)$.

\subsection{Geometric equivariant cobordism}
\label{subsection:MGEC}
Motivated by the work of Quillen \cite{Quillen} on complex
cobordism, Levine and Morel \cite{LM} gave a geometric construction of the 
algebraic cobordism and showed that this is a universal oriented Borel-Moore 
homology theory in $\Sch_k$. Based on the work of Levine and Morel, a 
theory of equivariant algebraic cobordism was constructed in \cite{Krishna1}.
This theory of equivariant cobordism was subsequently used in 
\cite{Krishna2}, \cite{Krishna3}, \cite{KU} and \cite{KK} to compute the
ordinary algebraic cobordism of many classes of smooth schemes. 
We recall the definition of this equivariant cobordism. 

Let $G$ be a linear algebraic group and let $X \in \Sm^G_k$. 
For any integer $i \ge 1$, let $V_i$ be a representation of 
$G$ and let $U_i$ be a $G$-invariant open subset of $V_i$ such that the 
codimension of $V_i \setminus U_i$ is at least $i$ and $G$ acts freely on 
$U_i$ such that the quotient ${U_i}/G$ is a quasi-projective scheme. 
Let ${\Omega^q_G(X)}_i$ denote the quotient 
${\Omega^q(X \stackrel{G}{\times} U_i)}/
{F^i\Omega^q(X \stackrel{G}{\times} U_i)}$, where
$\Omega^q(X)$ is the algebraic cobordism of Levine-Morel and
$F^i\Omega^q(X)$ is the subgroup of $\Omega^q(X)$ generated by
cobordism cycles which are supported on the closed subschemes of $X$ of
codimension at least $i$. It is known that ${\Omega^q_G(X)}_i$ is
independent of the choice of the pair $(V_i,U_i)$ and there is a natural
surjection ${\Omega^q_G(X)}_{i'} \surj {\Omega^q_G(X)}_i$ if $i' \ge i$.
The {\sl geometric} equivariant cobordism group $\Omega^q_G(X)$ is 
defined by
\begin{equation}\label{eqn:ECob} 
\Omega^q_G(X) : = {\underset {i} \varprojlim} \ \Omega^q_G(X)_i.
\end{equation}

It was shown in \cite{Krishna1} that this geometric version of the equivariant 
cobordism has all the properties
of an oriented cohomology theory on $\Sm^G_k$ except that it does not in
general have the localization sequence. We also remark that the
equivariant geometric cobordism $\Omega^*_G(-)$ in \cite{Krishna1} is  
defined for all schemes in $\Sch_k$ and is an example of a Borel-Moore
homology theory.

\subsection{Basic properties of $mgl^{*,*}_G(-)$}\label{subsection:BWECob}
Recall from \cite{Panin2} that the motivic cobordism theory $MGL$
has push-forward maps for projective morphisms between smooth schemes.
We need the following property of this map in order to define the push-forward
map on the $mgl^{*,*}_G(-)$.

\begin{lem}\label{lem:PushFor}
Consider the Cartesian square
\begin{equation}\label{eqn:PFor0}
\xymatrix@C1.2pc{
Y' \ar[r]^{f'} \ar[d]_{g'} & Y \ar[d]^{g} \\
X' \ar[r]_{f} & X}
\end{equation}
in $\Sm_k$ such that $f$ is projective. Suppose that either
\begin{enumerate}
\item
$g$ is a closed immersion and ~\eqref{eqn:PFor0} is transverse, or,
\item
$g$ is smooth. 
\end{enumerate}
One has then 
$g^* \circ f_*  =  f'_* \circ g'^* : MGL^{*,*}(X') \to MGL^{*,*}(Y)$.
\end{lem} 
\begin{proof}
We can write $f = p \circ i$, where $X' \inj \P^n \times X$ is a closed
immersion and $p : \P^n \times X \to X$ is the projection.
This yields a commutative diagram
\[
\xymatrix@C1.2pc{
Y' \ar[r]^<<<<{i'} \ar[d]_{g'} & \P^n \times Y \ar[r]^>>>{p'} \ar[d]^{h} & 
Y \ar[d]^{g} \\
X' \ar[r]_<<<{i} & \P^n \times X \ar[r]_>>>>{p} & X}
\]
where both squares are Cartesian. 

First suppose that $g$ is a closed immersion and ~\eqref{eqn:PFor0} is 
transverse.
Since the right square is transverse 
and so is the big outer square, it follows that the left square is also
transverse. In particular, we have $h^* \circ i_* = i'_* \circ g'^*$
by \cite[Definition~2.2-(2)]{Panin2}. On the other hand, we have
$g^* \circ p_* = p'_* \circ h^*$ by \cite[Definition~2.2-(3)]{Panin2}.
Combining these two, we get
\[
g^* \circ f_* = g^* \circ p_* \circ i_* = p'_* \circ h^* \circ i_*
= p'_* \circ i'_* \circ g'^* = f'_* \circ g'^*
\]
where the first and the last equalities follow from the functoriality
property of the push-forward ({\sl cf.}  \cite[Definition~2.2-(1)]{Panin2}).

Now suppose that $g$ is smooth. For the above proof to go through, only thing
we need to know is that the left square in the above diagram is still
transverse. But this is an elementary exercise using the fact that
$g, h, g'$ are all smooth and $T(g') = i'^*\left(T(h)\right)$, where $T(f)$
denotes the relative tangent bundle of a smooth map $f$. 
This proves the lemma.
\end{proof}

The following result describes the basic properties of
$mgl^{*,*}_G(-)$.

\begin{thm}\label{thm:WeakBasic}
The equivariant cobordism theory $mgl^{*,*}_G(-)$ on $\Sm^G_k$
satisfies the following properties. 
\begin{enumerate}
\item
$Functoriality :$  The assignment $X \mapsto mgl^{*,*}_G(X)$ is
a contravariant functor on $\Sm^G_k$. 
\item
$Push-forward :$ Given a projective map $f : X' \to X$ in $\Sm^G_k$,
there is a push-forward map $f_*: mgl^{a,b}_G(X') \to mgl^{a+2d, b+d}_G(X)$
where $d = \dim(X) - \dim(X')$. 
If $X'' \xrightarrow{f'} X' \xrightarrow{f} X$ are projective, then
$(f \circ f')_* = f_* \circ f'_*$. 
If the square  
\[
\xymatrix@C.7pc{
X' \ar[r]^{g'} \ar[d]_{f'} & X \ar[d]^{f} \\
Y' \ar[r]_{g} & Y}
\]
is transverse in $\sV_G$ where $f$ is a closed immersion, one has 
$g^* \circ f_* = {f'}_* \circ {g'}^* : mgl^{*,*}_G(X) \to mgl^{*,*}_G(Y')$.
\item
$Homotopy \ Invariance:$  If $f : E \to X$ is a $G$-equivariant vector 
bundle, then $f^*:   mgl^{*,*}_G(X) \xrightarrow{\cong} mgl^{*,*}_G(E)$. 
\item
$Chern \ classes :$ For any $G$-equivariant vector bundle $E
\xrightarrow{f} X$ of rank $r$, there are equivariant Chern class operators
$c^G_m(E) : mgl^{*,*}_G(X) \to mgl^{*,*}_G(X)$ for $0 \le m \le r$ with
$c^G_0(E) = 1$. These Chern classes 
have same functoriality properties as in the non-equivariant case. 
Moreover, they satisfy the Whitney sum formula. 
\item
$Exterior \ Product :$ There is a natural product map
\[
mgl^{a,b}_G(X) \otimes_{\Z} mgl^{a',b'}_G(X') \to mgl^{a+a',b+b'}_G(X \times X').
\]
In particular, $mgl^{*,*}_G(X)$ is a bi-graded ring for every $X \in \Sm^G_k$.
\item 
$Projection \ formula :$ For a projective map $f : X' \to X$ in
$\Sm^G_k$, one has for $x \in mgl^{*,*}_G(X)$ and $x' \in mgl^{*,*}_G(X')$,
the formula : $f_*\left(x' \cdot f^*(x)\right) = f_*(x') \cdot x$. 
\item
$Projective \ bundle \ formula :$
For an equivariant vector bundle $E$ of rank $n$ on $X$, the map
\[
\Phi_X : mgl^{*,*}_G(X) \oplus \cdots \oplus mgl^{*,*}_G(X) \to
mgl^{*,*}_G\left(\P(E)\right);
\]
\[
\Phi\left(a_0, \cdots , a_{n-1}\right) \ =
\ \stackrel{n-1}{\underset{i = 0}\sum} \ \pi^*(a_i) \cdot \xi^i
\]
is an isomorphism, where $\pi : \P(E) \to X$ is the projection map  
and $\xi = c^G_1\left(\sO_E(-1)\right)$.
\item
$Change \ of \ groups :$ If $H \subseteq G$ is a closed subgroup and 
$X \in \Sm^G_k$, then there is a natural restriction map $r^G_{H,X}: 
mgl^{*,*}_G(X) \to mgl^{*,*}_H(X)$. 
In particular, there is a natural forgetful map
\begin{equation}\label{eqn:forget}
r^G_X: mgl^{*,*}_G(X) \to \m(X).
\end{equation}
\item
$Morita \ Isomorphism : $
If $H \subseteq G$ is a closed subgroup and $X \in \Sm^H_k$, then there is a 
canonical isomorphism
$mgl^{*,*}_G(X \stackrel{H}{\times} G) \cong mgl^{*,*}_H(X)$.
\item
$Free \ action :$ If $X \in \Sm^G_{{free}/k}$, there is a natural isomorphism
\[
MGL^{*,*}(X/G) \xrightarrow{\cong} mgl^{*,*}_G(X).
\] 
\item
$Comparison \ with \ geometric \ cobordism : $
For any $X \in \Sm^G_k$ and $p \ge 0$, there is a natural isomorphism
$\Omega^p_G(X) \xrightarrow{\cong} mgl^{2p,p}_G(X)$.
\end{enumerate}
\end{thm}
\begin{proof}
The contravariant functoriality of $mgl^{*,*}_G(-)$ follows from the similar
property of motivic cobordism. If $f: X' \to X$ is a projective morphism,
and if $\rho = \left(V_i, U_i\right)$ is an admissible pair for $G$, then
it follows from \cite[Lemma~5.1]{Krishna1} that $f_i :
X' \stackrel{G}{\times} U_i \to X \stackrel{G}{\times} U_i$ is projective
for all $i \ge 1$. Moreover, we have seen in the proof of 
Lemma~\ref{lem:Elem-eq} that the diagram
\[
\xymatrix@C1.8pc{
{X' \stackrel{G}{\times} U_i} \ar[r]^{\iota_{X',i}} \ar[d]_{f_i} &
{X' \stackrel{G}{\times} U_{i+1}} \ar[d]^{f_{i+1}} \\
{X \stackrel{G}{\times} U_i} \ar[r]_{\iota_{X,i}} & 
{X \stackrel{G}{\times} U_{i+1}}}
\]
is transverse, where the horizontal maps are closed immersions. 
It follows from Lemma~\ref{lem:PushFor} that 
$\iota^*_{X,i} \circ (f_{i+1})_* = (f_i)_* \circ \iota_{X',i}^*$.
Taking the inverse limit of $(f_i)_*$, we get the push-forward map
$f_* : mgl^{a,b}_G(X') \to mgl^{a+2d,b+d}_G(X)$. The other part of property (2)
follows from the similar property of the motivic cobordism on
$\Sm_k$ by \cite[Definition~2.2-(2)]{Panin2} and Lemma~\ref{lem:PushFor}.

The proof of properties (3) through (6) is same as the proof of
\cite[Theorem~5.2]{Krishna1} and the proof of the projective bundle
formula follows directly from the similar result in the non-equivariant case.
The proof of Property (8) is straight forward and the proof of 
property (9) follows like \cite[Proposition~5.4]{Krishna1}.

To prove property (10), we observe in the case of free action of $G$ on $X$
that there are maps $X \stackrel{G}{\times} U_i \to
 X \stackrel{G}{\times} V_i \to X/G$ in $\Sm_k$. It follows from
the homotopy invariance of the motivic cobordism, Lemma~\ref{lem:Iso-open}
and the fourth property of an admissible gadget that the induced map
$MGL^{*,*}(X/G) \to MGL^{*,*}\left(X \stackrel{G}{\times} U_i\right)$ is
an isomorphism for all $i \gg 0$. In particular, the map
\[
MGL^{*,*}(X/G) \to \ {\underset{i}\varprojlim} \
MGL^{*,*}\left(X \stackrel{G}{\times} U_i\right)
\]
is an isomorphism. The property (11) follows directly from from
\cite[Theorem~3.1]{Levine1}. 
\end{proof}

\begin{remk}\label{remk:Local-NO}
We have seen in Theorem~\ref{thm:BPEC*} that the equivariant motivic 
cobordism theory $MGL^{*,*}_G(-)$ has localization sequences. However,
the proof of this localization sequence fails in the case of $mgl^{*,*}_G(-)$.
In fact, it was shown by Buh{\v{s}}taber and Mi{\v{s}}{\v{c}}enko 
\cite{BM2} that if $k^*(X)$ is the projective limit of the topological 
$K$-theory of the finite skeleta of a $CW$-complex $X$, then 
$k^*(-)$ does not satisfy the localization sequence. Using the results of
Buh{\v{s}}taber-Mi{\v{s}}{\v{c}}enko and Landweber \cite{Landweber}, one
can find such an example also for the complex cobordism.
Because of this, one does not expect localization sequence to be true for 
$mgl^{*,*}_G(-)$. This is one serious drawback of this theory.
We shall show however that the localization sequence does hold for 
$mgl^{*,*}_G(-)$ with the rational coefficients.
\end{remk}

\subsection{Comparison of $\mg(-)$ and $mgl^{*,*}_G(-)$  theories}
\label{subsection:MGL-mgl}
We have seen before that there is a natural transformation
of contravariant functors $\mg(-) \to mgl^{*,*}_G(-)$ on $\Sm^G_k$.
When it comes to computing the equivariant cobordism $\mg(X)$, it is 
often desirable to have these two functors isomorphic for $X$. Although we 
can not expect this to be the case in general, we have the following
version of Lemma~\ref{lem:MLTorus} in the case of torus action.
We shall show later in this text that this isomorphism always holds
with rational coefficients.
 
\begin{cor}\label{cor:MLTorus**}
Let $T$ be a split torus and let $X \in \Sm^T_k$ be smooth and projective.  
Then for any $a \ge b \ge 0$, the map 
\[
\phi_X : MGL^{a,b}_T(X) \to mgl^{a,b}_T(X)
\]
is an isomorphism. 
\end{cor}

\section{Equivariant cobordism with rational coefficients}
\label{section:ECRC}
In this section, we study the equivariant motivic cobordism with
rational coefficients and show in this case that the equivariant
cobordism of a smooth scheme with a group action can be computed
in terms of the limit of the ordinary motivic cobordism groups of
smooth schemes. This allows us to show in particular that if $G$ is a
connected reductive group, then the equivariant cobordism of any
$G$-scheme can be written in terms of the Weyl group invariants of the
equivariant cobordism of the given scheme for the action of a maximal torus.
An equivariant analogue of the Levine-Morel algebraic cobordism was
studied in \cite{Krishna1}. We also give the precise relation between the
two versions of equivariant cobordism in this section.

Recall from \cite[Remark~III.6.5]{Adams} that given an abelian group $R$,
there is a Moore space $M_R \in {\rm HoSsets}_{\bullet}$, where
${\rm HoSsets}_{\bullet}$ is the unstable homotopy category of pointed simplicial
sets. We can consider $M_R$ as an object of $\sH_{\bullet}(k)$ via the 
obvious functor ${\rm HoSsets}_{\bullet} \to \sH_{\bullet}(k)$.
For any spectrum $E \in \sS \sH (k)$, there is a  Moore spectrum 
$E_R = E \wedge M_R$.

\begin{defn}\label{defn:Coeff}
Let $G$ be a linear algebraic group over $k$. For any $X \in \Sm^G_k$,
the equivariant motivic cobordism of $X$ with coefficients in the
group $R$ is given by    
\begin{equation}\label{eqn:EMCR}
MGL^{a,b}_G(X; R) =  MGL^{a,b}(X_G; R) : = 
\Hom_{\sS\sH(k)}\left(\Sigma^{\infty}_T X_{+}, \Sigma^{a,b}{MGL}_{R}\right)
\end{equation}
where $X_G$ is a Borel space of the type $X_G(\rho)$ as in 
\S~\ref{subsection:Borel}. 
We also consider $mgl^{*,*}_G(-; R)$:
\[
mgl^{a,b}_G(X; R) = \ {\underset{i}\varprojlim} \
MGL^{a,b}\left(X \stackrel{G}{\times} U_i ; R \right).
\] 
\end{defn}
We set 
\[
MGL^{*,*}_G(X; R) = 
{\underset{0 \le b \le a}\oplus} \ MGL^{a,b}_G(X; R) \ {\rm and} \
mgl^{*,*}_G(X; R) = {\underset{0 \le b \le a}\oplus} \ mgl^{a,b}_G(X; R). 
\]
Note that since every $X \in \Sm_k$ is compact as a motivic space, 
it follows that there is a short exact sequence
\begin{equation}\label{eqn:Coeff*}
0 \to MGL^{a,b}(X) \otimes_{\Z} R \to MGL^{a,b}(X ; R) \to 
{\rm Tor}\left(MGL^{a+1,b}(X) , R\right) \to 0.
\end{equation}
In particular, for any $R \subseteq \Q$, the
natural map $MGL^{a,b}(X) \otimes_{\Z} R \to MGL^{a,b}(X; R)$ is an isomorphism.
But this is no longer true for the equivariant motivic cobordism because of the
fact that the spaces $X_G$ are not in general compact.
The following result gives a simple description of the rational equivariant
motivic cobordism.

\begin{thm}\label{thm:ECRat}
Let $G$ be a linear algebraic group and let $X \in \Sm^G_k$. Then for
all $a \ge b \ge 0$, the natural map 
\[
MGL^{a,b}_G(X; \Q) \to mgl^{a,b}_G(X; \Q)
\]
is an isomorphism. 
\end{thm}
\begin{proof}
Let $\rho = \left(V_i, U_i\right)$ be an admissible gadget for $G$.
In view of Proposition~\ref{prop:Limit}, it is enough to show that for any
$0 \le b \le a$, the inverse system 
$\{MGL^{a,b}\left(X_G(\rho,i); \Q \right)\}_{i \ge 1}$ satisfies the Mittag-Leffler
condition.

Let us denote the smooth scheme $X_G(\rho,i)$ in short by $X_i$ and let
$d_i$ denote the dimension of $X_i$.  Recall that $X_G = colim_i \ X_i$.
It follows from \cite[Corollary~10.6]{NSO} that for any $i \ge 1$, there is a
natural isomorphism 
\begin{equation}\label{eqn:ECRat0}
MGL^{a,b}(X_i; \Q) \ \xrightarrow{\cong} \
\stackrel{d_i + 2b-a}{\underset{j = b}\oplus} \
H^{2j+a-2b, j}(X_i; \Q) \otimes_{\Q} \bL^{b-j}_{\Q}
\end{equation}
\[
\hspace*{3.5cm} \xrightarrow{\cong} 
\stackrel{d_i + 2b-a}{\underset{j = b}\oplus} \
\CH^j(X, 2b-a ; \Q) \otimes_{\Q} \bL^{b-j}_{\Q},
\]
where $H^{*,*}(X)$ is the motivic cohomology of $X \in Spc$ given by the motivic
Eilenberg-MacLane spectrum and $\bL_{\Q}$ is the rationalization of the Lazard
ring $\bL$.

We fix an integer $i \ge 1$.
It follows from the localization sequence for the motivic cohomology and
the third property of an admissible gadget 
({\sl cf.} Definition~\ref{defn:Add-Gad}) that for any integer
$j \in [b, d_i+2b-a]$, there exists $m(i,j) \gg i$ such that for all $l \ge 
m(i,j)$, the restriction map
$H^{d_i+2b-a, j}\left(X_{l+1} ; \Q\right) \to 
H^{d_i+2b-a, j}\left(X \stackrel{G}{\times} 
\left(U_l \oplus W_l\right) ; \Q\right)$ is an isomorphism. 
On the other hand, the homotopy invariance property of the motivic cohomology
shows that the map 
$H^{d_i+2b-a, j}\left(X \stackrel{G}{\times} 
\left(U_l \oplus W_l\right) ; \Q\right) \to H^{d_i+2b-a, j}\left(X_l; \Q\right)$
is also an isomorphism. Setting 
$m(i) = \stackrel{d_i+2b-a}{\underset{j= b}{\rm max}}\ m(i,j)$,
we see that for all $l \ge m(i)$, the restriction map
\[
\stackrel{d_i + 2b-a}{\underset{j = b}\oplus} \
H^{2j+a-2b, j}(X_l; \Q) \ \to \
\stackrel{d_i + 2b-a}{\underset{j = b}\oplus} \
H^{2j+a-2b, j}(X_{m(i)}; \Q)
\]
is an isomorphism. 
Since $i \mapsto d_i$ is an strictly increasing function, it follows from
~\eqref{eqn:ECRat0} that the image of the restriction map
$MGL^{a,b}(X_l ; \Q) \to MGL^{a,b}(X_i ; \Q)$ does not depend on the choice
of $l \ge m(i)$. In other words, the inverse system 
$\{MGL^{a,b}\left(X_i; \Q \right)\}_{i \ge 1}$ satisfies the Mittag-Leffler
condition.
\end{proof}

We have seen in Remark~\ref{remk:Local-NO} that the $mgl^{*,*}_G(-)$-theory 
is not expected to satisfy the localization sequence with the integral
coefficients. 
The following result rectifies this
problem if one is working with the rational coefficients.

\begin{cor}\label{cor:Mot-Geom}
Let $\iota : Y \inj X$ be a closed immersion in $\Sm^G_k$  of codimension $d$.
There is then a long exact localization sequence
\[
\cdots \to mgl^{a-2d,b-d}_G(Y; \Q) \xrightarrow{\iota_*} 
mgl^{a,b}_G(X; \Q) \to mgl^{a,b}_G(X \setminus Y; \Q)
\]
\[
\hspace*{9cm} \xrightarrow{\partial}
mgl^{a-2d+1,b-d}_G(Y; \Q) \to \cdots .
\]
In particular, there is an exact sequence
\[
\Omega^{q-d}_G(Y; \Q) \to \Omega^q_G(X; \Q) \to \Omega^q_G(X \setminus Y ; \Q)
\to 0.
\]
for all $q \ge 0$.
\end{cor}  
\begin{proof}
The first long exact sequence follows from Theorems~\ref{thm:BPEC*} (6) and
~\ref{thm:ECRat}. The second exact sequence follows from the first,
together with Theorem~\ref{thm:WeakBasic} (11) and 
\cite[Proposition~5.3]{Krishna1}.
\end{proof}

\section{Reduction of arbitrary groups to tori}\label{section:RedT}
The theme of this section is to study the question of how to reduce the
problem of computing the equivariant motivic cobordism for the action of
a linear algebraic group $G$ to the case when the underlying group is a torus.
We prove various results in this direction and compute the motivic cobordism
of the classifying spaces of some reductive groups. We begin with the
following result.

\begin{prop}\label{prop:Borel-torus}
Let $G$ be a connected reductive group over $k$. Let $B$ be a Borel 
subgroup of $G$ containing a maximal torus $T$ over $k$. Then for any
$X \in \Sm^G_k$, the restriction map
\begin{equation}\label{eqn:Borel2}
\mbb\left(X\right) \xrightarrow{r^B_{T,X}} \mt\left(X\right)
\end{equation}
is an isomorphism.
\end{prop}
\begin{proof}
By the Morita Isomorphism of Theorem~\ref{thm:BPEC*},  we only need to show 
that
\begin{equation}\label{eqn:Borel1}
\mbb\left(B \stackrel{T}{\times} X\right) \cong
\mbb\left(X\right).
\end{equation}
By \cite[XXII, 5.9.5]{SGA3}, there exists a characteristic
filtration $B^u = U_0 \supseteq U_1 \supseteq \cdots \supseteq U_n =
\{1\}$ of the unipotent radical $B^u$ of $B$ such that ${U_{i-1}}/{U_i}$
is a vector group, each $U_i$ is normal in $B$ and $TU_i = T \ltimes U_i$. 
Moreover, this filtration also implies that for each $i$, the natural map
$B/{TU_i} \to B/{TU_{i-1}}$ is a torsor under the vector bundle
${U_{i-1}}/{U_i} \times B/{TU_{i-1}}$ on $B/{TU_{i-1}}$. Hence, the 
homotopy invariance ({\sl cf.} Theorem~\ref{thm:BPEC*}) gives an isomorphism
\[
\mbb\left(B/{TU_{i-1}} \times X\right) \xrightarrow{\cong}
\mbb\left(B/{TU_i} \times X\right).
\]
Composing these isomorphisms successively for $i = 1, \cdots ,n$, we get
\[
\mbb\left(X\right) \xrightarrow{\cong}
\mbb\left(B/T \times X\right).
\]
The canonical isomorphism of $B$-varieties $B \stackrel{T}{\times} X \cong
(B/T) \times X$ 
now proves ~\eqref{eqn:Borel1} and hence ~\eqref{eqn:Borel2}.
\end{proof}

\begin{prop}\label{prop:NRL}
Let $G$ be a possibly non-reductive group over $k$.
Let $G = H \ltimes G^u$ be the Levi decomposition of $G$ 
(which exists since $k$ is of characteristic zero).
Then the restriction map 
\begin{equation}\label{eqn:NRL0}
{\mg\left(X\right)} \xrightarrow{r^G_{H,X}} 
{\mh\left(X\right)}
\end{equation}
is an isomorphism.
\end{prop}
\begin{proof}
Since the ground field is of characteristic zero, the
unipotent radical $G^u$ of $G$ is split over $k$. Now the proof is exactly
same as the proof of Proposition~\ref{prop:Borel-torus}, where we just have to 
replace $B$ and $T$ by $G$ and $H$ respectively.
\end{proof}

\subsubsection{Action of Weyl group}\label{subsubsection:Weyl}
Let $G$ be a linear algebraic group and let $H \subseteq G$ be a closed normal
subgroup with quotient $W$. If $\rho = \left(V_i, U_i\right)_{i \ge 1}$ is an 
admissible gadget for $G$, then it is also an admissible gadget for $H$.  
If $X \in \Sm^G_k$, then $W$ acts on the ind-scheme $X_G(\rho) =
colim_i \ (X \stackrel{H}{\times} U_i)$ such that each closed 
subscheme $X \stackrel{H}{\times} U_i$ is $W$-invariant. In particular, $W$
acts on $MGL^{*,*}\left(X_G(\rho)\right) = \mg(X)$. One example where such a
situation occurs is when $H$ is a maximal torus of a reductive group and
$G$ is its normalizer. The quotient $W$ is then the Weyl group of $G$.
In that case, $\mg(X)$ becomes a $\Z[W]$-module.

\subsection{Rational results}\label{section:RResult}
Let $G$ be a connected reductive group and let $T$ be a split maximal torus
of $G$. Let $N$ be the normalizer of $T$ in $G$ with the associated Weyl 
group $W = N/T$. We first consider the case of equivariant cobordism
with rational coefficients and give the most complete result in this case.

\begin{thm}\label{thm:Weyl-Rat}
Let $G$ be a connected reductive group and let $T$ be a split maximal torus of
$G$ with the Weyl group $W$. Then for any $X \in \Sm^G_k$, the restriction map
$r^G_{T, X} : \mg(X) \to \mt(X)$ induces an isomorphism
\[
\mg(X;\Q) \xrightarrow{\cong} \left(\mt(X;\Q)\right)^W.
\]
\end{thm}
\begin{proof}
In view of Theorem~\ref{thm:ECRat}, we can replace $MGL^{*,*}_G(X; \Q)$ with
$mgl^{*,*}_G(X;\Q)$. Using the definition of $mgl^{*,*}_G(X;\Q)$ and the fact that
the inverse limit commutes with taking $W$-invariants, it suffices
to show that for an admissible gadget $\rho = \left(V_i, U_i\right)_{i \ge 1}$ 
for $G$, the map $\m\left(X \stackrel{G}{\times} U_i ; \Q\right) \to
\left(\m\left(X \stackrel{T}{\times} U_i; \Q \right)\right)^W$ is an isomorphism
for all $i \ge 1$. But this follows at once from ~\eqref{eqn:ECRat0} and
\cite[Corollary~8.7]{Krishna4}.
\end{proof}

\subsection{Motivic cobordism of classifying spaces}
\label{subsection:CCS}
Let $G$ be a connected reductive group over $k$ with a split maximal torus $T$
and the Weyl group $W$. In such a case, one says that $G$ is {\sl split}
reductive. Let $B$ denote a Borel subgroup of $G$ containing the
maximal torus $T$. Let $d$ denote the dimension of the flag variety $G/B$.
In this case, every character $\chi$ of $T$ yields
a line bundle $L_{\chi}$ on the classifying space $BT$ which restricts to a
line bundle $\sL_{\chi}:=G\times^B L_{\chi}$ on the flag variety $G/B$ via the 
maps
\begin{equation}\label{eqn:CCS0}
G/B \stackrel{\iota}{\inj} BT \xrightarrow{\pi} BG
\end{equation}
in ${\bf ISm}_k$.
Recall that the {\em torsion index} of $G$ is defined as the smallest positive 
integer $t_G$ such that $t_G$ times the class of a point in 
$H^{2d}(G/B,\Z)$ belongs to the subring of 
$H^*(G/B,\Z)$ generated by the first Chern classes of line bundles 
$\sL_{\chi}$ (e.g., $t_G=1$ for $G=GL_n$, see \cite{Totaro2} for computations 
of $t_G$ for other groups).
If $G$ is simply connected then this subring is generated by $H^2(G/B,\Z)$.
We shall denote the ring $\Z[t^{-1}_G]$ by $R$. 

Let $X \in \Sm^G_k$ and let $p_X: X \to \Spec(k)$ denote the structure map.
Let $\rho = \left(V_i, U_i\right)_{i \ge 1}$ be an admissible gadget for $G$. 
Set $X^i_G = X \stackrel{G}{\times} U_i$ and
$X_G = X_G\left(\rho\right)$. This gives rise to the commutative
diagram in 
\begin{equation}\label{eqn:CCS1}
\xymatrix@C1.5pc{
G/B \ar[r]^{\iota_X} \ar@{=}[d] & X_B \ar[r]^{\pi_X} \ar[d]_{p_{B, X}} &
X_G \ar[d]^{p_{G, X}} \\
G/B \ar[r]_{\iota} &  BB \ar[r]_{\pi} & BG}
\end{equation}
in ${\bf ISm}_k$, where $\iota$ and $\iota_X$ are strict closed embeddings.
A similar commutative diagram exists for each $X^i_G$. This in turn yields
a commutative diagram

\begin{equation}\label{eqn:CCS2}
\xymatrix@C1.5pc{
G/B \ar[r]^{\iota^i_X} \ar@{=}[dd] & 
X^i_B \ar[rr]^{\pi^i_X} \ar[dd]_{p^i_{B,X}} \ar[dr] & & X^i_G \ar[dr] 
\ar[dd]^>>>>>>{p^i_{G,X}} & \\
& & X_B \ar[rr]^<<<<<<{\pi_X} \ar[dd]_<<<<<{p_{B,X}} & & X_G \ar[dd]^{p_{G,X}} \\
G/B \ar[r]^{\iota^i} \ar[drr]_{\iota} & 
{U_i}/B \ar[rr]^>>>>>>>{\pi^i} \ar[dr] & & {U_i}/G \ar[dr] & \\
& & BB \ar[rr]_{\pi} & & BG.}
\end{equation}

\begin{lem}\label{lem:ELM-Com}
Let $f:X \to Y$ be a morphism in $\Sm^G_{{free}/k}$. Then the diagram of quotients
\[
\xymatrix@C.9pc{
X/B \ar[r] \ar[d] & X/G \ar[d] \\
Y/B \ar[r] & Y/G}
\]
is Cartesian such that the horizontal maps are smooth and projective.
If $f$ is a closed immersion, then this diagram is transverse.
\end{lem}
\begin{proof}
The top horizontal map is an {\'e}tale locally trivial smooth fibration with 
fiber $G/B$. Hence this map is proper by the descent property of properness.
Since this map is also quasi-projective, it must be projective. The same holds
for the bottom horizontal map.
Proving the other properties is an elementary exercise and can be shown
using the commutative diagram
\begin{equation}\label{eqn:ELM-Com0}
\xymatrix@C1.6pc{
X \ar[r] \ar[d]_{f} & X/B \ar[r] \ar[d] & X/G \ar[d] \\
Y \ar[r] & Y/B \ar[r] & Y/G.}
\end{equation}
One easily checks that the left and the big outer squares are Cartesian
and transverse if $f$ is a closed immersion. 
Since all the horizontal maps are smooth and surjective, the right square
must have the similar property.
\end{proof}

\begin{prop}\label{prop:Push-Pull}
For any $X \in \Sm^G_k$, there is a push-forward map 
\[
(\pi_X)_*: mgl^{a,b}_T(X) \to mgl^{a-2d, b-d}_G(X)
\]
which is contravariant for smooth maps and covariant for projective maps 
in $\Sm^G_k$. This map satisfies the projection formula
$(\pi_X)_*\left(x \cdot r^G_{T,X}(y)\right) = (\pi_X)_*(x) \cdot y$
for $x \in mgl^{*,*}_T(X)$ and $y \in mgl^{*,*}_G(X)$.
\end{prop} 
\begin{proof}
Let $\rho = \left(V_i, U_i\right)_{i \ge 1}$ be an admissible gadget for $G$.
For any $j \ge i \ge 1$, Lemma~\ref{lem:ELM-Com} yields 
a transverse square
\begin{equation}\label{eqn:ELM-Com1}
\xymatrix@C1.6pc{
X^i_B \ar[r]^{\pi^i_X} \ar[d]_{s^{i,j}_{B, X}} &
X^i_G \ar[d]^{s^{i,j}_{G, X}} \\
X^j_B \ar[r]_{\pi^j_X} & X^j_G}
\end{equation}
where the vertical maps are closed immersions and the horizontal maps are
smooth and projective. It follows from Lemma~\ref{lem:PushFor} that there
is a projective system of push-forward maps 
$\left\{(\pi^i_X)_*:  mgl^{a,b}(X^i_B) \to mgl^{a-2d, b-d}(X^i_G)\right\}$.
Taking the limit, we get the desired map 
$(\pi_X)_*: mgl^{a,b}_B(X) \to mgl^{a-2d, b-d}_G(X)$ and we can replace
$T$ by $B$ using Proposition~\ref{prop:Borel-torus}.
The covariant functoriality is obvious and the contravariant functoriality
follows directly from Lemmas~\ref{lem:ELM-Com} and ~\ref{lem:PushFor}.
The projection formula for $(\pi_X)_*$ and $r^G_{T,X}$ follows from the
projection formula for the maps $X^i_B \xrightarrow{\pi^i_X} X^i_G$ and
observing that $r^G_{T,X}$ is the inverse limit of the pull-back maps
$(\pi^i_X)^*$. 
\end{proof}

\begin{lem}\label{lem:ML-Gen}
Let $X \in \Sm^G_k$ be projective and let 
$\rho = \left(V_i, U_i\right)_{i \ge 1}$ be an admissible gadget for $G$. 
Then for any $a \ge b \ge 0$,
the projective system $\left\{MGL^{a,b}(X^i_G;R)\right\}_{i \ge 1}$ satisfies the 
Mittag-Leffler condition.
\end{lem}
\begin{proof}
It was shown in \cite[Proposition~4.8, Lemma~4.2]{KK} that there is an
element $a_0 \in mgl^{2d,d}(BB;R)$ such that $\iota^*(a_0)$ is the class of a 
rational point in $MGL^{2d,d}(G/B;R) = \Omega^d(G/B;R)$. 
Moreover, it was also shown in \cite[Proposition~4.8]{KK} that 
$\alpha_0 = \pi_*(a_0) \in mgl^{*,*}(BG;R)$ is an
invertible element, where $\pi_* : = (\pi_k)_*$ is the push-forward map of
Proposition~\ref{prop:Push-Pull}. 

Setting $a = p^*_{B,X}(a_0)$, we see that there is a class $a \in
mgl^{2d,d}_B(X;R)$ such that $\iota^*_X(a)$ is the class of a 
rational point in $MGL^{2d,d}(G/B;R) = \Omega^d(G/B;R)$. 
It follows moreover from Proposition~\ref{prop:Push-Pull} that
$\alpha = (\pi_X)_*(a) \in mgl^{*,*}_G(X;R)$ is an invertible element.

For any $j \ge i \ge 1$, we have a diagram
\begin{equation}\label{eqn:ML-Gen0}
\xymatrix@C1.6pc{
\m(X^j_G;R) \ar[r]^{(\pi^j_X)^*} \ar[d]_{(s^{i,j}_{G,X})^*} & 
\m(X^j_B;R) \ar[r]^{(\pi^j_X)_*} \ar[d]^{(s^{i,j}_{B,X})^*}  & 
\m(X^j_G;R) \ar[d]^{(s^{i,j}_{G,X})^*} \\
\m(X^i_G;R) \ar[r]_{(\pi^i_X)^*} & \m(X^i_B;R) \ar[r]_{(\pi^i_X)_*} & \m(X^i_G;R).}
\end{equation} 
The left square clearly commutes and the right square commutes by 
Lemma~\ref{lem:PushFor}. Notice that the map $\m(X^i_B;R) \to \m(X^i_T;R)$
is an isomorphism as shown in Proposition~\ref{prop:Borel-torus}.
Setting $s^i_{G,X} : X^i_G \inj X_G$ and taking limit over $j \ge i$, we see that 
$(s^i_{G,X})^* \circ (\pi_X)_* = (\pi^i_X)_* \circ (s^i_{B,X})^*$ for every
$i \ge 1$. Letting $a_i = (s^i_{B,X})^*(a)$ and $\alpha_i = (s^i_{G,X})^*(\alpha)$,
we conclude that $(\pi^i_X)_*(a_i) = \alpha_i$ for every $i \ge 1$.
Moreover, $\alpha_i \in \m(X^i_G;R)$ is invertible.

To verify the required Mittag-Leffler condition, fix any $i \ge 1$.
It follows from Lemma~\ref{lem:MLTorus} that there exists $i' \gg i$
such that 
\begin{equation}\label{eqn:ML-Gen1}
{\rm Image}\left((s^{i,j}_{B,X})^*\right) = 
{\rm Image}\left((s^{i,i'}_{B,X})^*\right) \ \  {\rm for \ all} \ 
j \ge i'.
\end{equation}

Given any $x \in \m(X^{i'}_G;R)$, we get for any $j \ge i'$:
\[
\begin{array}{lll} 
(s^{i,i'}_{G, X})^* \left(\alpha_{i'} \cdot x\right) & {=}^{\dagger} &
(s^{i,i'}_{G, X})^* \circ (\pi^{i'}_X)_*
\left(a_{i'} \cdot (\pi^{i'}_X)^*(x)\right) \\
& = &  (\pi^{i}_X)_* \circ (s^{i,i'}_{B, X})^* 
\left(a_{i'} \cdot (\pi^{i'}_X)^*(x)\right) \\
& {=}^{\dagger \dagger} & (\pi^{i}_X)_* \circ (s^{i,j}_{B, X})^*(y) \\
& = & (s^{i,j}_{G, X})^* \circ (\pi^{j}_X)_*(y),
\end{array}
\]
where the first equality is from the projection formula and the second
equality follows from ~\eqref{eqn:ML-Gen1} for some $y \in \m(X^j_B;R)$.
Since $\alpha_{i'} \in \m(X^{i'}_G;R)$ is invertible, we see that
${\rm Image}\left((s^{i,i'}_{G,X})^*\right) \subseteq  
{\rm Image}\left((s^{i,j}_{B,X})^*\right)$. Since the other inclusion is 
obvious, this verifies the desired Mittag-Leffler condition.  
\end{proof}

As an immediate consequence of Lemma~\ref{lem:ML-Gen} and 
Proposition~\ref{prop:Limit}, we get the following generalization of
Corollary~\ref{cor:MLTorus**}.

\begin{cor}\label{cor:MLG**}
Let $G$ be a connected and split reductive group over $k$ and let
$X \in \Sm^G_k$ be projective. Then for any $a \ge b \ge 0$, the
map
\[
\phi_X: MGL^{a,b}_G(X;R) \to mgl^{a,b}_G(X;R)
\]
is an isomorphism.
\end{cor}

\begin{lem}\label{lem:Push-Rat}
Let $G$ be a connected and split reductive group over $k$ and let
$X \in \Sm^G_k$ be projective. Let $a \in mgl^{2d,d}(X; R)$ be the
element constructed in the proof of Lemma~\ref{lem:ML-Gen} and let
$\theta_X : mgl^{a,b}_T(X;R) \to mgl^{a,b}_G(X;R)$ be the map
$x \mapsto (\pi_X)_*(a \cdot x)$. Then the maps
\[
\left(mgl^{a,b}_T(X; \Q)\right)^W \xrightarrow{\theta_X} 
mgl^{a,b}_G(X; \Q) \xrightarrow{r^G_{T,X}} \left(mgl^{a,b}_T(X; \Q)\right)^W 
\]
are isomorphisms.
\end{lem}
\begin{proof}
We shall follow the notations that we used in the proof of  
Lemma~\ref{lem:ML-Gen}.
Let $\rho = \left(V_i, U_i\right)_{i \ge 1}$ be an admissible gadget for $G$. 
For any $i \ge 1$, let $a_i = (s^i_{B,X})^*(a)$ and let 
$\theta^i_X: MGL^{a,b}(X^i_B) \to MGL^{a,b}(X^i_G)$
be the map $x \mapsto (\pi^i_X)_*(a_i \cdot x)$. 
It suffices to show that the maps
\begin{equation}\label{eqn:Push-Rat0}
\left(MGL^{a,b}(X^i_B; \Q)\right)^W \xrightarrow{\theta^i_X} 
MGL^{a,b}(X^i_G; \Q) \xrightarrow{(\pi^i_X)^*} \left(MGL^{a,b}(X^i_B; \Q)\right)^W 
\end{equation}
are isomorphisms for every $i \ge 1$.

Let $\wt{a_i} \in H^{2d,d}(X^i_B;\Q)$ be the image of the element
$a_i \in MGL^{2d,d}(X^i_B;\Q)$ under the natural transformation of functors
$MGL^{*,*}(-) \to H^{*,*}(-)$. By ~\eqref{eqn:ECRat0}, it suffices to show that
the map
\begin{equation}\label{eqn:Push-Rat1}
\left(H^{a,b}(X^i_B; \Q)\right)^W \xrightarrow{\theta^i_X} 
H^{a,b}(X^i_G; \Q) ; \ \ x \mapsto (\pi^i_X)_*(\wt{a_i} \cdot x)
\end{equation} 
is an isomorphism.

It follows from \cite[Lemma~6.4]{Krishna4} that there is a natural
isomorphism 
\[
\phi^i_X: H^*(G/B, \Q) \otimes_{\Q} H^{*,*}(X^i_G; \Q)
\xrightarrow{\cong} H^{*,*}(X^i_B, \Q),
\]
where $H^*(X) = {\underset{i \ge 0}\oplus} \ H^{2i,i}(X)$. 
Let $\iota_i : G/B \to X^i_B$ be the fiber of the map $X^i_B 
\xrightarrow{\pi^i_X} X^i_G$.
It is well known that the map $\left(H^*(G/B;\Q)\right)^W \to \Q$, given by 
$x \mapsto (p_{G/B})_*(\iota^*_i(a_i) \cdot x)$ is an isomorphism.
It follows from the construction of the map $\phi^i_X$ 
({\sl cf.} \cite[Lemma~6.4]{Krishna4}) that the map in
~\eqref{eqn:Push-Rat1} is an isomorphism.
\end{proof}

\begin{prop}\label{prop:Weyl-Inv}
Let $G$ be a connected reductive group with a split maximal torus $T$ and
the associated Weyl group $W$. Then for any $X \in \Sm^G_k$ which is 
projective, the map 
\[
r^G_{T,X} : \mg(X;R) \to \left(\mt(X;R)\right)^W
\]
is injective.
\end{prop}
\begin{proof}
In view of Corollary~\ref{cor:MLG**}, we can replace $\mg(X;R)$ and
$\mt(X;R)$ by $mgl^{*,*}_G(X;R)$ and $mgl^{*,*}_T(X;R)$
respectively. Let $B$ be a Borel subgroup of $G$ containing the maximal
torus $T$.  
Using the definition of $mgl^{*,*}_G(X;R)$ and the fact that
the inverse limit commutes with taking $W$-invariants, it suffices
to show that for an admissible gadget $\rho = \left(V_i, U_i\right)_{i \ge 1}$ 
for $G$, the map 
\[
(\pi^i_X)^*: \m(X^i_G ; R) \to \left(\m(X^i_B; R)\right)^W
\]
is injective for all $i \ge 1$.

To show this, let $x \in \m(X^i_G ; R)$. 
Let $a_i \in \m(X^i_B;R)$ and $\alpha_i \in \m(X^i_G;R)$ be as
in the proof of Lemma~\ref{lem:ML-Gen}. 
Using the projection formula, we see that
$(\pi^i_X)_*\left(a_i \cdot (\pi^i_X)^*(x)\right) = \alpha_i \cdot x$.
Since $\alpha_i \in \m(X^i_G;R)$ is invertible, we see that $(\pi^i_X)^*$ is 
injective.   
\end{proof}

For any $X \in \Sm_k$, let $MGL^{*}(X) = {\underset{i \ge 0}\oplus} \
MGL^{2i,i}(X)$ and for any $X \in \Sm^G_k$, let 
$MGL^{*}_G(X) = {\underset{i \ge 0}\oplus} \ MGL^{2i,i}_G(X)$.
We can prove a stronger form of Proposition~\ref{prop:Weyl-Inv} in the
following case.

\begin{thm}\label{thm:Weyl-Inv*}
Let $G$ be a connected reductive group with a split maximal torus $T$ and
the associated Weyl group $W$. Let $X \in \Sm^G_k$ be projective such that
$T$ acts on $X$ with only finitely many fixed points. Then the map 
\[
r^G_{T,X} : MGL^{*}_G(X;R) \to \left(MGL^*_T(X;R)\right)^W
\]
is an isomorphism.
\end{thm}
\begin{proof}
In view of Proposition~\ref{prop:Weyl-Inv}, we only need to prove the
surjectivity assertion. We can identify $MGL^{*, *}_G(X;R)$ and 
$mgl^{*,*}_G(X;R)$ using Corollary~\ref{cor:MLG**}.

It follows from Theorem~\ref{thm:WeakBasic} and 
\cite[Proposition~6.7]{Krishna4} that $MGL^*_T(k; R) \cong
\bL_R[[t_1, \cdots , t_n]]$. Using Lemma~\ref{lem:trivial-T} and
Theorem~\ref{thm:Main-Str}, we see that $MGL^*_T(X; R) \cong
\left(\bL_R[[t_1, \cdots , t_n]]\right)^r$, where $r$ is the number of $T$-fixed
points on $X$. In particular, the map
$MGL^*_T(X; R) \to MGL^*_T(X; \Q)$ is injective.

We now consider the commutative diagram 
\begin{equation}\label{eqn:Weyl-Inv*0}
\xymatrix@C1.8pc{
\left(MGL^*_T(X;R)\right)^W \ar[r]^>>>>{(\pi_X)_*} \ar[d]_{f} &
MGL^*_G(X; R) \ar[r]^<<<<{r^G_{T,X}} \ar[d] & 
\left(MGL^*_T(X;R)\right)^W \ar[d]^{f} \\
\left(MGL^*_T(X; \Q)\right)^W \ar[r]_>>>>{(\pi_X)_*} &
MGL^*_G(X; \Q) \ar[r]_<<<<{r^G_{T,X}} & 
\left(MGL^*_T(X; \Q)\right)^W.}
\end{equation}

Let $a \in MGL^*_T(X;R)$ and $\alpha = (\pi_X)_*(a) \in  MGL^*_G(X;R)$
be as in the proof of Lemma~\ref{lem:ML-Gen}. Recall that
$\alpha \in MGL^*_G(X;R)$ is invertible. 
For any $x \in MGL^*_G(X;R)$ and $y = r^G_{T,X}(x)$, it follows from 
Proposition~\ref{prop:Push-Pull}
that 
\[
\begin{array}{lll}
r^G_{T,X} \circ (\pi_X)_* (a \cdot y) & = & 
r^G_{T,X} \circ (\pi_X)_* \left(a \cdot r^G_{T,X}(x)\right) \\
& = & r^G_{T,X} \left(\alpha \cdot x\right) \\
& = & r^G_{T,X}(\alpha) \cdot y.
\end{array}
\]
In particular, it follows from  Lemma~\ref{lem:Push-Rat} that
for any $y \in \left(MGL^*_T(X;R)\right)^W$, one has
\[
r^G_{T,X}(\alpha) \cdot f(y) = r^G_{T,X} \circ (\pi_X)_* 
\left(a \cdot f(y)\right) = 
f\left(r^G_{T,X} \circ (\pi_X)_* (a \cdot y)\right).
\]
Equivalently, we get $f(y) = f\left(r^G_{T,X}(\alpha^{-1})\cdot 
\left(r^G_{T,X} \circ (\pi_X)_* (a \cdot y)\right)\right)$.
Since we have shown above that $f$ is injective, we 
get 
\[
y = r^G_{T,X}(\alpha^{-1})\cdot 
\left(r^G_{T,X} \circ (\pi_X)_* (a \cdot y)\right)
= r^G_{T,X}\left(\alpha^{-1}\cdot \left((\pi_X)_* (a \cdot y)\right)\right).
\]
This proves the required surjectivity.
\end{proof}

As an immediate consequence of  Theorem~\ref{thm:Weyl-Inv*}, we obtain
the following generalization of Totaro's theorem
\cite[Theorem~1.3]{Totaro2} to the case of motivic cobordism
of the classifying spaces of reductive groups.

\begin{cor}\label{cor:Weyl-Inv*Flag}
Let $G$ be a connected reductive group with a split maximal torus $T$ and
let $B$ be a Borel subgroup containing $T$. Then
\begin{enumerate}
\item
$MGL^*(BG;R) \cong \left(MGL^*(BT;R)\right)^W$ and
\item
$MGL^*(BT; R) \cong \left(MGL^*_T(G/B;R)\right)^W$.
\end{enumerate}
\end{cor}
\begin{proof}
The first part follows straightaway from Theorem~\ref{thm:Weyl-Inv*}
by taking $X = \Spec(k)$. The second part follows by applying 
Theorem~\ref{thm:Weyl-Inv*} to $X = G/B$ and then using the
identification $MGL^*_G(G/B) \cong MGL^*_T(k) = MGL^*(BT)$
by Theorem~\ref{thm:BPEC*} and Proposition~\ref{prop:Borel-torus}.
\end{proof}

As another application of the above results, we get the following
computation of the $T$-equivariant motivic cobordism of the flag variety
$G/B$.

\begin{cor}\label{cor:ECFlag}
Let $G$ be a connected and split reductive group over $k$ with a split
maximal torus $T$. Let $B$ be a Borel subgroup of $G$ containing $T$.
Then the natural map
\[
\Psi_{G/B}: S(T;R) {\underset{S(G;R)}\otimes} S(T;R) \to MGL^*_T(G/B; R)
\]
\[
\Psi(a \otimes b) = a \cdot r^G_{T, G/B}(b);
\]
is an isomorphism of $S(T;R)$-algebras.
\end{cor}
\begin{proof}
The above map is defined using the identification $\mg(G/B) = \mt(k)$ \\
$ = S(T)$. Now, the desired isomorphism follows from \cite[Theorem~4.6]{KK},
Theorem~\ref{thm:WeakBasic} (11) and Corollary~\ref{cor:MLG**}.
\end{proof}

\section{Realizations of equivariant motivic cobordism}
\label{section:Realization}
We shall assume in this section that the ground field $k$ is a subfield of
the field of complex numbers and we fix an embedding $\sigma: k \inj \C$.
Recall that there is a functor ${\rm HoSets}_{\bullet} \to \sH_{\bullet}(k)$,
where ${\rm HoSets}_{\bullet}$ denote the unstable homotopy category of
simplicial sets which is equivalent to the homotopy category of pointed
topological spaces ${\bf Top}_{\bullet}$ via the geometric realization 
functor. Let $\sS \sH$ denote the stable homotopy category 
the pointed topological spaces.

\subsubsection{Topological realization functor}
\label{subsubsection:TReal}
There is a topological realization functor $\Sm_k \to {\bf Top}$ which takes
a scheme $X$ over $k$ to the space $X^{\rm an} = X(\C)$ of the complex valued 
points on $X$ via the embedding $\sigma$. Then $X^{\rm an}$ is a complex 
manifold. Every motivic space $Y \in Spc$ can be written as
$\left({\underset{X \times \Delta^n \to Y}{colim}} \ X \times \Delta^n\right) 
\xrightarrow{\cong} Y$.
This gives a topological realization functor 
\begin{equation}\label{eqn:TR}
{\bf R}_{\C}: Spc_{\bullet} \to {\bf Top}_{\bullet} ;
\end{equation}
\[
{\bf R}_{\C}(Y) = {\underset{X \times \Delta^n \to Y}{colim}} 
\left(X^{\rm an} \times |\Delta^n|\right).
\]
It is clear from this that for any ind-scheme $X$, ${\bf R}_{\C}(X)$
is complex manifold, which may be infinite-dimensional. We shall write
${\bf R}_{\C}(X)$ often as $X^{\rm an}$ if $X$ is an ind-scheme.

Let ${\rm Sp}\left({\bf Top}, \C\P^1\right)$ denote the category of 
$\C\P^1$-spectra in the category ${\bf Top}_{\bullet}$. Let 
$\left({\sS \sH}, \C\P^1\right)$ denote the stable homotopy category
of pointed $\C\P^1$-spectra under the stable equivalence.
It is known, as can be found in \cite[\S A.7]{PPR}, that the above topological
realization functor descends to an exact functor
${\bf R}_{\C} : {\sS \sH}(k) \to \left({\sS \sH}, \C\P^1\right)$.
There is a suspension functor $\Sigma_{top}: \left({\sS \sH}, \C\P^1\right) \to
\left({\sS \sH}, S^1\right)$ from the category of $\C\P^1$-spectra to the
category of $S^1$-spectra which is given by
\[
\left(\Sigma_{top}M_{\bullet}\right)_{2n} = M_n \ \ {\rm and} \ \ 
\left(\Sigma_{top}M_{\bullet}\right)_{2n+1} = S^1 \wedge M_n = \Sigma M_n.
\]
This functor induces an equivalence of the stable homotopy categories.
This stable homotopy category is denoted by $\sS \sH$. 
In other words, there is an exact functor ${\bf R}_{\C} : {\sS \sH}(k) \to
\sS \sH$.

\noindent
{\bf{Notation :}}  Let $X \in \Sm_k$ and let $E \in \sS \sH$. For the rest
of  this section, we shall denote the generalized cohomology $E^*(X^{\rm an})$
in short by $E^*(X)$.

\subsubsection{Complex cobordism}\label{subsubsection:ComC}
Recall that the complex cobordism $MU$ is obtained by applying the suspension 
functor $\Sigma_{top}$ to the $\C\P^1$-spectrum
$\left(MU_0, MU_1, \cdots , \right)$, where $MU_n = Th(E_n)$, where
$E_n$ is the universal rank $n$-bundle on the classifying space $BU_n$.
The rank $(n+1)$-bundle $\sO_{BU_n} \oplus E_n$ defines a unique map
$BU_n \xrightarrow{i_n} BU_{n+1}$ such that one has a commutative diagram
like ~\eqref{eqn:bounding0}. This gives the bounding maps
$\C\P^1 \wedge Th(E_n) \to Th(E_{n+1})$ of the spectrum $MU$.
The weak equivalence of topological spaces ${E_n}/{(E_n \setminus \{0\})} \to
Th(E_n)$ and the homotopy equivalence $BU_n \cong BGL_n$ shows that
${\bf R}_{\C}(MGL) = MU$, where $MGL$ is the motivic Thom spectrum
({\sl cf.} ~\eqref{eqn:Thom-S}). In fact, this is an isomorphism of ring
spectra. 
In particular, for any ind-scheme $X$, there is a natural map
\[
\Hom_{{\sS \sH}(k)}\left(\Sigma^{\infty}_T X_{+}, \Sigma^{a,b}MGL\right)
\to \Hom_{{\sS \sH}}\left(\Sigma^{\infty}_{\C\P^1} X^{\rm an}_{+}, 
{\bf R}_{\C}\Sigma^{a,b}MU\right)
\]
for any $a \ge b \ge 0$.
Using the canonical isomorphism ${\bf R}_{\C}(\G_m) \cong S^1$, we see 
that there is a natural homomorphism
\begin{equation}\label{eqn:Alg-Top}
t_X: MGL^{a,b}(X) \to MU^{a}(X^{\rm an}).
\end{equation}

Recall that if $G$ is a complex Lie group acting on a finite-dimensional 
$CW$-complex $X$, the equivariant complex cobordism of $X$ is defined as
\begin{equation}\label{eqn:Equiv-CC}
MU^*_G(X) : = MU^*\left(X \stackrel{G}{\times} EG\right)
\end{equation}
where $EG \to BG$ is the universal principal $G$-bundle over the classifying
space $BG$. It is known that $MU^*_G(X)$ does not depend on the choice
of the universal principal bundle $EG \to BG$.  

Let $G$ be a linear algebraic group over $k$ and let $X \in \Sm^G_k$.
Then $G^{\rm an}$ is a complex Lie group and it follows from Lemma~\ref{lem:elem}
and ~\eqref{eqn:Alg-Top} that there is a natural homomorphism
\begin{equation}\label{eqn:Alg-Top-G}
t^G_X: MGL^{a,b}_G(X) \to MU^a_G(X).
\end{equation}
In particular, we get a natural ring homomorphism 
$t^G_X: MGL^*_G(X) \to MU^*_G(X)$ which takes an element $x \in MGL^{2a,a}_G(X)$
to an element $t^G_X(x) \in MU^{2a}_G(X)$.

\subsection{Equivariant motivic cohomology and the
Cycle class maps}\label{subsection:CCM}
For any abelian group $A$, let ${\bf H}_A$ denote the motivic
Eilenberg-MacLane $T$-spectrum in ${\sS \sH}(k)$ as defined by 
Voevodsky \cite{Voev1}. For a linear algebraic group $G$ over $k$ and
$X \in \Sm^G_k$, one defines the {\sl equivariant motivic cohomology}
of $X$ by
\begin{equation}\label{eqn:MCoh}
H^{a,b}_G(X;A) :=  
\Hom_{{\sS \sH}(k)}\left(\Sigma^{\infty}_TX_G(\rho)_{+}, \Sigma^{a,b}{\bf H}_A\right)
\end{equation}
where $\rho  = \left(V_i, U_i\right)$ is an admissible gadget for $G$.
One also defines the analogue of $mgl^{*,*}_G(X)$ as
\[
h^{a,b}_G(X;A) := \ {\underset{i}\varprojlim} \
H^{a,b}\left(X \stackrel{G}{\times} U_i; A\right).
\]

It is shown in \cite{Krishna4} that $h^{*,*}_G(-)$ is well-defined
and it is an example of an oriented cohomology theory on $\Sm^G_k$.
Moreover, it follows from the proof of Theorem~\ref{thm:ECRat} that the 
map $H^{a,b}_G(X;A) \to h^{a,b}_G(X;A)$ is in fact an isomorphism.
In particular, the analogue of Theorem~\ref{thm:WeakBasic} holds verbatim
for the equivariant motivic cohomology $H^{*,*}_G(-)$.

It was shown by Voevodsky in \cite[Proposition~3.8]{Voev2} that
${\bf R}_{\C}({\bf H}_A)$ is isomorphic to the topological 
Eilenberg-MacLane spectrum in ${\sS \sH}$.
In particular, one obtains a commutative diagram

\begin{equation}\label{eqn:ATCM}
\xymatrix@C1.8pc{
MGL^{a,b}_G(X;A) \ar[r]^{t^G_X} \ar[d] & MU^a_G(X;A) \ar[d] \\
H^{a,b}_G(X;A) \ar[r]_{c^G_X} & H^a_G(X;A)}
\end{equation}   
of the equivariant cohomology theories on $\Sm^G_k$. The horizontal maps
are called the cycle class maps.

\subsubsection{Totaro's refined cycle class map}
\label{subsubsection:RCmap}
In order to study the cycle class maps and the natural maps between the
equivariant versions of cobordism and the ordinary cohomology, we need to
recall the following notation for the tensor product
while dealing with projective systems of modules over a commutative ring.
Let $A$ be a commutative ring and let $\{L_i\}$ and $\{M_i\}$ be two projective
systems of $A$-modules. Following \cite{Totaro1}, one defines  
the {\sl topological} tensor product of $L$ and $M$ by 
\begin{equation}\label{eqn:TTP}
L {\wh{\otimes}}_A M \ : = \ {\underset{i}\varprojlim} \ 
\left(L_i \otimes_A M_i\right).
\end{equation}
Given a linear algebraic group $G$, $X \in \Sm^G_k$ and an $\bL$-module $A$,
the proof of Lemma~\ref{lem:Ind} shows that the tensor product
\[
mgl^{*,*}_G(X) {\wh{\otimes}}_{\bL} A =  \ {\underset{i}\varprojlim} \ 
\left(MGL^{*,*}(X \stackrel{G}{\times} U_i) \otimes_{\bL} A\right)
\]
is independent of the choice of an admissible gadget 
$\rho = \left(V_i, U_i\right)$ for $G$.
In case of the natural map $\mg(X) \to mgl^{*,*}_G(X)$ being an isomorphism, we
shall also use the notation $\mg(X){\wh{\otimes}}_{\bL} A$ for 
$mgl^{*,*}_G(X) {\wh{\otimes}}_{\bL} A$.

We mentioned earlier that the present definition of the Chow groups of the
classifying space of a linear algebraic group $G$ over $k$ was invented by
Totaro \cite{Totaro1}. He also showed that the cycle class map 
$c^G_k: \CH^*(BG) \to H^*(BG)$ in fact factors through a refined cycles class map
$\wt{c}^G_k: \CH^*(BG) \to MU^*(BG){\wh{\otimes}}_{\bL} \Z$. It is known that
$MU^*(X){\wh{\otimes}}_{\bL} \Z$ naturally maps to $H^*(X)$ and is a more 
refined topological invariant of a topological space $X$ than its singular 
cohomology $H^*(X)$.

Recall that that $R$ denotes the ring $\Z[t^{-1}_G]$, where $t_G$ is the 
torsion-index of $G$.
As a consequence of Corollary~\ref{cor:MLG**}, Theorem~\ref{thm:WeakBasic} and 
\cite[Proposition~7.2]{Krishna1}, we obtain the following generalization of
\cite[Theorem~2.1]{Totaro1}.

\begin{thm}\label{thm:RCCM}
For a linear algebraic group $G$ over $k$ and $X \in \Sm^G_k$ projective,
there is a natural refined cycle class map 
\[
\wt{c}^G_X: \CH^*_G(X;R) \to MU^*_G(X;R) {\wh{\otimes}}_{\bL_R} R.
\]
such that its composite with the natural map 
$MU^*_G(X;R) {\wh{\otimes}}_{\bL_R} R \to H^*_G(X;R)$ is the usual 
cycle class map $c^G_X$.
\end{thm}

\subsection{Comparison of equivariant motivic and complex 
cobordism}\label{subsection:Comp-eq}
As a consequence of the results of Quillen \cite{Quillen},
Levine-Morel \cite{LM} and Levine \cite{Levine1}, one knows that the
topological realization map $MGL^*(k) \to MU^*(pt)$ is an isomorphism.
Based on this isomorphism, we now prove some comparison results for
the equivariant motivic and complex cobordisms of schemes with group actions.
As a consequence, we verify some conjectures of Totaro about the cycle
class maps.
The following result is the topological analogue of 
Theorem~\ref{thm:Weyl-Inv*}.
This result for the singular cohomology was earlier proven by
Holm and Sjamaar \cite[Proposition~2.1]{HS}.

\begin{thm}\label{thm:Top-ML}
Let $G$ be a connected reductive group over $k$ with a split maximal torus $T$
and the associated Weyl group $W$.
Let $X \in \Sm^G_k$ be projective such that $T$ acts on $X$ with only finitely
many fixed points. Then the map 
\[
r^{G, \rm top}_{T,X} : MU^*_G(X;R) \to \left(MU^*_T(X;R)\right)^W
\]
is an isomorphism.
\end{thm}
\begin{proof}
This is essentially proven by an easy topological translation of the 
proof of Theorem~\ref{thm:Weyl-Inv*}.
We give a sketch of the main steps. 
The topological version of ~\eqref{eqn:CCS1} gives rise to the 
commutative diagram

\begin{equation}\label{eqn:BBG}
\xymatrix@C.7pc{
MU^*(B_G) \ar[r]^{\pi^*} \ar[d]_{p^*_{G,X}} & MU^*(B_T) \ar[d]^{p^*_{T,X}} 
\ar[r]^{\iota^*} & MU^*(G/B) \ar@{=}[d] \\
MU^*_G(X) \ar[r]_{\pi^*_X} & MU^*_T(X) \ar[r]_{\iota^*_X} & MU^*(G/B).}
\end{equation}

It was shown in \cite[Lemma~4.2]{KK} that there are elements
$\{\rho_{w,X} : w \in W\}$ in $MU^*_T(X)$ 
such that $\{\iota^*_X(\rho_{w,X}) : w \in W\}$
forms an $\bL$-basis of $MU^*(G/B)$. Moreover, we can choose
$\rho_{w_0,X} = 1$, where $w_0 \in W$ is the longest length element.
It follows from \cite[Lemma~4.3]{KK} that the map

\begin{equation}\label{eqn:BBG0}
\Psi^{\rm top}_X: MU^*(BT) {\underset{MU^*(BG)}\otimes} MU^*_G(X) \to
MU^*_T(X);
\end{equation}
\[
\Psi^{\rm top}_X (x \otimes y) = p^*_{T,X}(x) \cdot \pi^*_X(y)
\]
is $W$-equivariant and an isomorphism of $MU^*(BT)$-modules.
In particular, we get
\[
\Psi^{\rm top}_X(1 \otimes y) = \Psi^{\rm top}_X \left(\iota^*_X(\rho_{w_0,X})
\otimes y \right) = \pi^*_X(y).
\]
Hence, to show that $r^{G, \rm top}_{T,X}$ is injective, it suffices to show that 
the map
$MU^*_G(X) \xrightarrow{1 \otimes id} 
(MU^*(G/B) {\underset{\bL}\otimes} MU^*_G(X))^W$ is injective.
But to do this, we only have to observe from the projection formula for the map
$p_{G/B} : G/B \to {\rm pt}$ that
${p_{G/B}}_* \left(\rho \cdot p^*_{G/B} (x)\right) =
{p_{G/B}}_*(\rho) \cdot x = x$, where $\rho \in MU^*(G/B)$
is the class of a point.
This gives a right inverse of the map $p^*_{G/B}$ and hence a right inverse of
$1 \otimes id=p^*_{G/B} \otimes id$.

To prove the surjectivity of the map $r^{G, \rm top}_{T,X}$, we first note from
our assumption and the topological analogue of Theorem~\ref{thm:Main-Str}
that $MU^*_T(X) \cong \left(\bL[[t_1, \cdots , t_n]]\right)^r$, where
$n = \rank(T)$ and $r$ is the number of $T$-fixed points on $X$. 
In particular, the map $MU^*_T(X;R) \to MU^*_T(X;\Q)$ is injective.

Now, the surjectivity argument of Theorem~\ref{thm:Weyl-Inv*} goes through
verbatim, where one has to replace $a \in MGL^*_T(X;R)$ with $t^G_X(a) \in 
MU^*_T(X;R)$ and $\alpha \in MGL^*_G(X;R)$ with $t^G_X(\alpha) \in MU^*_G(X;R)$
and then use the fact that the surjectivity result holds over the 
rationals by \cite[Theorem~8.8]{Krishna1}.
\end{proof}

\subsubsection{Totaro's conjectures}\label{subsubsection:TConj}
It was conjectured ({\sl cf.} \cite[Introduction]{Totaro1}) that 
the refined cycle class $\CH^*(BG) \to MU^*(BG) {\wh{\otimes}}_{\bL} \Z$
should be an isomorphism. Totaro modified this conjecture to an expectation
that this map should be an isomorphism after localization at certain prime
$p$. We shall show below that the refined cycle class map is in fact an
isomorphism after inverting the torsion index of the group $G$.
We first have the following stronger result.

\begin{thm}\label{thm:Alg-Top-Com}
Let $G$ be a connected reductive group over $k$ with a split maximal torus $T$. 
Let $X \in \Sm^G_k$ be projective such that $T$ acts on $X$ with only finitely
many fixed points. Then the maps  
\[
t^G_X: MGL^*_G(X;R) \to MU^*_G(X;R)  \ \  {\rm and}
\]
\[
c^G_X: \CH^*_G(X;R) \to H^*_G(X;R)
\]
are isomorphisms. In particular, $MU^*_G(X;R)$ and $H^*_G(X;R)$
have no element in odd degrees.
\end{thm}
\begin{proof}
It follows from Theorems~\ref{thm:BBH}, ~\ref{thm:WeakBasic}, 
Corollary~\ref{cor:MLTorus**} and \cite[Thorem~3.7]{KK} that the map
$MGL^*_T(X) \xrightarrow{t^T_X} MU^*_T(X)$ is an isomorphism. 

A much simpler argument shows that the map $\CH^*_T(X) \xrightarrow{c^T_X}
H^*_T(X)$ also is an isomorphism. 
To see this quickly, take a canonical admissible gadget
$\left(V_i, U_i\right)$ for $T$ and observe that 
$X_i = X \stackrel{T}{\times} U_i$ is then a smooth cellular scheme
and hence the map $\CH^*(X_i) \to H^*(X_i)$ is an isomorphism.
It follows that the map $c^T_X$ is an isomorphism.

The isomorphism of $t^G_X$ follows immediately from the isomorphism of
$t^T_X$, combined with Theorems~\ref{thm:Weyl-Inv*} and ~\ref{thm:Top-ML}.
The isomorphism of $c^G_X$ follows from the isomorphism of $c^T_X$,
combined with \cite[Corollary~5.9]{Krishna4}, \cite[Lemma~3.6]{KK} and 
\cite[Proposition~2.1]{HS}.
\end{proof}

\begin{thm}\label{thm:ATComp2}
Let $G$ be a connected reductive group over $k$ with a split maximal torus $T$. 
Let $X \in \Sm^G_k$ be projective such that $T$ acts on $X$ with only finitely
many fixed points. Then the maps
\begin{equation}\label{eqn:ATComp2*4}
\CH^*_G(X;R) \xrightarrow{\wt{c}^G_X}
MU^*_G(X;R) {\wh{\otimes}}_{\bL_R} R \to H^*_G(X;R)
\end{equation}
are isomorphisms.
\end{thm}
\begin{proof}
It follows from Theorem~\ref{thm:Alg-Top-Com} that the composite map
$\CH^*_G(X;R) \xrightarrow{c^G_X} H^*_G(X;R)$ is an isomorphism.
Thus, we only need to show that the refined cycle class map 
$\wt{c}^G_X$ ({\sl cf.} Theorem~\ref{thm:RCCM}) is surjective.

Let $\rho = \left(V_i, U_i\right)$ be an admissible gadget for $G$. 
Let us denote the mixed space $X \stackrel{G}{\times} U_i$ in short by $X_i$.
It follows from our assumption, Theorem~\ref{thm:BBH} and 
\cite[Lemma~3.6]{KK} that $H^*_T(X)$ is torsion-free. It follows from
\cite[Proposition~2.1]{HS} that $H^*_G(X;R)$ is torsion-free.
It follows subsequently using the Atiyah-Hirzebruch spectral sequence
({\sl cf.} \cite[Corollary~2]{Landweber}, \cite[Lemma~2.2]{Totaro1})
that the map $MU^*_G(X;R) \to {\underset{i}\varprojlim} \ MU^*(X_i;R)$ is an
isomorphism and moreover, for each positive integer $i$, there is $j \ge i$
such that 
\begin{equation}\label{eqn:ATComp2*0} 
{\rm Image}\left(MU^*_G(X;R) \to MU^*(X_i;R)\right) =
{\rm Image}\left(MU^*(X_j;R) \to MU^*(X_i;R)\right).
\end{equation}

We conclude from Corollary~\ref{cor:MLG**} and  Theorem~\ref{thm:Alg-Top-Com}
that there is a commutative diagram of $\bL_R$-modules
\begin{equation}\label{eqn:ATComp2*1}
\xymatrix@C1.8pc{
MGL^*_G(X;R) \ar[r]^{t^G_X} \ar[d] &  MU^*_G(X;R) \ar[d] \\
{\underset{i}\varprojlim} \ MGL^*(X_i;R) \ar[r] &
{\underset{i}\varprojlim} \ MU^*(X_i;R)}
\end{equation}
in which all arrows are isomorphisms.

We now show the surjectivity of $\wt{c}^G_X$.
Using the isomorphisms in ~\eqref{eqn:ATComp2*1} and 
\cite[Proposition~7.2]{Krishna1}, we need to show that the map 
\begin{equation}\label{eqn:ATComp2*2}
{\underset{i}\varprojlim} \ \frac{MGL^*(X_i;R)}{\bL_R^{<0}MGL^*(X_i;R)} \ \to \
{\underset{i}\varprojlim} \ \frac{MU^*(X_i;R)}{\bL_R^{<0}MU^*(X_i;R)} 
\end{equation}
is surjective. Since the bottom horizontal arrow in ~\eqref{eqn:ATComp2*1}
is an isomorphism, it suffices to show that the map
\begin{equation}\label{eqn:ATComp2*3}
{\underset{i}\varprojlim} \ MU^*(X_i;R) \to 
{\underset{i}\varprojlim} \ \frac{MU^*(X_i;R)}{\bL_R^{<0}MU^*(X_i;R)} 
\end{equation} is surjective. To show this, it suffices to show that
${\underset{i}{\varprojlim}^1} \ \bL_R^{<0}MU^*(X_i;R) = 0$. Using the
surjectivity ${\underset{i}{\varprojlim}^1} \ 
\bL_R^{<0} \otimes_{\bL_R} MU^*(X_i;R) 
\surj {\underset{i}{\varprojlim}^1} \ \bL_R^{<0}MU^*(X_i;R)$, it is enough to show
that ${\underset{i}{\varprojlim}^1} \ \bL_R^{<0} \otimes_{\bL_R} MU^*(X_i;R) =0$.

To prove this last assertion, it suffices to show that the projective system
$\{\bL_R^{<0} \otimes_{\bL_R} MU^*(X_i;R)\}$ satisfies the Mittag-Leffler 
condition.
On the other hand, it follows from ~\eqref{eqn:ATComp2*0} that the
projective system $\{MU^*(X_i;R)\}$ satisfies the Mittag-Leffler condition.
From this, it follows immediately that the same holds for
$\{\bL_R^{<0} \otimes_{\bL_R} MU^*(X_i;R)\}$. We have thus proven the surjectivity
of $\wt{c}^G_X$. This completes the proof of the theorem.
\end{proof}

\begin{cor}\label{cor:Alg-Top-Com1}
Let $G$ be a connected split reductive group over $k$. Then the maps
$MGL^*(BG;R) \to MU^*(BG;R)$ and $\CH^*(BG;R) \to 
MU^*(BG;R) {\wh{\otimes}}_{\bL_R} R$ are isomorphisms.
\end{cor}

\section{Motivic cobordism of quotient stack}\label{section:Qst}
In this section, we show how one can use equivariant motivic cobordism
to define the motivic cobordism for quotient stacks. The motivic cohomology
of such stacks was earlier defined by Edidin-Graham \cite{EG}.
Our definition is based on the following result.

\begin{prop}\label{prop:Ind-ad-gadget}
Let $G$ and $H$ be two linear algebraic groups acting on two smooth
schemes $X$ and $Y$ respectively such that $[X/G] \cong [Y/H]$ as
stacks. There is then a canonical isomorphism $X_G \cong Y_H$ of
motivic spaces in $\sH(k)$.
\end{prop}
\begin{proof}
Let $\sX$ denote the stack $[X/G] \cong [Y/H]$.
Let $\rho = \left(V_i, U_i\right)_{i \ge 1}$ and 
$\rho' = \left(V'_i, U'_i\right)_{i \ge 1}$ be admissible gadgets for 
$G$ and $H$ respectively. We set $\ov{X}^i_G(\rho) = [(X \times V_i)/G]$
and $\ov{Y}^j_H(\rho') = [(Y \times V'_j)/H]$.

This yields representable morphisms of stacks
\begin{equation}\label{eqn:Ind-ad0}
X^i_G(\rho) \inj \ov{X}^i_G(\rho) \xrightarrow{f_i} [X/G] \cong \sX
\ \ {\rm and} \ Y^j_H(\rho') \inj 
\ov{Y}^j_H(\rho') \xrightarrow{g_j} [Y/H] \cong \sX.
\end{equation}

For each $i, j \ge 1$, we consider the fiber product diagrams of stacks

\[
\xymatrix@C2pc{
\sU_{i,j} \ar[r]^{p_{i,j}} \ar[d]_{q_{i,j}} & X^i_G(\rho) \ar[d]^{f_i} & 
\sV_{i,j} \ar[r]^{\ov{p}_{i,j}} \ar[d]_{\ov{q}_{i,j}} & X^i_G(\rho) \ar[d]^{f_i} \\
Y^j_H(\rho') \ar[r]_{g_j} & \sX & \ov{Y}^j_H(\rho') \ar[r]_{g_j} & \sX.}
\]

Since each $g_j$ in the diagram on the right is a vector bundle
map of stacks with fiber $V'_j$, we see that
each $\ov{p}_{i,j}$ is a vector bundle with fiber $V'_j$. In particular, each
$\sV_{i,j}$ is a smooth scheme and 
$\sU_{i,j} \subsetneq \sV_{i,j}$ is an open subscheme.

For a fixed $i \ge 1$, we get a sequence 
$\left(\sV_{i,j}, \sU_{i,j}, f_{i,j}\right)_{j \ge 1}$ of pairs of smooth schemes
where $\sV_{i,j} \to X^i_G(\rho)$ is s vector bundle, $\sU_{i,j} \subsetneq
\sV_{i,j}$ is an open subscheme and $f_{i,j} : \left(\sV_{i,j}, \sU_{i,j}\right)
\to \left(\sV_{i,j+1}, \sU_{i, j+1}\right)$ is the natural map of pairs of
smooth schemes over $X^i_G(\rho)$. It follows moreover from the property
of $\rho'$ being an admissible gadget for $H$ that 
$\left(\sV_{i,j}, \sU_{i,j}, f_{i,j}\right)_{j \ge 1}$ is an admissible
gadget over $X^i_G(\rho)$ in the sense of \cite[Definition~4.2.1]{MV}.

Setting $\sU_i = colim_j \ \sU_{i,j}$ and $p_i = colim_j \ p_{i,j}$, we
conclude from \cite[Proposition~4.2.3]{MV} that $\sU_i \xrightarrow{p_i}
X^i_G(\rho)$ is an $\A^1$-weak equivalence. Taking the colimit of these
maps as $i \to \infty$, we conclude that $\sU \xrightarrow{p} X_G(\rho)$ is 
an $\A^1$-weak equivalence, where $\sU = colim_{i,j} \ \sU_{i,j}$.
By reversing the roles of $\rho$ and $\rho'$, we see that the map
$\sU \xrightarrow{q} Y_H(\rho')$ is also an $\A^1$-weak equivalence.
This concludes the proof.
\end{proof}

\begin{defn}\label{defn:Mot-cob-qt-st}
Let $\sX$ be a smooth stack of finite type over $k$ which is isomorphic to
a stack of the form $[X/G]$ where $G$ is a linear algebraic group acting on 
a smooth scheme $X$ over $k$. 
We define the {\sl motivic cobordism} of $\sX$ as 
\[
MGL^{a,b}(\sX) : = MGL^{a,b}_G(X).
\]
\end{defn}
It follows from Proposition~\ref{prop:Ind-ad-gadget} that
$MGL^{a,b}(\sX)$ is well defined. We let 
$MGL^{*,*}(\sX)$ to be the sum ${\underset{a,b} \bigoplus} \ MGL^{a,b}(X)$.  
It follows from Theorem~\ref{thm:BPEC*} that the association
$\sX \mapsto MGL^{*,*}(\sX)$ is a contravariant functor from the
category of smooth quotient stacks into the the category of bigraded
commutative rings. Furthermore, this functor satisfies homotopy invariance,
localization, theory of Chern classes and projective bundle formula.

\end{document}